\documentclass[a4,12pt,usenames]{article}
\usepackage{amssymb}
\usepackage{mathrsfs} 
\usepackage[title]{appendix}
\usepackage{amsmath}
\usepackage{amsthm}
\usepackage{graphicx}
\usepackage[all]{xy}
\usepackage{authblk}
\usepackage{accents}
\usepackage{wasysym}

\DeclareFontFamily{OT1}{pzc}{}
\DeclareFontShape{OT1}{pzc}{m}{it}{<-> s * [1.150] pzcmi7t}{}
\DeclareMathAlphabet{\mathpzc}{OT1}{pzc}{m}{it}

\makeatletter
\let\mcnewpage=\newpage
\newcommand{\Trick}{
\renewcommand\newpage{%
        \if@firstcolumn
            \hrule width\linewidth height0pt
                \columnbreak
        \else
                \mcnewpage
        \fi
}
}
\makeatother

\theoremstyle{plain}

\newtheorem{theorem}{Theorem}[section]
\newtheorem{proposition}[theorem]{Proposition}
\newtheorem{corollary}[theorem]{Corollary}
\newtheorem{lemma}[theorem]{Lemma}

\newtheorem{maintheorem}{Theorem}

\newenvironment{mainthm}[1]
{\renewcommand{\themaintheorem}{#1}
\begin{maintheorem}}
{\end{maintheorem}}

\theoremstyle{definition}
\newtheorem{definition}[theorem]{Definition}
\newtheorem{example}[theorem]{Example}
\newtheorem{remark}[theorem]{Remark}

\newtheorem{fundamental}[theorem]{Construction}

\newtheorem*{notation}{Notation}

\newtheorem*{terminology}{Terminology}

\newcommand{\thistheoremname}{}
\newtheorem*{genericthm*}{\thistheoremname}
\newenvironment{namedthm*}[1]
  {\renewcommand{\thistheoremname}{#1}%
   \begin{genericthm*}}
  {\end{genericthm*}}

\makeatletter

\newenvironment{figurehere}
  {\def\@captype{figure}}
  {}
\makeatother

\def\SU{\mathrm{SU}}

\def\SO{\mathrm{SO}}

\def\Strip{\mathcal{S}}
\def\End{\mathcal{E}}

\def\tr{\mathrm{tr}}

\def\ZZ{\mathbb{Z}}

\def\RR{\mathbb{R}}
\def\CC{\mathbb{C}}

\def\PP{\mathbb{P}}

\def\Sph{\mathbb{S}^2}

\def\MT{\mathcal{MT}}
\def\MP{\mathcal{MP}}

\def\inj{\mathrm{inj}}

\def\Mcal{\mathcal{M}}% \def\Mcal{\text{\footnotesize{$\mathpzc{M}$}}}
\def\proj{\mathrm{proj}}
\def\B{\mathrm{Bal}}

\def\Sy{\mathrm{S}}
\def\ASy{\mathrm{A}}
\def\BEL{\psi_{Bel}}

\def\bm#1{\text{\boldmath$#1$}}

\def\vc#1{{\bm{#1}}}   % vectors

\def\pa{\partial} % boundary/derivations
\def\1{\vc{1}}
\def\0{\vc{0}}

\def\th{\vartheta}
\def\ceil#1{\lceil{#1}\rceil}
\def\floor#1{\lfloor{#1}\rfloor}

\def\sys{\mathrm{sys}}

\def\area{\mathrm{Area}}

\def\MSPH{\mathcal{MS}}
\def\rar{\rightarrow}
\def\lra{\longrightarrow}

\def\Hur{\mathcal{H}}
\def\Vor{\mathcal{V}}

\def\twosimpl{\mathring{\Delta}^2}

\def\Modd{\mathcal{M}}
\def\Dodd{\mathcal{D}}

\def\MSc{\MSPH_{g,n}^{{\scriptscriptstyle{\ge}} s}}
\def\lA{({\scriptstyle{\le}} A)}

\usepackage{parskip}
\begin{document}

\title{Moduli of spherical tori with one conical point}

\author[1]{Alexandre Eremenko}
\affil[1]{{\small{Purdue University, Department of Mathematics (eremenko@math.purdue.edu)}}}

\author[2]{Gabriele Mondello}
\affil[2]{{\small{``Sapienza'' Universit\`a di Roma, Department of Mathematics (mondello@mat.uniroma1.it)}}}

\author[3]{Dmitri Panov}
\affil[3]{{\small{King's College London, Department of Mathematics (dmitri.panov@kcl.ac.uk)}}}

\maketitle

\abstract{\noindent
In this paper we determine the topology of the moduli space $\MSPH_{1,1}(\th)$ of surfaces of genus one
with a Riemannian metric of constant curvature $1$ and one conical point of angle $2\pi\th$.
In particular, for $\th\in (2m-1,2m+1)$ non-odd, $\MSPH_{1,1}(\th)$ is connected, has orbifold Euler characteristic
$-m^2/12$, and its topology depends on the integer $m>0$ only. 
For $\th=2m+1$ odd, $\MSPH_{1,1}(2m+1)$ has $\ceil{m(m+1)/6}$ connected components.
For $\th=2m$ even, $\MSPH_{1,1}(2m)$ has a natural complex structure
and it is biholomorphic to $\mathbb{H}^2/G_m$ for a certain subgroup $G_m$ of $\mathrm{SL}(2,\ZZ)$
of index $m^2$, which is non-normal for $m>1$.
}

\setcounter{tocdepth}{2}
\tableofcontents

\section{Introduction and main results}

The subject of this paper is the moduli space of spherical tori
with one conical point. We recall that a spherical metric
on a surface $S$ with {\em conical points} at the points ${\bm x}=\{x_1,\ldots, x_n\}\in S$
is a Riemannian metric of curvature $1$ on $\dot{S}:=S\setminus {\bm x}$,
such that a neighbourhood of $x_j$ is isometric to a cone with
a conical angle $2\pi\th_j>0$.

%, see, for example, \cite{R}.
%{\bf{WHAT CITATION HERE ABOVE?}}\\

Let us immediately specify what we mean by the moduli space $\MSPH_{g,n}(\th)$ of spherical surfaces in this paper. As a set, $\MSPH_{g,n}(\th)$ parametrizes compact, connected, oriented surfaces of genus $g$ with a spherical metric that has conical angles $(2\pi\th_1,\ldots 2\pi\th_n)$ at marked points $x_1,\ldots, x_n$. Two surfaces correspond to the same point of the space if there is a marked isometry from one to the other. In order to define a topology on $\MSPH_{g,n}(\th)$, we consider the bi-Lipschitz distance between marked surfaces, as in \cite{Gro}. Such a distance defines a metric, and the corresponding topology on $\MSPH_{g,n}(\th)$ is called the {\it Lipschitz topology}; its properties are discussed in Section \ref{secLipschitz}.

As a spherical metric defines a conformal structure on the surface, we have the {\em forgetful map} $F:\MSPH_{g,n}(\theta)\to \Mcal_{g,n}$, where $\Mcal_{g,n}$ is the
moduli space of conformal structures on $(S,\bm{x})$.

Since a neighbourhood of a smooth point on $S$ is isometric to
an open set on the sphere equipped with the standard spherical metric, by an analytic continuation we obtain an orientation-preserving locally isometric {\it{developing map}}
$f:\dot{S}\to \mathbb S^2$. Strictly speaking, the
developing map is defined on the universal cover of
$\dot{S}$ but it is sometimes convenient to
think of it as a multivalued function on $\dot{S}$.

The developing map defines a representation of the fundamental group
of $\dot{S}$ to the group ${\rm SO}(3)$ of
rotations of the unit sphere $\mathbb S^2$. The image of this representation
is called the {\em monodromy group.}

The goal of this article is to provide an explicit description of the moduli space $\MSPH_{1,1}(\th)$ of spherical tori with one conical point. 

Spherical tori with one conical point were also studied in \cite{Lin1, Lin2, Lin4, Lin5, eisenstein, EG, E}.

\subsection{Main results}

Our main results consist
of Theorems \ref{mainonodd}-\ref{eventheta}
and they are stated in the following three subsections.
%We will split our results into two cases. 
%First, we consider tori with $\th$ which is not an odd integer.

\subsubsection{$\th$ not an odd integer}

\begin{mainthm}{A}[Topology of $\MSPH_{1,1}(\th)$ for $\th$ not odd]\label{mainonodd} 
Take $\th\in (1,\infty)$  that is not an odd integer and set $m=\floor{\frac{\th+1}{2}}$. The moduli space $\MSPH_{1,1}(\th)$ of spherical tori with a conical point of angle $2\pi\th$ is a connected orientable two-dimensional orbifold of finite type with the following properties.
\begin{itemize}
\item[(i)]
As a surface, $\MSPH_{1,1}(\th)$ has genus $\floor{\frac{m^2-6m+12}{12}}$
and $m$ punctures.
\item[(ii)]
The  moduli space $\MSPH_{1,1}(\th)$ has orbifold Euler characteristic $\chi(\MSPH_{1,1}(\th))=-\frac{m^2}{12}$.
Moreover, it has at most one orbifold point of order $4$
and at most one orbifold point of order $6$. All the other points are orbifold points of order $2$.
\item[(iii)]
The moduli space $\MSPH_{1,1}(\th)$ has one orbifold point of order $6$ if and only if $d_1(\th, 6\mathbb Z)>1$. 
\item[(iv)] 
The moduli space $\MSPH_{1,1}(\th)$ has one orbifold point of order $4$ if and only if $d_1(\th, 4\mathbb Z)>1$. 
\end{itemize}
\end{mainthm}

Note that for $\th=2m$ this theorem gives a positive answer to the question of Chai, Lin, and Wang \cite[Question 4.6.6, a]{Lin2}, whether $\MSPH_{1,1}(2m)$ is connected.

We refer to \cite{orbifold} for a general treatment of orbifolds.
In fact we adopt a slightly more general definition of orbifolds %(as in \cite{noneffective})
that includes the case in which all points can have orbifold order greater than $1$.
The definition of orbifold Euler characteristic is given
at page 29 of \cite{orbifold}. This is coherent with the definition used, for example,
in \cite{hz}. A few properties of the orbifold Euler characteristic are listed in Remark \ref{defchi}.%Section \ref{secmainodd}.

Note that in \cite{EGMP} we used a different convention and we endowed our moduli spaces
with an orbifold structure for which the order of each point is half the number of automorphisms
of the corresponding object. Thus, the orbifold Euler characteristics computed 
in \cite{EGMP} are twice the ones that would be obtained
following the convention of the present paper.

\begin{remark}[Orbifold structure and isometric involution]\label{rmk:orbi}
For $\th$ not odd, spherical metrics in $\MSPH_{1,1}(\th)$
are invariant under the unique conformal involution $\sigma$ of tori (see Proposition \ref{conformaliso}).
Thus every such spherical torus is a double cover of a spherical surface of genus $0$ with conical points of angles $(\pi\th,\pi,\pi,\pi)$, and so
the moduli space $\MSPH_{1,1}(\th)$ is homeomorphic to
$\MSPH_{0,4}\left(\frac{\th}{2},\frac{1}{2},\frac{1}{2},\frac{1}{2}\right)/\Sy_3$ as a topological space. On the other hand,
the orbifold order of a point in $\MSPH_{1,1}(\th)$ exactly corresponds
to the number of (orientation-preserving) self-isometries of the corresponding spherical torus.
This explains why every point of $\MSPH_{1,1}(\th)$ has even orbifold order, as stated
in Theorem \ref{mainonodd}. Thus $\MSPH_{1,1}(\th)$ is not isomorphic to the orbifold quotient $\MSPH_{0,4}\left(\frac{\th}{2},\frac{1}{2},\frac{1}{2},\frac{1}{2}\right)/\Sy_3$.
%
%but the orbifold Euler characteristic of $\MSPH_{1,1}(\th)$ is $\frac{1}{12}$
%the Euler characteristic of $\MSPH_{0,4}\left(\frac{\th}{2},\frac{1}{2},\frac{1}{2},\frac{1}{2}\right)$. 
\end{remark}

An important geometric input on which Theorem \ref{mainonodd} hinges is the notion of   {\it balanced spherical triangles} and Theorem \ref{twotrianglesTH}, describing the relation between spherical tori and balanced triangles. 

\begin{definition}[Spherical polygons]
A {\it spherical polygon}  $P$ with angles $\pi\cdot(\th_1,\ldots,\th_n)$ is a closed disk equipped with a Riemannian metric of constant curvature $1$, with $n$ distinguished boundary points $x_1,\ldots, x_n$ which are called \emph{vertices}, and such that the arcs between the adjacent vertices are geodesics forming an interior angle $\pi\th_i$ at the $i$-th vertex. Two polygons are {\it isometric} if there is an isometry between them that preserves the labelling. 
%Two oriented polygons are {\it{congruent}} if they are isometric through an orientation-preserving
%isometry.
%In case the isometry preserves orientations of the polygons we call them {\it congruent}.
\end{definition}

Spherical polygons with two or three vertices are called \emph{digons} or \emph{triangles}\footnote{We note that spherical triangles in the sense of our definition are called sometimes {\it{Schwarz--Klein triangles}}, to distinguish them
from triangles understood as broken geodesic lines on the sphere.
See, for instance, \cite{EGnew}.} correspondingly. 

\begin{definition}[Balanced triangles]
A spherical triangle $\Delta$ with angles $\pi\cdot(\th_1,\,\th_2,\,\th_3)$ is called {\it balanced} if the numbers $\th_1,\, \th_2,\, \th_3$ satisfy the three triangle inequalities. 
If the triangle inequalities are satisfied strictly, we call the triangle {\it strictly balanced}. If for some permutation $(i,j, k)$ of $(1,2,3)$  we have  $\th_i=\th_j+\th_k$ we call the triangle {\it semi-balanced}. If $\th_i>\th_j+\th_k$  for some $i$, we call the triangle {\it unbalanced}.
 \end{definition}
 
We mention that semi-balanced triangles are called {\it{marginal}} in \cite{EGMP} and \cite{EGnew}.

Whenever a spherical triangle is   realised as a subset of a surface we will induce on it the orientation of the surface. We will say that two oriented spherical surfaces (or polygons) are {\it conformally isometric} (or {\it congruent}) if there is an orientation preserving isometry from one surface (or polygon) to the other.

\begin{terminology}[Integral angles]
%In  this paper we  use the following convention. 
Throughout the paper, angles will be measured in radiants.
Nevertheless, an angle $2\pi\th$ at a conical point of a spherical surface is called {\it{integral}} if $\th\in\ZZ_{>0}$;
similarly, an angle $\pi\th$ at a vertex of a spherical polygon is called {\it{integral}} if $\th\in\ZZ_{>0}$.
\end{terminology}

Now we describe a construction that will be omnipresent in this paper. 

\begin{fundamental}\label{fund1}
To each spherical triangle $\Delta$ with vertices $x_1,x_2,x_3$ one can associate a spherical torus $T(\Delta)$ with one conical point by taking a conformally isometric triangle $\Delta'$ with vertices $x_1', x_2', x_3'$ and isometrically identifying each side $x_ix_j$ with the side $x_j'x_i'$
(in such a way that $x_i$ is identified to $x'_j$ and $x_j$ is identified to $x'_i$)
 for $i,j\in \{1,\,2,\,3\}$. The angle at the conical point of $T(\Delta)$, that corresponds to the vertices of the triangles, 
is twice the sum of the angles of $\Delta$.
If $\Delta$ is endowed with an orientation, then $T(\Delta)$ canonically inherits an orientation.
\end{fundamental}

To state the next result we need two more notions. Let $T$ be a spherical torus with one conical point. An isometric orientation-reversing involution on $T$ will be called a {\it rectangular involution} if its set of fixed points consists of two connected components. By a {\it geodesic loop $\gamma$} based at a conical point $x$ we  mean a loop based at $x$, which is geodesic in $\dot T=T\setminus \{x\}$ and which passes through $x$ only at its endpoints.

\begin{mainthm}{B}[Canonical decomposition of a spherical torus for non-odd $\th$]\label{twotrianglesTH}
Let $(T,x)$ be a spherical torus with one conical point of angle $2\pi\th$ such that $\th\in (1,\infty)\setminus (2\mathbb Z+1)$. 
\begin{itemize}
\item[(i)]
If $T$ does not have a rectangular involution, then there exists a unique (up to a re-ordering) triple of geodesic loops $\gamma_1,\, \gamma_2,\, \gamma_3$ based at $x$ that cut $T$ into two congruent  strictly balanced spherical triangles. 
\item[(ii)]
If $T$ has a rectangular involution, there exist exactly two (unordered) triples
of geodesic loops such that each of them cuts $T$ into two congruent balanced triangles. 
Moreover, such triangles are semi-balanced.
These two triples are exchanged by the rectangular involution.
\end{itemize} 
\end{mainthm}

By Theorem \ref{twotrianglesTH} to each spherical torus $T$ one can associate an essentially unique  balanced spherical triangle $\Delta(T)$. 
Such uniqueness will permit us to reduce the description of the moduli space $\MSPH_{1,1}(\th)$ to that of the moduli space of balanced triangles of area $\pi(\th-1)$.

%{\bf Odd $\th$}. 
\subsubsection{$\th$ odd integer}

The case when $\th$ is an odd integer is quite different,
as not all spherical metrics are invariant under the
unique (nontrivial) conformal involution $\sigma$ of the tori.
We begin by stating our result for metrics that are $\sigma$-invariant.

\begin{mainthm}{C}[Topology of $\MSPH_{1,1}(2m+1)^\sigma$]\label{mainodd2}
Fix an integer $m\geq 1$ and consider 
the moduli space $\MSPH_{1,1}(2m+1)^\sigma$ of tori with a $\sigma$-invariant spherical metric of area $4m\pi$.
\begin{itemize}
\item[(a)]
As a topological space, $\MSPH_{1,1}(2m+1)^\sigma$ is homeomorphic to the disjoint union of $\ceil{\frac{m(m+1)}{6}}$ two-dimensional 
open disks.
\item[(b)]
$\MSPH_{1,1}(2m+1)^\sigma$ is naturally endowed with the structure of a $2$-dimensional orbifold
with $\ceil{\frac{m(m+1)}{6}}$ connected components, which can be described as follows.
\begin{itemize}
\item[(b-i)]
If $m\not\equiv 1\pmod 3$, then all components are isomorphic to the quotient $\Dodd$
of $\twosimpl=\{y\in\RR^3_+\,|\,y_1+y_2+y_3=2\pi\}$ by the trivial $\ZZ_2$-action.
Hence, every point of $\MSPH_{1,1}(2m+1)^\sigma$ has orbifold order $2$.
\item[(b-ii)]
If $m\equiv 1\pmod 3$, then one component is isomorphic to the quotient
$\Dodd'$ of $\twosimpl$ by $\ZZ_2\times\ASy_3$,
where $\ZZ_2$ acts trivially and $\ASy_3$ acts by cyclically permuting the coordinates
of $\twosimpl$, and all the other components are isomorphic to $\Dodd$.
Hence, one point of $\MSPH_{1,1}(2m+1)^\sigma$ has orbifold order $6$
and all the other points have order $2$.
\end{itemize}
\end{itemize}
\end{mainthm}

\begin{remark}
Similarly to Remark \ref{rmk:orbi},
as a topological space
$\MSPH_{1,1}(2m+1)^\sigma$ is homeomorphic to
$\MSPH_{0,4}\left(m+\frac{1}{2},\frac{1}{2},\frac{1}{2},\frac{1}{2}\right)/\Sy_3$
(though they are not isomorphic as orbifolds).
%By Theorem \ref{mainodd2}, such topological spaces are homeomorphic
%to the disjoint union of $\ceil{\frac{m(m+1)}{6}}$ open $2$-disks.
Thus, Theorem \ref{mainodd2} has connection with the results contained in \cite{Lin2}, \cite{Lin5}, \cite{eisenstein},
\cite{E} and \cite{EG}.
\end{remark}

The following description of the moduli space
of tori with metrics that are not necessarily $\sigma$-invariant
will be deduce from Theorem \ref{mainodd2}.

\begin{mainthm}{D}[Topology of $\MSPH_{1,1}(2m+1)$]\label{mainodd} 
For each positive integer $m$ the moduli space 
$\MSPH_{1,1}(2m+1)$ is a $3$-dimensional orbifold with 
$\ceil{\frac{m(m+1)}{6}}$ connected components.
\begin{itemize}
\item[(i)]
If $m\not\equiv 1\pmod 3$, then all components of $\MSPH_{1,1}(2m+1)$ are isomorphic to
the quotient $\Modd$ of $\twosimpl\times\RR$ by the involution $(y,t)\mapsto (y,-t)$.
\item[(ii)]
If $m\equiv 1\pmod 3$, then one component of $\MSPH_{1,1}(2m+1)$ is isomorphic
to the quotient $\Modd'$ of $\twosimpl\times\RR$ by
$\ZZ_2\times\ASy_3$, where $\ZZ_2$ acts via the involution $(y,t)\mapsto (y,-t)$
and the alternating group $\ASy_3$
acts by cyclically permuting the coordinates of $\twosimpl$.
All the other components are isomorphic to $\Modd$.
\end{itemize}
The locus $\MSPH_{1,1}(2m+1)^\sigma$ of $\sigma$-invariant metrics correspond to $t=0$.
\end{mainthm}

In order to understand what happens for spherical metrics that are not necessarily
$\sigma$-invariant,
recall that a spherical surface is called {\it coaxial} if its monodromy belongs to a one-parameter subgroup $\SO(3,\RR)$. By \cite[Theorem 5.2]{Lin2},  every spherical metric on $T$ with one conical point of angle $2\pi(2m+1)$ is {\it coaxial}. For this reason the moduli spaces $\MSPH_{1,1}(2m+1)$ are three-dimensional. Indeed, every spherical metric belongs to a $1$-parameter family
of metrics that induce the same projective structure (see also Corollary \ref{coaxtorus}):
we will say that metrics in the same $1$-parameter family are {\it projectively equivalent}.
Hence $\MSPH_{1,1}(2m+1)^\sigma$ is also isomorphic to the moduli space
$\MSPH_{1,1}(2m+1)/\proj$ of projective classes of spherical tori of area $4m\pi$.

Another major difference with the non-odd case is that
the forgetful map $\MSPH_{1,1}(\th)\rightarrow \Mcal_{1,1}$ is proper (see \cite{MP:systole}) and surjective for $\th$ non-odd. This does not happen for odd $\th$
(see \cite{Lin1}).
The boundary of $\MSPH_{1,1}(2m+1)/\proj$ inside the space of $\mathbb{C}P^1$-structures
describes interesting real-analytic curves (see \cite{Lin2}), that are investigated in the
sequel paper \cite{EGMP}.

Theorems \ref{mainodd2}-\ref{mainodd} will rely on the following result
that links moduli spaces of tori to moduli spaces of balanced triangles with integral angles.

\begin{mainthm}{E}[Canonical decomposition of a spherical torus with odd $\th$]\label{twointiranglesTH} 
Fix a spherical torus with one conical point of angle $2\pi(2m+1)$. 
In the same projective class
there exists a unique spherical torus $(T,x)$ that admits an isometric orientation-preserving involution. Moreover, there exists a unique collection of three geodesic loops $\gamma_1,\, \gamma_2,\, \gamma_3$ based at $x$ that cut $T$ into two congruent  balanced spherical triangles $\Delta$ and $\Delta'$  with integral angles $\pi\cdot (m_1,m_2,m_3)$.
\end{mainthm}

\subsubsection{$\th$ even integral}
%{\bf Even $\th$.} 
Our final main result concerns the moduli spaces $\MSPH_{1,1}(2m)$, where $m$ is a positive integer. It is known \cite{Lin2, ET} that such moduli spaces have a natural holomorphic structure with respect to which they are compact Riemann surfaces with punctures. This is the unique conformal structure which makes the forgetful  map to $\Mcal_{1,1}$ holomorphic. With this structure
$\MSPH_{1,1}(2m)$ is an algebraic curve.

\begin{mainthm}{F}[$\MSPH_{1,1}(2m)$ is a Belyi curve]\label{eventheta} 
For each integer $m>0$ there exists a subgroup  $G_m<\mathrm{SL}(2,\mathbb Z)$ of index $m^2$ such that the orbifold $\MSPH_{1,1}(2m)$ is biholomorphic  to the quotient $\mathbb H^2/G_m$. 
Such $G_m$ is non-normal for $m>1$.
Moreover, the points in $\mathbb H^2/G_m$ that project to the geodesic ray $[i,\infty)$ in the modular curve $\mathbb H^2/\mathrm{SL}(2,\mathbb Z)$ correspond to tori $T$ such that the triangle $\Delta(T)$ has one integral angle.
\end{mainthm}

\subsection{Analytic representation of spherical metrics}

Let $(T,x)$ be a spherical torus with a conical singularity at $x$ of angle $2\pi\th$.
The pull-back of the spherical metric 
via the universal cover $\mathbb C=\tilde{T}\to T$ has 
%a density $\rho$ with respect to 
%the Euclidean metric in $\mathbb C$ (so that this pull-back is defined by the
%line element $\rho(z)|dz|$. 
area element $e^u|dz|^2$.
Then function $u$
satisfies the non-linear PDE
\begin{equation}\label{mean}
\Delta u+2e^u=2\pi(\th-1)\delta_{\Lambda},
\end{equation}
where $\delta_{\Lambda}$ is the sum of delta-functions over the lattice
$\Lambda$ and $T$ is biholomorphic to $\mathbb C/\Lambda$. 
So our results describe the moduli spaces
of pairs $(\Lambda,u)$, where $u$ is a $\Lambda$-periodic
solutions of (\ref{mean}).

Equation (\ref{mean}) is the simplest representative of the
class of ``mean field equations''
which have important applications in physics \cite{T}.

The general solution of (\ref{mean}) can be
expressed in terms of the developing map $f:\mathbb C\to\mathbb CP^1$
related to the conformal factor $u$ by
$$u=\log\frac{4|f'|^2}{(1+|f|^2)^2}$$
and the developing map $f=w_1/w_2$ is the ratio of two linearly independent
solutions $w_1,\, w_2$ of the Lam\'e equation: 
\begin{equation}\label{lame}
w''=\left(\frac{\th^2-1}{4}\wp-c\right) w,
\end{equation}
where $\wp$ is the Weierstrass function of the lattice $\Lambda$
and $c\in\mathbb C$ is an accessory parameter.
So our results can be also interpreted as a description of the
moduli space of projective structures
on tori
whose monodromies are subgroups of $\mathrm {SO}(3,\RR)$.

Most of the known results on spherical tori are formulated in
terms of equations (\ref{mean}) and (\ref{lame}).
For example, it is proved in \cite{Lin3} that when $\th$ is
not an odd integer, then the Leray--Schauder degree
of the non-linear operator in (\ref{mean}) equals $\lfloor(\th+1)/2\rfloor$.
An especially well-studied case is the classical Lam\'e equation
(\ref{lame}) where $\th$ is an integer, see \cite{Lin2,EGMP} and
references there. Solutions of (\ref{lame}) with odd integer $\th$
are special functions of mathematical physics, \cite{WW, Ma}.

%%%%%%%%%%%%%%%%%%%%%%%%%%%%%%%%%%%%%%%%%%%%%%%%%%%%%%

\subsection{The idea of the proof of Theorem \ref{mainonodd}}

Here we give a brief summary of the proof of Theorem \ref{mainonodd}, since various parts of it stretch through the whole paper.
Fix $\th>1$ not odd and consider spherical tori with a conical point of angle $2\pi\th$, and area $2\pi(\th-1)$.
The proof of Theorem \ref{mainonodd} develops through the following steps.

\begin{itemize}
\item
On every torus the unique non-trivial conformal involution is an isometry (Proposition \ref{conformaliso} (i)).
\item
Every spherical torus is obtained by gluing two isometric copies of a spherical balanced triangle
with labelled vertices in an essentially unique way (Theorem \ref{twotrianglesTH}, proven in Section \ref{sec:TH}).
Such result has a clear refinement for tori with a $2$-marking (namely, a labelling of its $2$-torsion points), see Construction \ref{fund2}.
\item
The doubled space $\MT^{\pm}_{bal}(\th)$
of balanced triangles of area $\pi(\th-1)$ is the double of the space $\MT_{bal}(\th)$ of balanced triangles of area $\pi(\th-1)$
and it describes oriented balanced triangles up to some identifications
that only involve semi-balanced triangles (Definition \ref{def-oriented}).
\item
The space $\MT_{bal}(\th)$ is an orientable connected surface with boundary and its topology is completely determined
(Proposition \ref{baltopo})
and so is the topology of $\MT^{\pm}_{bal}(\th)$
(Proposition \ref{or-bal-tri}).
\item
As a topological space, the space $\MSPH_{1,1}^{(2)}(\th)$ of isomorphism classes of $2$-marked tori
is homeomorphic to $\MT^{\pm}_{bal}(\th)$ (Theorem \ref{torusbalancehomeo}).
\item
As an orbifold, $\MSPH_{1,1}^{(2)}(\th)$ is isomorphic to
the quotient of $\MT^{\pm}_{bal}(\th)$ by the trivial $\ZZ_2$-action.
This allows to determine the topology and the orbifold Euler characteristic of $\MSPH_{1,1}^{(2)}(\th)$
(Theorem \ref{markedtori}).
\item
As an orbifold, the map $\MSPH_{1,1}^{(2)}(\th)\rar\MSPH_{1,1}(\th)$ that forgets the $2$-marking is an unramified  orbifold $\Sy_3$-cover,
where $\Sy_3$ acts on $\MSPH_{1,1}^{(2)}(\th)$ by permuting the $2$-markings (Remark \ref{orbifold-correct}).
This allows to describe the points in $\MSPH_{1,1}(\th)$ of orbifold order greater than $2$ (Proposition \ref{toriwithsym})
and to determine the topology and the orbifold Euler characteristic of $\MSPH_{1,1}(\th)$
(Theorem \ref{mainonodd}, towards the end of Section \ref{secmainodd}).
\end{itemize}

%%%%%%%%%%%%%%%%%%%%%%%%%%%%%%%%%%%%%%%%%%%%%%%%%%%%
\subsection{Content of the paper}

The relation between spherical tori with one conical point and balanced spherical triangles is established in
Section \ref{sec:Voronoi}, which culminates in the proof of Theorem \ref{twotrianglesTH}.
The section contains a careful analysis of the Voronoi graph of a torus and of the action of the unique
non-trivial conformal involution $\sigma$ on its spherical metric.

In Section \ref{sec:balanced} we describe the topology of the space of balanced triangles
of area $\pi(\th-1)$ and of its double, separately considering the case $\th$ non-odd and $\th$ odd. 
Here we visualize the space of spherical triangles with assigned area, which is a manifold,
by looking at its image (which we call {\it{carpet}}) through the angle map $\Theta$. 
The balanced carpet will turn out to be a useful tool in computing the topological invariants
of the space of balanced triangles.

In Section \ref{sec:moduli} we describe the topology of the moduli spaces of spherical tori with one conical point, 
endowed with the Lipschitz metric (which we study in Section \ref{secLipschitz}). 
For $\th$ non-odd, we first establish a homeomorphism between the doubled space
of balanced triangles and the topological space of $2$-marked tori using tools from Section \ref{secLipschitz}.
Then we prove Theorem \ref{mainonodd}. For $\th$ odd, 
we first prove Theorem \ref{twointiranglesTH} using results from Section \ref{sec:Voronoi} and Section \ref{sec:balanced},
which immediately allows us to prove part (a) of Theorem \ref{mainodd2}.
%Then we show that the space of $\sigma$-invariant $2$-marked spherical tori is homeomorphic
%to the doubled space of balanced triangles. 
Then we endow our moduli space of 
$\sigma$-invariant spherical tori with a $2$-dimensional orbifold structure and we prove
part (b) of Theorem \ref{mainodd2}. Finally, using one-parameter projective deformations of $\sigma$-invariant spherical metrics,
we put a $3$-dimensional orbifold structure on the moduli space of (not necessarily $\sigma$-invariant) tori
and we prove Theorem \ref{mainodd}.

In Section \ref{sec:Belyi} we analyse the moduli space of tori with $\th$ even
and we prove Theorem \ref{eventheta} by identifying such moduli space
to a Hurwitz space of covers of $\CC P^1$ branched at three points.
This permits us to exhibit such moduli space as a Belyi curve and to characterize tori that sit on the
$1$-dimensional skeleton of its dessin.

Section \ref{secLipschitz} deals with properties of the Lipschitz metric
on moduli spaces of spherical surfaces with conical points with area bounded from above.
The main result of the section is Theorem \ref{propersys} on properness of the
inverse of the systole function. Then the treatment is specialized to tori with one conical point
of angle $2\pi\th$ with $\th$ non-odd (or with $\th$ odd and a $\sigma$-invariant metric).
The section culminates with establishing the homeomorphism between the space $2$-marked tori and the doubled space
of balanced triangles, needed in Section \ref{sec:moduli}.
A last remark explains how to use such result to endow our moduli spaces with an orbifold structure.

In the short Appendix \ref{sec:monodromy} we prove a general $\mathrm{SU}(2)$-lifting theorem
for the monodromy of a spherical surface, and we apply to the case of $\th$ odd and $\th$ even to explain
their special features.

%%%%%%%%%%%%%%%%%%%%%%%%%%%%%%%%%%%%%%%%%%%%%%%%%%%%

%%%%%%%%%%%%%%%%%%%%%%%%%%%%%%%%%%%%%%%%%%%%%%%%%%%%

\subsection{Acknowledgement} We would like to thank Andrei Gabrielov, Alexander Lytchak, Anton Petrunin, and Dimitri Zvonkine
for useful discussions.

A. E. was supported by NSF grant DMS-1665115.

G. M. was  partially supported by INdAM and by grant PRIN 2017 ``Moduli and Lie theory''.

D. P. was supported by EPSRC grant EP/S035788/1.

%\input torus-one-point-intro.tex

%%%%%%%%%%%%%%%%%

\section{Voronoi diagram and proof of Theorem \ref{twotrianglesTH}}\label{sec:Voronoi}

In this section we will study the Voronoi graphs  of  spherical tori $(T,x,\th)$ with one conical point and prove Theorem \ref{twotrianglesTH}. 

\subsection{Properties of Voronoi graphs, functions and domains}

In this subsection we remind the definition of Voronoi graph \cite[Section 4]{MP:systole} and apply it to spherical tori with one conical point.

\begin{definition}[Voronoi function and Voronoi graph]\label{voronoiDefintion} Let $S$ be a surface with a spherical metric and conical points $\bm{x}$. The {\it Voronoi function} $\Vor_S:S\rar\RR$ is defined as $\Vor_S(p):=d(p,\bm{x})$. 
The {\it Voronoi graph} $\Gamma(S)$ is the locus of points $p\in\dot{S}$ at which
the distance $d(p,\bm{x})$ is realized by two or more geodesic arcs joining $p$ to $\bm{x}$. We will simply write $\Gamma=\Gamma(S)$ when no ambiguity is possible.
The {\it Voronoi domains} of $S$ are connected components of the complement  $S\setminus\Gamma(S)$. Each Voronoi domain $D_i$ contains a unique conical point $x_i$ and this point is the closest conical point to all the points in the domain.

\end{definition}

Various properties of Voronoi functions, graphs, and domains of spherical surfaces were proven in \cite[Section 4]{MP:systole}, and the following lemma lists some of the facts needed here. To formulate the last two properties we need one more definition.

\begin{definition}[Convex star-shaped polygons]
Let $D$ be a disk with a spherical metric, containing a unique conical point $x\in D$ and such that its boundary is composed of a collection of geodesic segments. We say that $D$ is a {\it convex and star-shaped polygon} if any two neighbouring sides of $D$ meet under an interior angle smaller than $\pi$ and for any point $p\in D$ there is a unique geodesic segment that joins $x$ with $p$.
\end{definition}

\begin{proposition}[Basic properties of the Voronoi function and graph]\label{recallVoron} Let $S$ be a spherical surface of genus $g$ with conical points $x_1,\ldots, x_n$.
\begin{itemize}

\item[(i)] The Voronoi function is bounded from above by $\pi$, namely $\Vor_S<\pi $.

\item[(ii)] The Voronoi graph $\Gamma(S)$  is a graph with geodesic edges embedded in $S$ and contains at most $-3\chi(\dot S)=6g-6+3n$ edges. 

\item[(iii)]  The valence of each vertex of $\Gamma(S)$ is at least three. For any point $p\in \Gamma(S)$ its valence coincides with the \emph{multiplicity} $\mu_p$, i.e., there exist exactly $\mu_p$ geodesic segments in $S$ of length $\Vor_S(p)$ that join $p$ with conical points of $S$.  

\item[(iv)] The metric completion of  each Voronoi domain\footnote{The metric completion can differ from the closure of the domain inside $S$, see the rightmost example in Figure~\ref{fig:threvor}.} is a convex and star-shaped polygon with a unique conical point in its interior.

\item[(v)] Let $\gamma$ be an open edge of $\Gamma(S)$. Let $D_i$ and $D_j$ be two Voronoi domains adjacent to $\gamma$. Let $\Delta\subset D_i$ and $\Delta'\subset D_j$ be the two triangles with one vertex $x_i$ or $x_j$ correspondingly, and the opposite side $\gamma$. Then $\Delta$ and $\Delta'$ are anti-conformally isometric by an isometry fixing $\gamma$. 
\end{itemize} 
\end{proposition}

\begin{proof} (i) This is proven in \cite[Lemma 4.2]{MP:systole}.

(ii) This is proven in \cite[Lemma 4.5, Corollary 4.7]{MP:systole}.

(iii) The valence of vertices is at least three by \cite[Corollary 4.7]{MP:systole}. The valence of a point on $\Gamma(S)$ coincides with its multiplicity by \cite[Lemma 4.5]{MP:systole}.

(iv) The convexity is proven in \cite[Lemma 4.8]{MP:systole}. The fact that each domain is star-shaped follows from the fact that each point $p$ in it can be joined by a unique geodesic segment of length $\Vor_S(p)$ with the conical point. Such a segment varies continuously with $p$, since $\Vor_S(p)<\pi$.  
%4) This is proven in \cite[Lemma 4.8]{MP:systole}.

(v) To find the isometry between $\Delta$ and $\Delta'$ just notice that by definition each point $p\in \gamma$ can be joined by two geodesics of the same length with $x_i$ and $x_j$. Also these two geodesics intersect $\gamma$ under the same angle. The isometry between the triangles is obtained by the map exchanging each pair of such geodesics.
\end{proof}

\begin{figure}[!ht]
\vspace{-0.5cm}
\begin{center}
\includegraphics[scale=0.4]{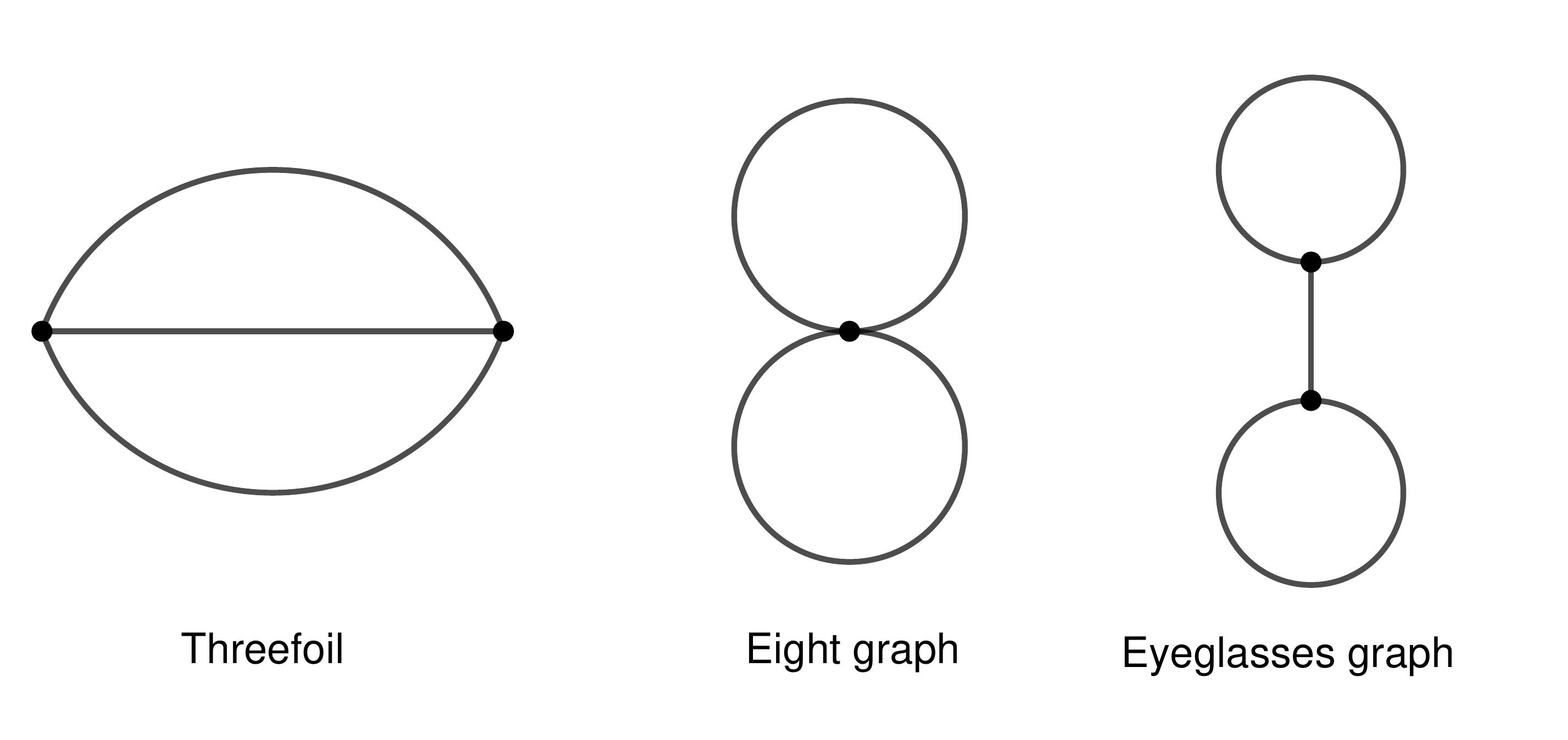}
\end{center}
\vspace{-0.5cm}
\caption{Voronoi graphs on a sphere with three conical points}
\label{fig:threegraphs}
\end{figure}

\begin{example}[Voronoi graph in a sphere with $3$ conical points]
Let $S$ be a sphere with three conical points. It follows from Proposition \ref{recallVoron} (ii) that the Voronoi graph $\Gamma(S)$ is either a trefoil graph or an eight graph, or an eyeglasses graph, see Figure~\ref{fig:threegraphs}. Indeed, $\Gamma(S)$ splits $S$ into three disks, and it has at most three edges.  
\end{example}

The next definition and remark explain how to define Voronoi functions and graphs for spherical polygons, mimicking Definition \ref{voronoiDefintion}.

\begin{definition}[Voronoi function and graph of a polygon]
Let $P$ be a spherical polygon with vertices $\bm{x}$. The Voronoi function  $\mathcal{V}_P:P\rightarrow\mathbb{R}$ is defined  as $\Vor_P(p):=d(p,\bm{x})$. The Voronoi graph $\Gamma(P)$ of $P$ consists of points $p$ of two types: first, the points for which there exists at least two geodesic segments of length $d(p,\bm{x})$ that join $p$ with $\bm{x}$; second, the points $p$ on $\partial P$ for which the closest vertex of $P$ does not lie on the edge to which $p$ belongs. 
\end{definition}

\begin{remark}[Doubling a polygon: Voronoi function and graph]\label{polygonRemark}
To each spherical polygon $P$ one can associate a sphere $S(P)$ with conical singularities by {\it doubling}\footnote{
Given a topological space $X$ and a closed subset $A$, the {\it{doubling of $X$ along $A$}} is obtained
from $X\times\{0,1\}$ by identifying $(a,0)\sim (a,1)$ for every $a\in A$.}
%Given a metric length space $X$ with a closed subset $A\subset X$, the doubling of $X$ across $A$ is obtained by gluing two copies of $X$ along $A$.} 
$P$ across its boundary. Such a sphere has an anti-conformal isometry that exchanges $P$ and 
its isometric copy $P'$, and fixes their boundary. It is easy to see that the function $\Vor_{S(P)}$ restricts to $\Vor_P$ on $P\subset S$ and to $\Vor_{P'}$ on $P'\subset S$. One can also check that the Voronoi graph $\Gamma(S(P))$ is the union $\Gamma(P)\cup \Gamma(P')$. As a result, the statements of Proposition \ref{recallVoron} have their analogues for spherical polygons. 
 \end{remark}

The following lemma gives an efficient criterion permitting one to verify whether a given geodesic graph on a spherical surface is in fact its Voronoi graph.

\begin{lemma}[Voronoi graph criterion]\label{vorcrit} Let $S$ be a spherical surface of genus $g$ with conical points $x_1,\ldots, x_n$ and let $\Gamma'(S)\subset S$ be a finite graph with geodesic edges  embedded in $S$. Then $\Gamma'(S)=\Gamma(S)$ if and only if the following two conditions hold.
\begin{itemize}
\item[(a)] 
$S\setminus \Gamma'(S)$ is a union of disks whose metric completions are convex and star-shaped polygons each with a unique conical point in its interior. 
\item[(b)] 
For each point $p\in \Gamma'(S)$ all geodesic segments, that join $p$ with some conical point of $S$ and intersect $\Gamma'(S)$ only at $p$, have the same length. 
\end{itemize}
\end{lemma}

\begin{proof} Since by Proposition \ref{recallVoron} the graph $\Gamma(S)$ satisfies the conditions (a) and (b), we only need to prove the ``only if'' direction. 

For each conical point  $x_i$ let $D_i$ be the Voronoi domain of $x_i$ (namely the connected component of  $S\setminus \Gamma(S)$ that contains $x_i$), and let $D_i'$ be the  component of $S\setminus \Gamma'(S)$ containing $x_i$. Let's assume by contradiction that there is a point $p\in D_i$ that is not contained in $D_i'$. By definition of $D_i$ there is a unique geodesic segment $\gamma(p)$ of length $\Vor_S(p)$ that joins $p$ with $x_i$. Denote by $\gamma'(p)$ the connected component of the intersection $\gamma(p)\cap D_i'$ that contains $x_i$ and let $p'\notin D_i'$ be the point in its closure. Clearly $p'$ belongs to $\Gamma'(S)$. By (a) each component of $S\setminus \Gamma'(S)$ is star-shaped, so using (b) we get a second (different from $\gamma'(p)$) geodesic segment of length $\Vor_S(p')$ that joins $p'$  with a conical point. Hence $p'\in\Gamma(S)$, which  contradicts the fact that $p'$ is in $D_i$.

We proved that $D_i\subset D_i'$ for each $i$. It follows that $D_i=D_i'$, hence $\Gamma'(S)=\Gamma(S)$. 
\end{proof}

%The next lemma characterises  Voronoi graphs of spheres with three conical points in terms of their conical angles. Such graphs can  
%
%By a tripod we call a graph with two vertices 

\begin{lemma}[Voronoi graphs of a sphere with three conical points]\label{vorthreesphere} Let $S$ be a sphere with three conical points $x_i$ of conical angles $2\pi\th_i$. %obtained by doubling of a spherical triangle $x_1x_2x_3$ with angles $\pi\th_1$, $\pi\th_2$, $\pi\th_3$. 
\begin{itemize}
\item[(i)] $\Gamma(S)$ is a trefoil if and only if $\th_1$, $\th_2$, $\th_3$ satisfy the triangle inequality strictly. 
\item[(ii)] $\Gamma(S)$ is an eight graph if and only if $\th_i=\th_j+\th_k$ for some permutation $(i, j ,k)$ of $\{1,2,3\}$. 
\item[(iii)] $\Gamma(S)$ is an eyeglasses graph if and only if  $\th_i> \th_j+\th_k$ for some permutation $(i, j ,k)$ of $\{1,2,3\}$.
\item[(iv)] In the cases (i) and (ii) the vertices of $\Gamma(S)$ are equidistant from $x_1,x_2,x_3$. In the case (iii) the vertices of $\Gamma(S)$ are not equidistant from $x_1,x_2,x_3$.  
\end{itemize}
\end{lemma}

\begin{figure}[!ht]
\vspace{-0.3cm}
\begin{center}
\includegraphics[scale=0.5]{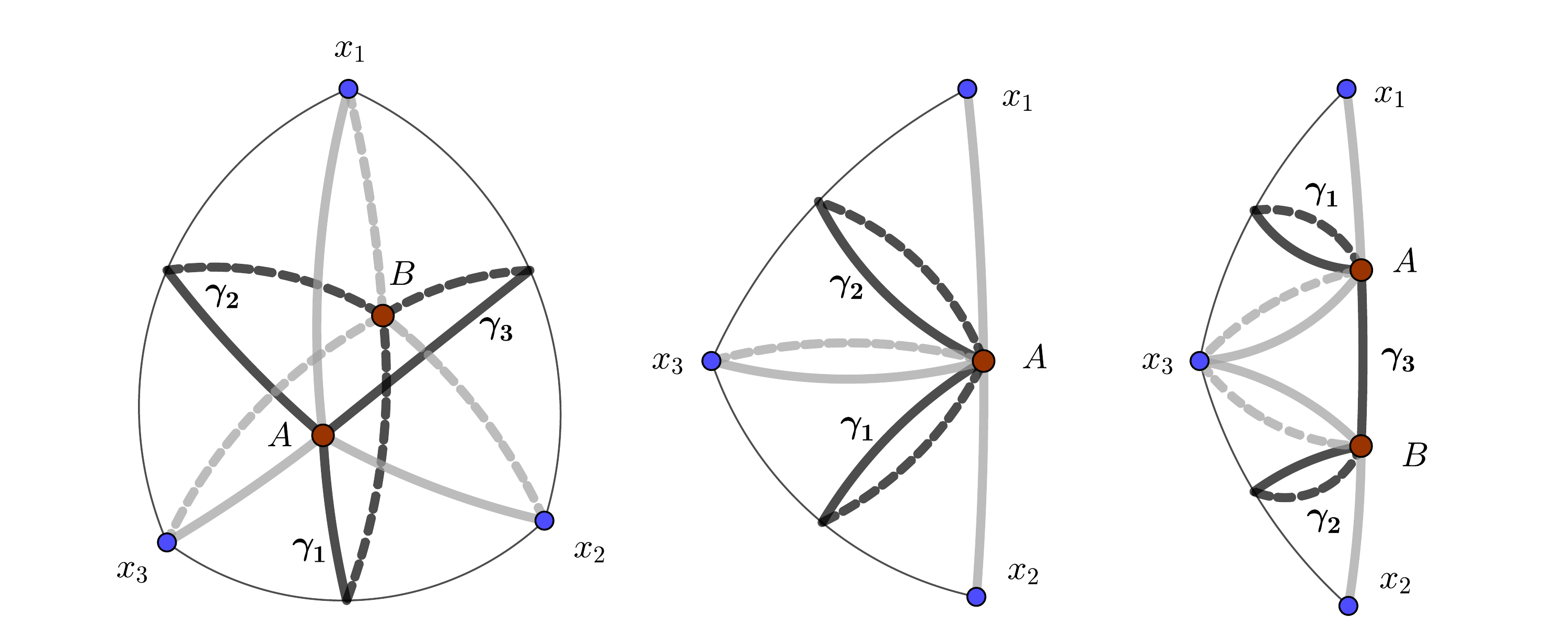}
\end{center}
\vspace{-0.5cm}
\caption{Three types of spheres}
\label{fig:threvor}
\end{figure}

\begin{proof} It is enough to prove the ``only if'' parts of claims (i), (ii), (iii) 
the cases are mutually exclusive and so the ``if'' part will follow as well.

For the proof of the ``only if'' part all three cases are treated in a similar way. Let us consider, for example, the case when $\Gamma(S)$ is a trefoil graph. Let's show that in this case $\th_i$ satisfy  the triangle inequality strictly. Denote the two vertices of $\Gamma(S)$ as $A$ and $B$. The three edges if $\Gamma(S)$ cut $S$ into three Voronoi disks, each of which contains one conical point. Let us denote these three  segments of $\Gamma(S)$ by $\gamma_1,\gamma_2,\gamma_3$, as it is shown on the leftmost picture in Figure~\ref{fig:threvor}. Let us join each of the $x_i$ with the vertices $A$ and $B$ by geodesics $x_iA$, $x_iB$ of lengths $\Vor_S(A)$ and $\Vor_S(B)$ correspondingly. These geodesic segments are depicted in gray. 

Consider now the spherical quadrilaterals $Ax_3Bx_1$, $Ax_1Bx_2$ and $Ax_2Bx_3$ into which the gray geodesics cut $S$. 
It follows from Proposition \ref{recallVoron} (v)
%that each of three quadrilaterals has an anti-conformal isometric involution fixing the $\gamma_i$ therein contained. In particular 
for $i,j\in\{1,\,2,\,3\}$ that the angles of $Ax_iBx_j$ at $x_i$ and $x_j$ are equal. This implies that $\th_1,\th_2,\th_3$ satisfy the triangle inequality strictly.
 
(ii, iii) In a similar way one treats the cases when $\Gamma(S)$ is an eight graph or an eyeglasses graph, the corresponding two pictures are shown in Figure~\ref{fig:threvor}.
 
(iv) This is clear from the way $\Gamma(S)$ is embedded in $S$, see Figure~\ref{fig:threvor}. In particular, if $\Gamma(S)$ is an eyeglasses graph, $d(A,x_1)=d(A,x_3)<d(A,x_2)$ and $d(B,x_2)=d(B,x_3)<d(B,x_1)$. 
\end{proof}

\subsection{The circumcenters of balanced triangles}

It is well-known that  the circumcenter of a Euclidean triangle $\Delta$ is contained in $\Delta$ if and only  $\Delta$ is not obtuse. Moreover, in the case when $\Delta$ is right-angled, the circumcenter is the mid-point of the hypotenuse. It is also a classical fact that the circumcenter of a Euclidean triangle is the point of intersection of the axes\footnote{The {\it{axis}} of a segment is the perpendicular through the midpoint of such segment.} 
of its sides.
%perpendicular bisectors 
The next theorem is a generalisation of the above statements to spherical triangles. By an {\it involution triangle} we mean a triangle that admits an anti-conformal isometric involution that fixes one vertex and exchanges the other two\footnote{Note that every Euclidean or hyperbolic isosceles triangle admits an isometric involution exchanging the equal sides. This is not the case for spherical triangles, for example the triangle with angles $5\pi/2,13\pi/2, 9\pi/2$ is equilateral but clearly has no symmetries.}. 

\begin{theorem}[Circumcenters of balanced triangles]\label{circumcenter} Let   $\Delta$ be a spherical triangle with vertices $x_1,\,x_2,\,x_3$.
\begin{itemize}
\item[(i)] The triangle $\Delta$ contains a point $O$ equidistant from $x_1,\,x_2,\,x_3$
%such that $d(O,x_1)=d(O,x_2)=d(O,x_3)$ 
if and only if $\Delta$ is balanced.
\item[(ii)] The point $O$ (equidistant from $x_1,\,x_3,\,x_3$) is in the interior of $\Delta$ if and only if $\Delta$ is strictly balanced. The point $O$ is the midpoint of a side of $\Delta$ if an only if $\Delta$ is semi-balanced.

\item[(iii)] If $\Delta$ is strictly balanced, then the geodesic segments $Ox_1, Ox_2, Ox_3$ cut $\Delta$ into three involution triangles. %Moreover, $\Gamma(\Delta)$ is the union of the symmetry axes of these three triangles.

\item[(iv)] Suppose that $\Delta$ is semi-balanced and the angle  $\angle x_i=\pi\th_i$ is the largest one. Then $O$ is the midpoint of the side opposite to $x_i$, and $x_iO$ cuts $\Delta$ into two involution triangles. 
%Moreover, $\Gamma(\Delta)$ is the union of the symmetry axes of these two triangles.

\end{itemize}

\end{theorem}

To prove this theorem we need the following lemma. 

\begin{lemma}[Some isosceles triangles are involution triangles]\label{isoceltri} Let $\Delta$ be a spherical triangle with vertices $q_1,\,q_2,\, q_3$ and denote by $|q_iq_j|$ the length of the side $q_iq_j$. Suppose that $|q_1q_2|=|q_1q_3|<\pi$ and $\angle q_1<2\pi$. Then there is an isometric reflection $\tau$ of $\Delta$ that fixes $q_1$ and exchanges $q_2$ with $q_3$. In particular $\angle q_2=\angle q_3$. Moreover, $\tau$ pointwise fixes a geodesic segment that joins $q_1$ with the midpoint of $q_2q_3$ and splits $\Delta$ into two isometric triangles. Furthermore, $|q_2q_3|<2\pi$.
\end{lemma}
\begin{proof} Consider first the case when $\angle q_1=\pi$. In this case $\Delta$ can be isometrically identified with a digon so that $q_1$ is identified with the midpoint of one of its sides. Since each digon has an isometric reflection fixing the midpoints of both sides, the lemma holds.

From now on we assume  that $\angle q_1\ne \pi$. 
Consider  the unique  spherical triangle $\Delta'\subset \mathbb S^2$ with vertices $q_1',q_2',q_3'$ such that $|q_1'q_2'|=|q_1'q_3'|=|q_1q_2|$, $\angle q_1'=\angle q_1$, and $\area(\Delta')<2\pi$. We will show that $\Delta'$ admits an isometric embedding into $\Delta$ that sends $q_i'$ to $q_i$. This will  prove the lemma since this implies that $\Delta$ is isometric to a triangle obtained by gluing a digon to the side $q_2'q_3'$ of $\Delta'$. And such a triangle clearly has an isometric reflection $\tau$. This will also  prove that $|q_2q_3|<2\pi$, since $|q_2'q_3'|<2\pi$ and either $|q_2q_3|=|q_2'q_3'|$ or $|q_2q_3|+|q_2'q_3'|=2\pi$.

To prove the existence of the embedding, denote by $\iota: \Delta\to \mathbb S^2$ the developing map of triangle $\Delta$. We may assume that $\iota(q_i)=q_i'$, $\iota(q_1q_2)=q_1'q_2'$, and $\iota(q_1q_3)=q_1'q_3'$.
Note that $\iota$ sends $q_2q_3$ to the unique\footnote{This circle is unique since $\angle q_1\ne \pi$, and also it intersects the segments $q_1'q_2', q_1'q_3'$ only at the points $q_2',q_3'$.} geodesic circle that contains $\iota(q_2)$ and $\iota(q_3)$. Hence, it is not hard to see that the preimages of $\Delta'$ in $\Delta$ form a union of some number of isometric copies of $\Delta'$. One of them, that contains sides $q_1q_2$ and $q_1q_3$ of $\Delta$, is the embedding we are looking for.
\end{proof}

\begin{remark}
We note that this lemma is sharp in the sense that none of the two conditions  $|q_1q_2|=|q_1q_3|<\pi$ and $\angle q_1< 2\pi$ can be dropped.
\end{remark}

\begin{proof}[Proof of Theorem \ref{circumcenter}] (i) Let $S(\Delta)$ be the sphere obtained by doubling $\Delta$ across its boundary, i.e., by gluing $\Delta$ with the triangle $\Delta'$ that is anti-conformally isometric to $\Delta$. Then by Remark \ref{polygonRemark} the graph $\Gamma(S(\Delta))$ is the union of $\Gamma(\Delta)$ with $\Gamma(\Delta')$. 

Suppose first that $\Delta$ contains a point $O$ equidistant from all $x_i$'s. Then, since the restriction of $\Vor_{S(\Delta)}$ to $\Delta$ equals $\Vor_{\Delta}$, we see that $O$ is equidistant from $x_i$ on $S$ as well. So by Proposition \ref{recallVoron} (iii) the point $O$ corresponds to a vertex of $\Gamma(S(\Delta))$ of multiplicity at least $3$. Furthermore, by Lemma \ref{vorthreesphere} (iv) we conclude that $\Gamma(S)$ is either a trefoil or an eight graph. Hence again by Lemma \ref{vorthreesphere} the triangle $\Delta$ is balanced.

Suppose now that $\Delta$ is balanced, i.e., $\th_1,\,\th_2,\,\th_3$ satisfy the triangle inequality. Then by Lemma \ref{vorthreesphere} (i), (ii) the graph $\Gamma(S(\Delta))$ is a trefoil or a eight graph, and so by Lemma \ref{vorthreesphere} (iv) there is a point $O$ in $S$ equidistant from all $x_i$. It follows  that $\Delta$ contains such a point as well.

(ii) We first prove the ``only if'' direction. Suppose that $O$ is in the interior of $\Delta$. Then $\Gamma(S(\Delta))$ has two vertices of valence $3$. So according to (i), $\Gamma(S(\Delta))$ is a trefoil. Hence, $\Delta$ is strictly balanced by Lemma \ref{vorthreesphere} (i). 

Suppose that $O$ is on the boundary of $\Delta$. Without loss of generality assume that $O$ is on  the side of $\Delta$ opposite to $x_1$. For $i=1,2,3$ let $\gamma_i$ be the geodesic segment of length $\Vor_{\Delta}(O)$ that joins $O$ with $x_i$. Let $\gamma_i'$ be the image of $\gamma_i$ in $\Delta'\subset S(\Delta)$ under the anti-conformal involution. Since the multiplicity of $O$ in $\Gamma(S)$ is at most $4$ we conclude that $\gamma_2=\gamma_2'$, $\gamma_3=\gamma_3'$. Hence, $O$ is the midpoint of the side $x_2x_3$.

To prove the ``if'' direction one needs to apply Lemma \ref{vorthreesphere} (iv). Indeed, if $\Delta$ is strictly balanced, $\Gamma(S(\Delta))$ has two vertices of multiplicity $3$ and one of them lies in $\Delta$. If $\Delta$ is semi-balanced, $\Gamma(S(\Delta))$ has one vertex and it has to lie on the boundary of $\Delta$.

\begin{figure}[!ht]
\vspace{0cm}
\begin{center}
\includegraphics[scale=0.4]{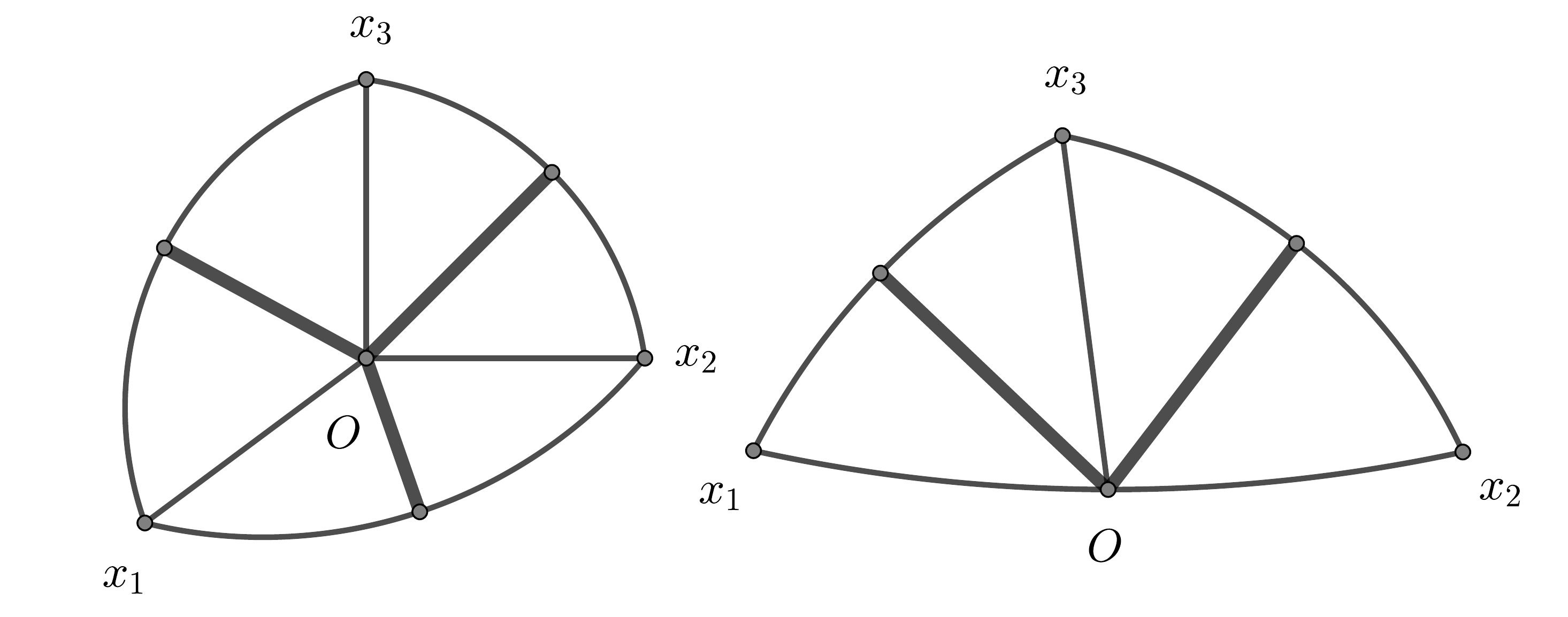}
\end{center}
\vspace{-0.5cm}
\caption{Voronoi graphs of balanced triangles}
\label{fig:triangVor}
\end{figure}

(iii) Since $\Delta$ is strictly balanced, by (ii) there is a point  $O$ in the interior of $\Delta$ equidistant from  points $x_1,x_2,x_3$. Since $\Vor_{\Delta}(O)<\pi$, we have $|Ox_1|=|Ox_1|=|Ox_3|<\pi$. Hence all three isosceles triangles $x_iOx_j$ are involution triangles by Lemma \ref{isoceltri}. 

%Denote by $p_i$  the midpoint of the side of $\Delta$ opposite to $x_i$, and let $Op_i$ be the axis of symmetry of the corresponding involution triangle. Denote  $Op_1\cup Op_2\cup Op_3$  by $\Gamma'$ and let's prove that $\Gamma'=\Gamma(\Delta)$.  For this we apply Lemma \ref{vorcrit} to $\Gamma'$. We need to show that the conditions a) and b) are satisfied. Consider one of connected components of $\Gamma'\setminus\Delta$, say the quadrilateral $x_3p_1p_2$
%It is clear that $\Gamma'$ cuts $\Delta$ in three quadrilaterals that are star-shaped with respect to vertices $x_i$, see Figure  Moreover, 

(iv) This proof is identical to the proof of (iii) and we omit it.
\end{proof}

\begin{remark}\label{vorbalanced}  Theorem \ref{circumcenter}  can be used to construct the Voronoi graph $\Gamma(\Delta)$ of a balanced triangle  $\Delta$ with
vertices $x_1,\,x_2,\,x_3$. Indeed, according to this theorem, the geodesic segments $Ox_i$ cut $\Delta$ into three or two involution triangles, and using a variation of Lemma \ref{vorcrit} one can show that $\Gamma(\Delta)$ is the union of symmetry axes of these triangles. See Figure~\ref{fig:triangVor}.
\end{remark}

We will see that some results we are interested in about balanced triangles indeed
concern the following class of triangles.

\begin{definition}[Short-sided triangles]
A spherical triangle {\it short-sided} if all its sides have length $l_i<2\pi$.
In this case, we set $\bar{l}_i:=\min(l_i,\,2\pi-l_i)$.
\end{definition}

Theorem \ref{circumcenter} has the following two simple corollaries. 
%We call a spherical triangle {\it short-sided} if all its sides are shorter than $2\pi$.

\begin{corollary}[Balanced triangles are short-sided]\label{lessthan2pi} Let $\Delta$ be a balanced triangle with vertices $x_1,\,x_2,\, x_3$. Then $\Delta$ is short-sided, i.e.~$|x_ix_j|<2\pi$.
\end{corollary}
\begin{proof} Let us treat the case when $\Delta$ is strictly balanced. The semi-balanced case is similar. By Theorem \ref{circumcenter} (iii) the triangle
$\Delta$ can be cut into $3$ involution triangles $x_iOx_j$ where $\angle O<2\pi$ and $|Ox_i|=|Ox_j|<\pi$. Applying Lemma \ref{isoceltri} to the triangle $x_iOx_j$ we conclude that $|x_ix_j|<2\pi$.
\end{proof}

\begin{corollary}[Short geodesic in a balanced triangle]\label{balancedsystole} Let $\Delta$ be a balanced triangle with vertices $x_1,\,x_2,\, x_3$. Suppose that $\{i,j,k\}=\{1,2,3\}$ and such that the value $\bar{l}_k=\min(|x_ix_j|, 2\pi-|x_ix_j|)$ is minimal. Then there is a geodesic segment $\gamma_{\Delta}$ in $\Delta$ that joins $x_i$ with $x_j$ and such that $\ell(\gamma_{\Delta})=\bar{l}_k\le 2\pi/3$, which in fact realizes the minimum distance between distinct vertices.
\end{corollary}
\begin{proof} Let us again treat the case when $\Delta$ is strictly balanced.  Let $x_iOx_j$ be three involution triangles in which $\Delta$ is cut. Consider the developing map $\iota:\Delta\to \mathbb S^2$. Then for each $\{i,j,k\}=\{1,2,3\}$ the value $\bar{l}_k$ is equal to the distance between $\iota(x_i)$ and $\iota(x_j)$ on $\mathbb S^2$, and so
$d(x_i,x_j)\geq d(\iota(x_i),\iota(x_j))=\bar{l}_k$.
For this reason,  it is not hard to see, that the minimum of the value $\bar{l}_k$ is attained for the triangle $x_iOx_j$ for which the angle at $O$ is the minimal one. In particular in such a triangle the angle at $O$ is at most $2\pi/3$. It follows that there is a geodesic segment $\gamma_{\Delta}$ in such a triangle $x_iOx_j$ of length less than $2\pi/3$ that joins $x_i$ and $x_j$. Since it cuts out of $x_iOx_j$ a digon with one side $x_ix_j$ we conclude that $\ell(\gamma_{\Delta})=\bar{l}_k=d(x_i,x_j)$.
\end{proof}

\subsection{Isometric conformal involutions on tori}

In this short section we prove the following useful proposition.

\begin{lemma}[Invariance of projective structures on one-pointed tori]\label{lemma:proj-inv}
Let $(T,x)$ be a flat one-pointed torus
and let $\sigma$ be its unique nontrivial conformal involution.
Then every projective structure on $T$ whose Schwarzian derivative has at worst a double pole
at $x$ is invariant under $\sigma$.
\end{lemma}
\begin{proof}
 We represent our torus $T$ as $\mathbb C/\Lambda$ where
$\Lambda$ is a lattice in $\mathbb C$, and suppose that $x$ corresponds to the lattice points.
We also endow $T$ with the corresponding projective structure.

%Consider the projective structure associated with the metric,
%and let $\iota:\tilde{T}=\mathbb C\to\mathbb C$ be its developing map.
The involution $\sigma$ pulls back to the map $z\mapsto-z$ on $\tilde{T}=\mathbb C$.
The Schwarzian derivative (see, for example \cite{SG})
of a projective structure %with respect to the projective structure associate to the flat metric
is a quadratic differential on the torus $T$. By hypothesis, it has at worst a double pole at $x$.
%, so it is
%of the form $g(z)dz^2$ where $g$ is an elliptic function
%with double poles on $\Lambda$. 
The vector space of such quadratic differentials is $2$-dimensional, generated by
the constants and the Weierstrass elliptic function. Hence, all its elements are invariant
under the involution $\sigma$, and so are all solutions of the associated Schwarz equations.
As a consequence, all such projective structures are $\sigma$-invariant.
\end{proof}

\begin{proposition}[Spherical metrics and conformal involution]\label{conformaliso} 
Let $\sigma$ be the unique conformal involution of 
a spherical torus $T$ that fixes the unique conical point $x$. 
\begin{itemize}
\item[(i)]
If $\th\notin 2\mathbb{Z}+1$, then $\sigma$ is an isometry.
\item[(ii)]
If $\th\in 2\mathbb Z+1$, then each projective equivalence class of spherical metrics
is parametrized by a copy of $\RR$, on which $\sigma$ acts as an orientation-reversing diffeomorphism. Thus, $\sigma$ is an isometry for a unique spherical metric in its projective equivalence class.
\end{itemize}
\end{proposition}

\begin{proof}
Consider the projective structure associated to a spherical metric on $(T,x)$.
By Lemma \ref{lemma:proj-inv}, such projective structure is $\sigma$-invariant.

%Indeed, such functions form a two-dimensional space generated by constants and the Weierstrass elliptic function. 
%Since this differential has only one pole on the torus, the residue at this pole vanishes, and we conclude that $g$ is an even function, that is the differential isinvariant with respect to our involution. 
%So all solutions
%of the Schwarz equation are invariant with respect to the involution.
%
%Consider the following equivalence relation
%on the set of all solutions of the Schwarz equation: $f_1\sim f_2$ if
%$f_1=\phi\circ f_2$ for some $\phi\in {\rm PSU}(2)$. It is evident that
%the involution respects this equivalence relation. 

(i) Every spherical metric is non-coaxial by Lemma \ref{coaxtorus}, and so
%If our metric is not co-axial,
in each projective equivalence class there is at most one spherical metric.
Hence, such metric must be invariant under $\sigma$.

(ii) By Lemma \ref{coaxtorus}, the monodromy $\rho$ of a spherical metric $h$ in $\MSPH_{1,1}(2m+1)$ is coaxial. It cannot be trivial, since this would imply that
$T$ covers $\Sph$ with $x$ as unique ramification point.
Fix an element $\alpha$ of $\pi_1(T)$ such that $\rho(\alpha)=e^X\neq I$
with $X\in\mathfrak{su}_2$.
%The monodromy of $h$ consists of elements of type $e^{sX}$ with $s\in\RR$.
If $\iota$ is the developing map associated to $h$,
all the spherical metrics $h_t$ projectively equivalent to $h$
have developing maps $e^{itX}\circ \iota$ with $t\in\RR$.
Suppose that $\sigma^*h=h_{t_0}$.
Since $\sigma(\alpha)=\alpha^{-1}$, it follows that
$(\rho\circ\sigma)(\alpha)=e^{-X}$ and so $\sigma^*(h_t)=h_{t_0-t}$.
Replacing $h$ by $h_{t_0/2}$, we obtain $\sigma^*(h_t)=h_{-t}$.
It can be observed that $\iota$ is uniquely determined by requiring that it maps
the conical point to the maximal circle fixed by $\rho$.
%
%This completes the proof in non-coaxial case.
%If the reader feels uneasy about the action of $z\mapsto-z$ on
%multivalued functions with an infinite order ramification point at $0$,
%one may think of solutions of the Schwarz equation as meromorphic
%germs at a half-period, say $\omega_1$ with the involution acting at them
%by the rule $f(z)\mapsto f(2\omega_1-z)$.
%
%
%
%When  the angle at the conic singularity is $2\pi(2m+1)$ with integer $m$,
%the metric is coaxial. In this case, there is a whole $1$-parameter
%family of spherical metrics on $T$ that define the same projective structure on $T$. This family can be parametrized by $\mathbb R$,
%and there is exactly one metric in the family which is fixed by the involution.
\end{proof}

Proposition \ref{conformaliso}(ii) 
was proved in \cite[Theorem 5.2]{Lin2}, see also   \cite[Theorem 1]{EG}.

\subsection{Proof of Theorem \ref{twotrianglesTH} }\label{sec:TH}

The goal of this section is to prove Theorem \ref{twotrianglesTH} and to make preparations for the proof of Theorem \ref{mainodd2}. Throughout the whole section we will mainly consider the class of tori that have a conformal isometric involution. By Proposition \ref{conformaliso} we know that such an involution exists automatically in the case when the conical angle is not $2\pi(2m+1)$. We start with the following simple lemma.

\begin{lemma}[Points of $\Gamma$ fixed by a conformal isometric involution]\label{geofromsigma} Let $S$ be a spherical surface with conical points $\bm{x}$ that admits an isometric conformal involution $\sigma$. Let $p$ be a point in $\dot S=S\setminus {\bm x}$ fixed by $\sigma$. Then $p$ belongs to $\Gamma(S)$, its multiplicity $\mu_p$ is even, and there exist exactly $\frac{\mu_p}{2}$ geodesic segments or loops\footnote{We always assume that a geodesic loop or segment can intersect $\bm x$ only at its endpoints.} of lengths $2\Vor_S(p)<2\pi$ based at $\bm{x}$ and passing through $p$. The point $p$ cuts each such geodesic segment into two halves of equal length.
\end{lemma}
\begin{proof}
Consider any geodesic segment $\gamma$ of length $\Vor_S(p)$ that joins $p$ with one of the conical points. Since $\sigma(\gamma)\ne \gamma$ we see that $p$ belongs to $\Gamma(S)$. If $p$ is not a vertex of $\Gamma(S)$ then $\gamma$ and $\sigma(\gamma)$ are the only two geodesic segments of length $\Vor_S(p)$ that join $p$ with $\bm{x}$. Clearly, since $\sigma$ is a conformal involution  the union $\gamma\cup \sigma(\gamma) $ is a geodesic segment or loop based at $\bm{x}$. Its length is less than $2\pi$ by Proposition \ref{recallVoron} (i).

The case when $p$ is a vertex of $\Gamma(S)$ is similar. Since $\sigma$ is a conformal involution and it sends $\Gamma(S)$ to $\Gamma(S)$ we see that the valence of $p$ in $\Gamma_S$ is even. By Proposition \ref{recallVoron} (iii) the number $\mu_p$ of geodesic segments of length $\Vor_S(p)$ that join $p$ with $\bm{x}$ is equal to this valence. Clearly, altogether these $\mu_p$ segments form $\frac{\mu_p}{2}$ geodesic segments (or loops) of length $2\Vor_S(p)$ for all of which $p$ is  midpoint.
\end{proof}

Now, we concentrate on the case of spherical tori with one conical point. It will be convenient for us to recall first the construction of {\it hexagonal} and {\it square} flat tori.

\begin{example}[Flat hexagonal and square tori]\label{regularexample} Let $T_{6}$ and $T_{4}$ be the flat tori obtained by identifying opposite sides of  a regular flat hexagon and a square correspondingly. Denote by $\Gamma_{6}\subset T_{6}$ and $\Gamma_{4}\subset T_{4}$ the graphs formed by the images of polygons boundaries. Then it is easy to check that $\Gamma_{6}$ and $\Gamma_{4}$ are Voronoi graphs in $T_{6}$ and $T_{4}$ with respect to the images of the centres of the polygons.
\end{example}

\begin{lemma}[Voronoi graph of a spherical torus]
Let $T$ be a spherical torus with one conical point and let $\Gamma$ be its Voronoi graph.
Then $\Gamma$ is either a trefoil or an eight graph. In the first case the pair $(T,\Gamma)$ is homeomorphic to the pair $(T_{6},\Gamma_{6})$. In the second case it is homeomorphic to the pair $(T_{4},\Gamma_{4})$.
\end{lemma}
\begin{proof} By \cite[Corollary 4.7]{MP:systole} the Voronoi graph $\Gamma$ has at most three edges and two vertices.
Since the complement to the Voronoi graph is a disk, the graph has at least two edges. 

Suppose first that $\Gamma$ has three edges. By \cite[Corollary 4.7]{MP:systole} the vertices of $\Gamma$ have multiplicity at least $3$, so $\Gamma$ is a trivalent graph with two vertices, i.e., a trefoil or an eyeglasses graph. It is a classical fact that only the trefoil admits an embedding in the torus with a connected complement. Moreover, such an embedding is unique up to a homeomorphism of the torus\footnote{One way to prove this statement is to identify the complement to $\Gamma$ with a regular hexagon, thus inducing a flat metric on $T$ making it isometric to $T_{6}$.}. 
The statement of lemma then clearly holds. The case when $\Gamma$ has two edges is similar. 
\end{proof}

The following is the main proposition on which the proof of Theorem \ref{twotrianglesTH} relies.

\begin{proposition}[From tori to balanced triangles]\label{threegeodesics} 
Let $(T,x)$ be a spherical torus with one conical point $x$ and suppose that $T$ has a non-trivial isometric conformal involution $\sigma$. Let $\Gamma(T)$ be the Voronoi graph of $T$.
\begin{itemize}
\item[(i)]
Suppose $\Gamma(T)$ is a trefoil. Then $\sigma$ permutes the two vertices of $\Gamma(T)$, and  fixes the mid-points $p_1,p_2,p_3$ of the three edges of $\Gamma(T)$. Moreover, there exist exactly three $\sigma$-invariant simple geodesic loops $\gamma_1, \gamma_2,\gamma_3$ based at $x$ such that $\gamma_i$ intersects $\Gamma(T)$ orthogonally at $p_i$. These geodesic loops cut the torus into the union of two congruent strictly balanced triangles that are exchanged by $\sigma$. 
\item[(ii)]
Suppose $\Gamma(T)$ is an eight graph with the vertex $A$. Then $\sigma$ fixes the vertex and the mid-points $p_1,p_2$ of the two edges of $\Gamma(T)$. Moreover there exist four $\sigma$-invariant simple geodesic loops $\gamma_1, \gamma_2, \eta_1,\eta_2$ based at $x$ and uniquely characterised by the following properties.
Each geodesic $\gamma_i$ intersects $\Gamma(T)$ orthogonally at $p_i$. Each geodesic $\eta_i$ passes through $A$ and has length $2d(A,x)$.  Moreover, for $i=1,2$ the triple of loops $\gamma_1,\gamma_2,\eta_i$ cuts $T$ into the union of two congruent semi-balanced triangles that are exchanged by $\sigma$. 
\item[(iii)]
$T$ has a rectangular involution if and only if its Voronoi graph is an eight graph.
For a torus $T$ with a rectangular involution the triangles in which $\gamma_1,\gamma_2,\eta_1$ cut $T$ are reflections of  the triangles in which $\gamma_1,\gamma_2,\eta_2$ cut $T$.
\end{itemize} 
\end{proposition}

\begin{proof}
(i) Since $\sigma$ is an isometry of $T$ it sends $\Gamma(T)$ to itself. Let's denote the vertices of $\Gamma(T)$ by  $A$ and $B$. 
Since their valence is $3$ and $\sigma$ is a conformal isometric involution, $\sigma$ can fix neither $A$ nor $B$. Indeed, begin $\sigma$ of order $2$, if $\sigma$ fixed $A$, then it would fix at least one half-edge outgoing from $A$, and so it would be the identity. Hence  $\sigma$ permutes $A$ and $B$, which implies in particular that $A$ and $B$ are at the same distance from $x$. 

Next, since $\sigma$ is an orientation preserving involution, and $\Gamma(T)$ is a trefoil, from simple topological considerations it follows that $\sigma$ sends each edge $\gamma_i$ of  $\Gamma(T)$ into itself. It follows that the midpoints of the edges $p_1,\,p_2,\,p_3$ are fixed by $\sigma$.

\begin{figure}[!ht]
\vspace{-0.3cm}
\begin{center}
\includegraphics[scale=0.4]{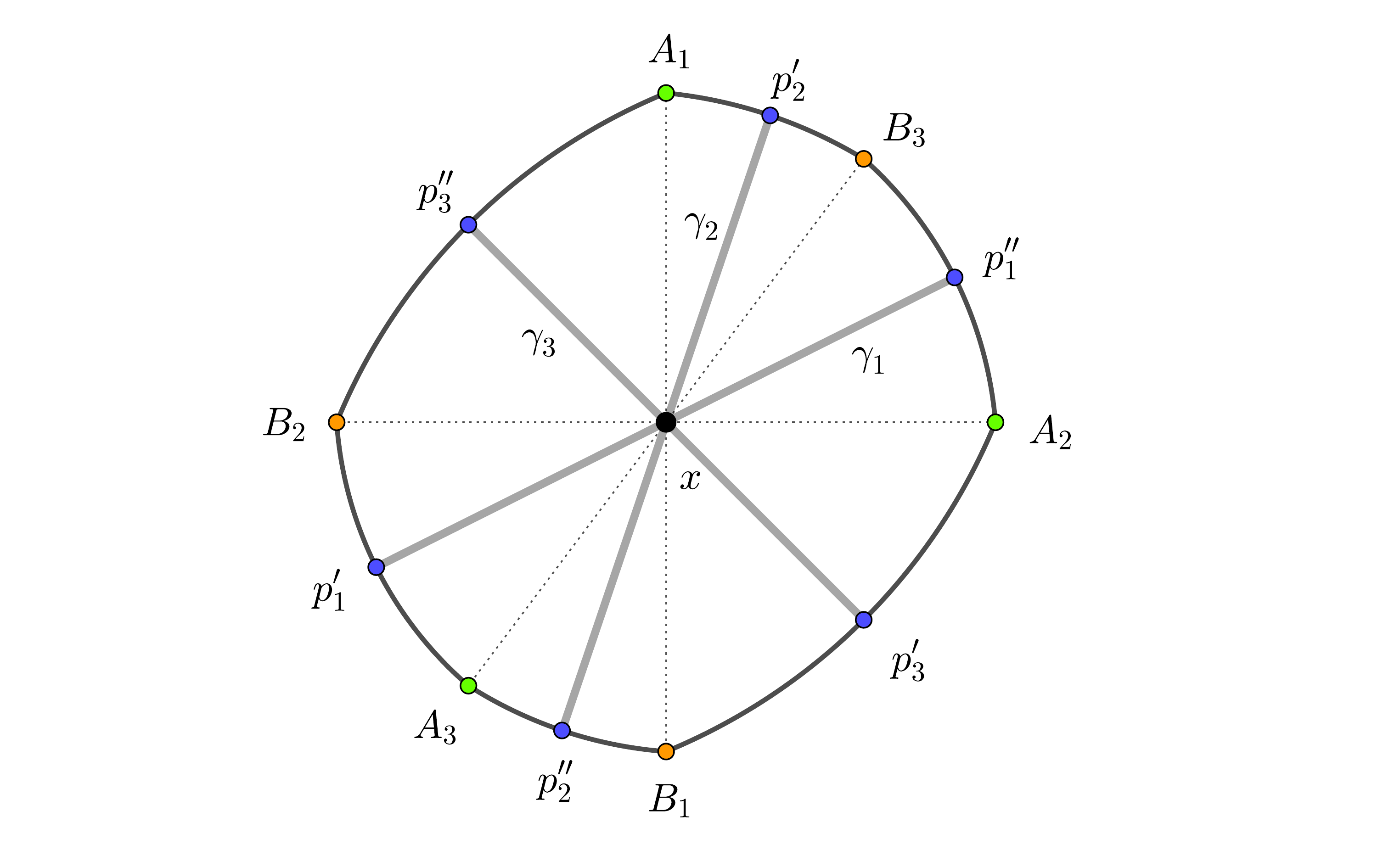}
\end{center}
\vspace{-0.5cm}
\caption{Trefoil case}
\label{fig:vortrefoil}
\end{figure}

Let us now cut $T$ along $\Gamma(T)$ and consider the completion $\bar D$ of the obtained open disk. Clearly, $\bar D$ is a spherical hexagon with the conical point $x$ in its interior.  Moreover, $\sigma$ induces an isometric involution on $\bar D$ without fixed points on $\partial \bar D$. It follows that $\sigma$ sends each vertex of $\bar D$ to the opposite one. 

Next, let's denote the vertices of $\bar D$ by $A_1, B_2, A_3, B_1, A_2, B_3$ as is shown in Figure~\ref{fig:vortrefoil}. Here all the points $A_i$ correspond to $A$ and $B_i$ to $B$ when we assemble $T$ back from the disk. In a similar way we mark midpoints of the sides of $\bar D$ by $p_i'$ and $p_i''$.

%\begin{center}
%\includegraphics[scale=0.4]{FatColoured}
%\end{center}

According to Lemma \ref{geofromsigma}, for each $i$ there is a geodesic loop $\gamma_i$ of length $2d(p_i,x)$ based at $x$ for which $p_i$ is the midpoint. Let us  show that $\gamma_1,\gamma_2,\gamma_3$ cut $T$ into two equal strictly balanced triangles whose vertices are identified to the point $x$.

Indeed, the first triangle, which we will call $\Delta_A$, is assembled from three quadrilaterals $A_1p_3''xp_2'$, $A_2p_1''xp_3'$, $A_3p_2''xp_1'$. The second triangle $\Delta_B$ is assembled from the remaining three quadrilaterals. Clearly, $\sigma(\Delta_A)=\Delta_B$, so these two triangles are congruent.

Finally,  $\Delta_A$ is strictly balanced according to Theorem \ref{circumcenter} (i), indeed the point $A$ lies in the interior of  $\Delta_A$ and is at distance $d(A,x)$ from all the vertices of $\Delta_A$.

%Finally, to show that $\Delta_A$ is strictly balanced we need to prove that the angles $\angle p_2'xp_3''$, $\angle p_3'xp_1''$, $\angle p_1'xp_2''$ satisfy strict triangle inequality. This is so thanks to the fact that the following pairs of triangles are  isomteric 
%$$ p_3''xA_1\cong  p_3'xA_2,\,  p_2'xA_1\cong  p_2''xA_3,\, p_1''xA_2\cong  p_1'xA_3.  $$

(ii) Let us now consider the case when $\Gamma(T)$ is an eight graph with a vertex labelled by $A$. Clearly, $A$ is fixed by $\sigma$ since this is the unique point of $\Gamma(T)$ of valence $4$.

\begin{figure}[!ht]
\vspace{-0.3cm}
\begin{center}
\includegraphics[scale=0.4]{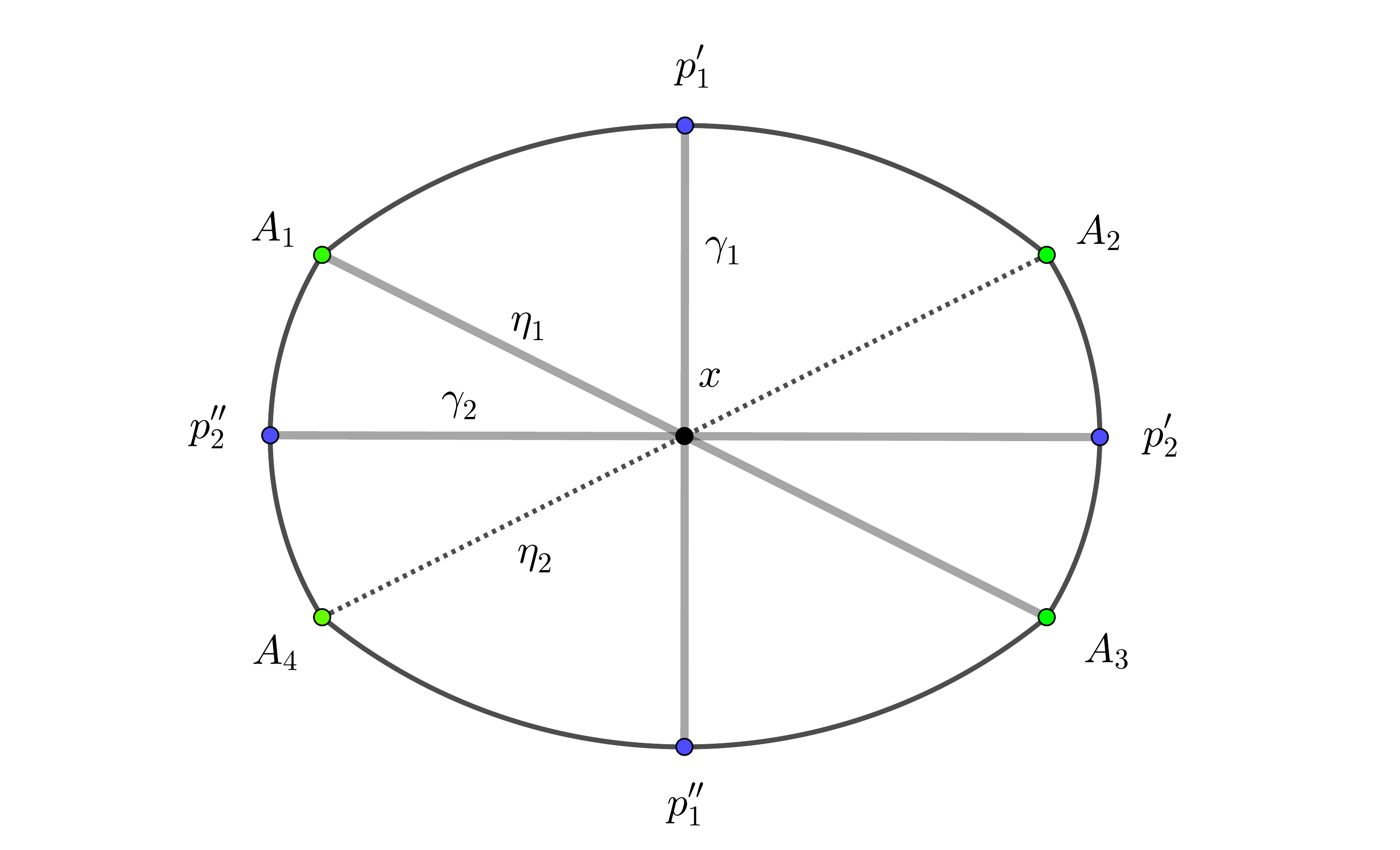}
\end{center}
\vspace{-0.5cm}
\caption{Eight graph case}
\label{fig:voreightgraph}
\end{figure}

%\begin{center}
%\includegraphics[scale=0.4]{sphericalrectangle}
%\end{center}

As before we see that the midpoints $p_1, p_2$ of the two edges of $\Gamma(T)$ are fixed by $\sigma$ and this gives us two $\sigma$-invariant geodesic loops $\gamma_1$ and $\gamma_2$. 
To construct $\eta_1$ and $\eta_2$ we apply Lemma \ref{geofromsigma} to the point $A$.

Now let us cut $T$ along the Voronoi graph $\Gamma(T)$  and consider the completion $\bar D$ of the obtained open disk. Clearly, this disk is a quadrilateral with one conical point in the interior. Let us mark the vertices of this quadrilateral and the midpoints of its edges as it is shown in Figure~\ref{fig:voreightgraph}.

As before, the loops $\gamma_1, \gamma_2, \eta_1$ cut  $T$ into two congruent triangles exchanged by $\sigma$. To show that these triangles are semi-balanced consider one of these triangles obtained as a union of two triangles $A_1xp_2''$, $ A_3xp_1''$ and the quadrilateral $xp_1'A_2p_2'$. To assemble this triangle one has to identify the pairs of sides $(A_1p_2'', A_2p_2')$  and $(A_2p_1', A_3p_1'')$. The resulting triangle is semi-balanced by Theorem \ref{circumcenter} (ii).

(iii) Suppose first that $\Gamma(T)$ is an eight graph. Then we are in the setting of the case 2 of this proposition. Let us construct an involution $\tau_1$  of $\bar D$ that fixes pointwise $\gamma_1$. We define $\tau_1$ so that $\tau_1(A_1)=A_2$, $\tau_1(A_3)=A_4$. Then in order show that $\tau_1$ extends to $\bar D$ it is enough to show that the triangle $A_1xA_4$ is isometric to $A_2xA_3$ and that the geodesic $\gamma_1$ is the axis of symmetry of both triangles $A_1xA_2$ and $A_3xA_4$.  The former statement follows from Proposition \ref{recallVoron} (v). To prove the latter statement, note again that  $A_1xA_2$ is isometric to $A_4xA_3$ by  Proposition \ref{recallVoron} (v) and then compose this isometry with $\sigma$. This induces desired reflections on both triangles $A_1xA_2$ and $A_4xA_3$. The involution $\tau_2$ fixing $\gamma_2$ is constructed in the same way.

%Applying Proposition \ref{recallVoron} (v) and Theorem \ref{circumcenter} (iv) we see that the quadrilateral $\bar D$ has two reflections $\tau_1$ and $\tau_2$. Here $\tau_1$ pointwise fixes $\gamma_1$ and sends $A_1$ to $A_2$ and $A_4$ to $A_3$, while $\tau_2$ pointwise fixes $\gamma_2$ and sends $A_1$ to $A_4$ and $A_2$ to $A_3$. Clearly, both $\tau_1$ and $\tau_2$ induce a rectangular involutions on $T$. For example $\tau_1$ pointwise fixes $\gamma_1$ and the geodesic $A_1A_4$. 

%Suppose now that $T$ has a rectangular involution $\tau$. Let us show that $\Gamma(T)$ is an eight graph. Assume by contradiction that this is not the case. Then we are in the setting of part 1). Since $\tau$ is a rectangular involution, its fixed locus is a union of two geodesic loops. Exactly one of these loops passes through $x$.  Clearly, the other loop must be contained entirely in $\Gamma(T)$. But the latter is impossible since $\Gamma(T)$ is composed of three geodesic segments meeting at angles different from $\pi$.  This is a contradiction.

Suppose now that $T$ has a rectangular involution $\tau$. Let us show that $\Gamma(T)$ is an eight graph.  Since $\tau$ is a rectangular involution, its fixed locus is a union of two disjoint geodesic loops. One of these loops passes through $x$ while the other one, say $\xi$, is a simple smooth closed geodesic. For any point $p\in \xi$ there exist at least two length minimizing geodesic segments that join it with $x$ (they are exchanged by $\tau$). It follows that $\xi$ lies in $\Gamma(T)$. And since a trefoil graph can't  contain a smooth simple closed geodesic, we conclude that $\Gamma(T)$ is an eight graph. 
\end{proof}
Later  we will need the following statement, which is a part of the proof of Proposition \ref{threegeodesics}.

\begin{remark}\label{bysector} Suppose we are in the case (ii) of Proposition \ref{threegeodesics}. Consider the four sectors in which geodesic loops $\eta_1$ and $\eta_2$ cut a neighbourhood of $x$. Then for each $i=1,\,2$ the geodesic loop $\gamma_i$ bisects two of these sectors. %To see this, take, for example, the triangle $A_1xA_2$ in $\bar D$. One can apply Lemma \ref{isoceltri} to  $A_1xA_2$, and  we see that the segment $p_1'x$ lying on $\gamma_1$ splits the triangle into two isometric ones.
\end{remark}

The final preparatory proposition of this subsection is the converse to Proposition \ref{threegeodesics}.

\begin{proposition}[From balanced triangles to tori]\label{balancedtorus} 
Let $\Delta$ be a balanced triangle and let $\Delta'$ be a triangle congruent to it. Let $T(\Delta)$ be the torus obtained by identifying the sides of $\Delta$ and $\Delta'$ through orientation-reversing isometries.
\begin{itemize}
\item[(i)]
The Voronoi graph $\Gamma(T(\Delta))$ coincides with the union in $T(\Delta)$ of $\Gamma(\Delta)$ and $\Gamma(\Delta')$. 
\item[(ii)]
If $\Delta$ is strictly balanced then the Voronoi graph $\Gamma(T(\Delta))$ has two vertices. Moreover,  the images of the three sides of $\Delta$ in $T(\Delta)$ coincide with three canonical geodesic loops $\gamma_1,\gamma_2,\gamma_3$ on $T(\Delta)$ constructed in Proposition \ref{threegeodesics} (i).
\item[(iii)]
If $\Delta$ is semi-balanced then $\Gamma(T(\Delta))$ has one vertex.  Moreover the images of the three sides of $\Delta$ in $T(\Delta)$ coincide with three canonical geodesic loops $\gamma_1,\gamma_2,\eta_i$ on $T(\Delta)$ constructed in Proposition \ref{threegeodesics} 2). Here the side of $\Delta$ opposite to the largest angle of $\Delta$ corresponds to $\eta_i$. 
\end{itemize}
\end{proposition}

\begin{figure}[!ht]
\vspace{-0.5cm}
\begin{center}
\includegraphics[scale=0.4]{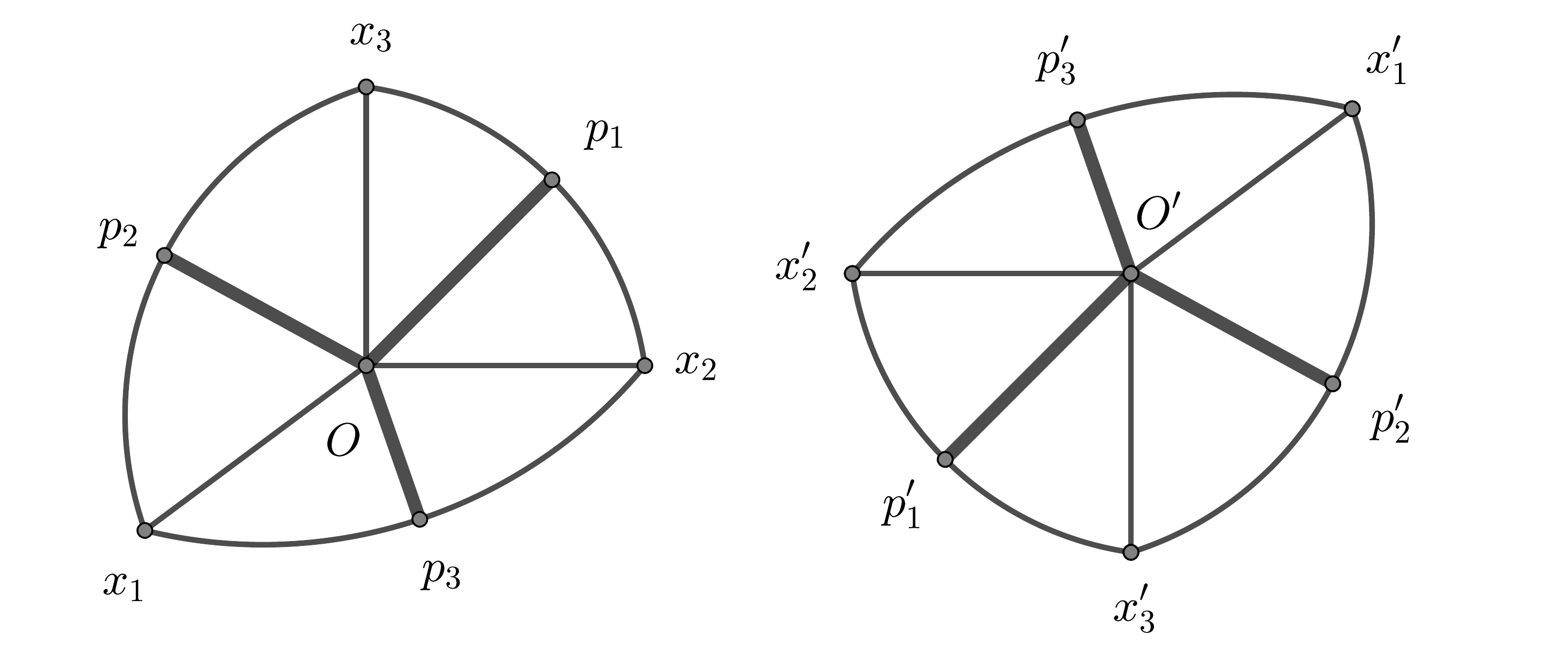}
\end{center}
\vspace{-0.5cm}
\caption{Two isomorphic triangles $\Delta$ and $\Delta'$}
\label{fig:triangVorTor}
\end{figure}

\begin{proof}
(i) Assume first that $\Delta$ is strictly balanced. Let $\check{\Gamma}$  be the graph obtained as the union $\Gamma(\Delta)\cup \Gamma(\Delta')$. In order to prove that  $\check{\Gamma}=\Gamma(T(\Delta))$, it is enough to show that  $\check{\Gamma}$ satisfies the properties (a) and (b) of Lemma \ref{vorcrit}. 

Recall that by Theorem \ref{circumcenter} (ii) there is a point $O$ in the interior of $\Delta$ that is equidistant from points $x_i$.  
Denote by $p_i$ and $p_i'$ the midpoints of sides opposite to $x_i$ and $x_i'$ as in Figure~\ref{fig:triangVorTor}. Then by Remark \ref{vorbalanced}, $\Gamma(\Delta)$ is the union of the segments $Op_i$ and $\Gamma(\Delta')$ is the union of the segments $Op_i'$. It follows that $T(\Delta)\setminus \check{\Gamma}$ is a convex and star-shape with respect to $x$, which means that property (a) of Lemma \ref{vorcrit} holds. As for property (b), it holds since $\Gamma(\Delta)$ and $\Gamma(\Delta')$ are Vornoi graphs of $\Delta$ and $\Delta'$.
 
The case when $\Delta$ is semi-balanced case is treated in the same way, so we omit it. 
 
(ii) Since $\Delta$ is strictly balanced, it follows from (i) that $\Gamma(T(\Delta))$ has two vertices. Now, it follows from (i) that for any permutation $\{i,\,j,\,k\}$ the side $x_ix_j\subset T(\Delta)$ intersects an edge of $\Gamma(T(\Delta))$ at its midpoint and it is orthogonal to it at this point. Hence, by Proposition \ref{threegeodesics} (ii) each geodesic $x_ix_j$ coincides with the geodesic loop $\gamma_k$. 

(iii) The proof of this result is similar the case (ii) and we omit it.
\end{proof}

\begin{remark}We note that the statement of Proposition \ref{balancedtorus}  does not hold for any unbalanced triangle. Indeed, if $\Delta$ is unbalanced one can still construct a torus $T(\Delta)$ from $\Delta$ and its copy of $\Delta'$. However, the union of the Voronoi graphs of $\Delta$ and $\Delta'$ will be an eyeglasses graph in $T(\Delta)$. Such a graph can never be the Voronoi graph of a torus with one conical point.
\end{remark}

Now we are ready to prove Theorem \ref{twotrianglesTH}.

\begin{proof}[Proof of Theorem \ref{twotrianglesTH}] 
Let $T$ be a spherical torus with one conical point of angle $2\pi\th$ with $\th\notin 2\mathbb Z+1$. According to Proposition \ref{conformaliso}, there exists a conformal isometric involution $\sigma$ on $T$. Hence we can apply Proposition \ref{threegeodesics}. 
In particular, by Proposition \ref{threegeodesics} (iii) the torus $T$ has a rectangular involution if and only if $\Gamma(T)$ is an eight graph.

(i) The Voronoi graph $\Gamma(T)$ of $T$ is a trefoil and we get a collection of three geodesics $\gamma_1,\gamma_2,\gamma_3$ that cut $T$ into two congruent strictly balanced triangles. Such a collection of geodesics is unique on $T$ by Proposition \ref{balancedtorus}. 

(ii) The Voronoi graph $\Gamma(T)$ is an eight graph, and by Proposition \ref{threegeodesics} we get two triples of geodesics $\gamma_1,\gamma_2,\eta_1$ and $\gamma_1,\gamma_2,\eta_2$ both cutting $T$ into two congruent semi-balanced triangles. Again, it follows from Proposition \ref{balancedtorus} that these two triples are the only ones that cut $T$ into two isometric balanced triangle, and they are exchanged by the rectangular involution.
\end{proof}

%\input torus-one-point-diagram.tex

%%%%%%%%%%%%%%%%%

\section{Balanced spherical triangles}\label{sec:balanced}

%The main goal of this section is to prove Theorem \ref{mainonodd}. 

The main goal of this section is to describe the space of balanced spherical triangles with assigned area.
To do this, we  recall in Section \ref{subsecAlltriangles} several theorems describing the inequalities satisfied by the angles of spherical triangles. We also give explicit constructions of such triangles. Section \ref{subsecsmooth} is mainly expository. It recalls the results from \cite{EGnew} that the space $\MT$ of all (unoriented) spherical triangles has a structure of a three-dimensional real-analytic manifold. From this we deduce that the space of balanced triangles of a fixed non-even area is a smooth bordered surface.  In Section \ref{subsectbalancefix} we describe a natural cell decomposition of the space $\MT_{bal}(\th)$ of all balanced triangles of fixed area $\pi(\th-1)$ with $\th\notin 2\mathbb Z+1$.

%This will permit us to prove Theorem \ref{mainonodd} in Section \ref{secmainodd}.

\subsection{The shape of spherical triangles}\label{subsecAlltriangles}

We start this section by recalling the classifications \cite{eremenko:three} of spherical triangles. In fact, such triangles are in one-to-one correspondence with spheres with a spherical metric with three conical points, provided we exclude spheres and triangles with all integral angles. Indeed, for each $S^2$ with a spherical metric and three conical points, that are not all integral, there is a unique isometric anti-conformal involution $\tau$, such that $S^2/\tau$ is a spherical triangle. Conversely, for each spherical triangle $\Delta$ we can take the sphere $S(\Delta)$ glued from two copies of it.

%To state two theorems from \cite{eremenko:three} we introduce the following notations. 
It will be useful to introduce the following notation.

\begin{notation}
We denote by $\mathbb Z_{e}^3$ the subset of $\mathbb Z^3$ consisting of triples $(n_1,n_2,n_3)$ with $n_1+n_2+n_3$ even. By $d_1$ we denote the $\ell_1$ distance in $\mathbb R^3$ defined by $d_1(\bm{v},\bm{w})=\sum_i|v_i-w_i|$.
If a spherical triangle has angles $\pi\cdot(\th_1,\th_2,\th_3)$, then we call $(\th_1,\th_2,\th_3)\in\RR^3$ its associated {\it{angle vector}}. 
\end{notation}

\begin{figurehere}
%\begin{figure}[!ht]
%\vspace{-0.3cm}
%\begin{center}
\includegraphics[scale=0.4]{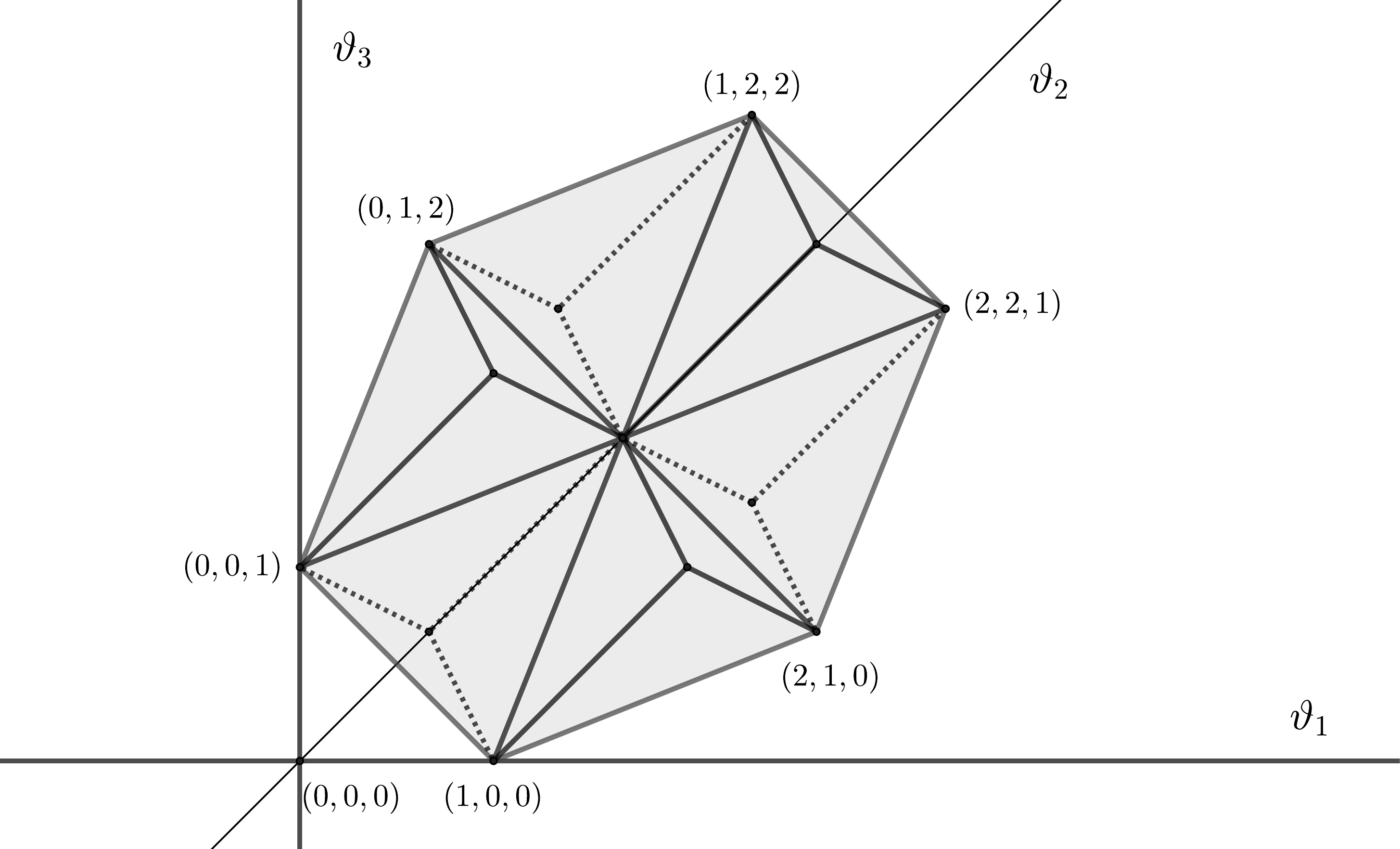}
%\end{center}
%\vspace{-0.2cm}
\caption{Angle vectors of spherical triangles}
\label{fig:tetrahedrons}
%\end{figure}
\end{figurehere}

We collect the results into three subsections, depending on the number of integral angles, and we remind
that there cannot be a triangle with exactly two integral angles.

\subsubsection{Triangle with no integral angle}

The first result we want to recall from \cite{eremenko:three} is the following.

\begin{theorem}[Triangles with non-integral angles \cite{eremenko:three}]\label{nonintriang} Suppose $\th_1,\th_2,\th_3$ are positive and none of them is integer. A spherical triangle with angles $\pi\cdot (\th_1,\th_2,\th_3)$ exists if and only if 
\begin{equation}\label{holconstr}
d_1((\th_1,\th_2,\th_3),\mathbb Z_{e}^3)>1.
\end{equation} 
Moreover such a triangle is unique when it exists. 
\end{theorem}

The unique triangle with three non-integral angles $\pi\cdot (\th_1,\th_2,\th_3)$ will be denoted by $\Delta(\th_1,\th_2,\th_3)$. 

\begin{remark}\label{tetrarem} Let us decipher Inequality (\ref{holconstr}). Note first, that the subset $d_1((\th_1,\th_2,\th_3),\mathbb Z_{e}^3)\le 1\subset \mathbb R^3$ is a union of octahedra of diameter $2$ centred at points of $\mathbb Z_{e}^3$. The complement to this set is a disjoint union of open tetrahedra. Each such tetrahedron is contained in a unit cube with integer vertices. This collection of tetrahedra is invariant under translations of $\mathbb R^3$ by elements of $\mathbb Z_{e}^3$. Theorem \ref{nonintriang} states that if a point $(\th_1,\th_2,\th_3)\in \mathbb R_{>0}^3$ lies in one of such tetrahedra, the corresponding spherical triangle exists and it is unique. Figure~\ref{fig:tetrahedrons} depicts the union of six such tetrahedra in the octant $\mathbb R_{>0}^3$.
\end{remark}

An explicit construction of balanced spherical triangles can be found in \cite[Section 3.1.2]{mondello-panov:constraints}. In fact, it was already used by Klein \cite{Kl}.

\subsubsection{Triangles with one integral angle}

The second result we wish to recall from \cite{eremenko:three} is the following.

\begin{theorem}[Triangles with one integral angle \cite{eremenko:three}]\label{intriang}  
If $\th_1$ is an integer and $\th_2,\,\th_3$ are not integers, then a spherical triangle with angles $\pi\cdot(\th_1,\th_2,\th_3)$ exists if and only if at least one of the following conditions is satisfied.

%is that either $|\th_2-\th_3|$  or $\th_2+\th_3$ is an integer $l$ of opposite parity from $\th_1$, and either
\begin{itemize}

\item[(a)] $|\th_2-\th_3|$ is an integer $n$ of opposite parity from $\th_1$ and  $n=|\th_2-\th_3|\le \th_1-1$.

\item[(b)] $\th_2+\th_3$ is an integer $n$ of opposite parity from $\th_1$ and
$n=\th_2+\th_3\le \th_1-1$.
\end{itemize}
 
Moreover, when $\th_i$ satisfy (a) or (b), there is a one-parameter family of triangles with angles $\pi\cdot(\th_1,\th_2,\th_3)$ and this family is parametrised by the length $|x_1x_2|$ (or $|x_1x_3|$).
\end{theorem}

%Note that for fixed $(\th_1, \th_3,\th_3)$ as in Theorem \ref{intriang} there is a whole one parameter family of triangles,  see Corollary \ref{oneintconst} 3). 

It is obvious that triangles satisfying the hypotheses of Theorem \ref{intriang} (b) are never balanced.

\begin{remark}\label{oneintrem}
It is easy to see that in the case when a triple $(\th_1,\th_2,\th_3)$ of positive numbers satisfies
the triangle inequality and the integrality constraints of Theorem \ref{intriang} (a), there are  integers $n_1,n_2,n_3\geq 0$ and a number $\theta\in (0,1)$ such that $\th_1=n_2+n_3+1$, $\th_2=n_1+n_3+\theta$, $\th_3=n_1+n_2+\theta$.
\end{remark}

Finally, we present a full description of balanced triangles with exactly one integral angle.  

\begin{proposition}[Balanced triangles with one integral angle]\label{oneintconst}
Let $\Delta$ be a balanced spherical triangle with vertices $x_1,\,x_2,\,x_3$ and angles $\pi\cdot (\th_1,\th_2,\th_3)$, where $\th_1$ is an integer while $\th_2,\,\th_3$ are not integers. Let $n_1,\,n_2,\,n_3,\,\theta$ be as in Remark \ref{oneintrem}. Then the following holds. 
\begin{itemize} 

\item[(i)]  $|x_2x_3|=\pi$.
\item[(ii)] There exists a unique pair of geodesic segments $\gamma_{12}, \gamma_{13}\subset \Delta$ with $|\gamma_{12}|+|\gamma_{13}|=\pi$, that cut $\Delta$ into the following three domains. The first is a digon with angles $\pi n_3$ bounded by the sides $x_1x_2$ and $\gamma_{13}$. The second is a  digon with angles $\pi n_2$ bounded by the sides $x_1x_3$ and $\gamma_{13}$. The third is a triangle with sides $\gamma_{12}$, 
$\gamma_{13}$ and $x_2x_3$, and angles $\pi(\theta+n_1,\theta+n_1,1)$ opposite to the sides.
\item[(iii)] All balanced triangles with angles $\pi(\th_1,\th_2,\th_3)$ are parametrised by the interval $(0,\pi)$ where one can choose as a parameter either $|x_1x_2|$
or $2\pi-|x_1x_2|$, depending on whether $n_3$ is even or odd.
\end{itemize}
\end{proposition}
\begin{proof}
(i) Since $\Delta$ is balanced, by Corollary \ref{lessthan2pi} we have $|x_1x_2|, |x_2x_3|, |x_3x_1|<2\pi$. Consider the developing map $\iota: \Delta\to \mathbb S^2$. Since $\th_1$ is integer, the images $\iota(x_1x_2), \, \iota(x_1x_3)$ belong to one great circle  $C$ in $\mathbb S^2$. At the same time, since the angle $\th_2$ is non-integer, the image $\iota(x_2x_3)$ does not belong to $C$. This means that $\iota(x_2)$ and $\iota(x_3)$ are opposite points on $\mathbb S^2$ and so $|x_2x_3|=\pi$.

(ii) Since $|x_2 x_3|=\pi$ by part (i), there exists a maximal digon embedded in $\Delta$, with one edge equal to $x_2 x_3$.
The other edge of such digon must pass through $x_1$ by maximality, and so it is the concatenation of two geodesics
$\gamma_{12}$ from $x_1$ to $x_2$ and $\gamma_{13}$ from $x_1$ to $x_3$, that form an angle $\pi$ at $x_1$.
It is easy to see that these are the geodesics we are looking for. 
The uniqueness of $\gamma_{12},\gamma_{13}$ follows, because $n_1$ and $\theta$ are uniquely determined.
%
%Consider the unique digon $B_{1-\th}$ with angles $\pi(1-\th)$ and let us construct a new spherical triangle $\Delta'$  by gluing one side of $B_{1-\th}$ to the side $x_2x_3$ of $\Delta$. Clearly, $\Delta'$ is a spherical triangle with integral angles and so we can apply to it Lemma \ref{nonintside}. This lemma gives us a unique pair of geodesics $\gamma_{12},\, \gamma_{13}$ that clearly lie inside $\Delta$.

(iii) follows from part (ii).
\end{proof}

The next lemma is a partial converse to Proposition \ref{oneintconst} (i).

\begin{lemma}[Balanced triangles with one edge of length $\pi$]\label{sidepi} Let $\Delta$ be a balanced spherical triangle with vertices $x_1,\,x_2,\,x_3$ and angles $\pi(\th_1,\th_2,\th_3)$. Suppose that $|x_2x_3|=\pi$. Then $\th_1$ is integer. 
\end{lemma}
\begin{proof} Consider the developing map $\iota: \Delta\to \mathbb S^2$. Since $|x_ix_j|<2\pi$ by Corollary \ref{lessthan2pi}, we see that $\iota(x_i)\ne \iota(x_j)$ for $i\ne j$. In order to show that $\th_1$ is integer it is enough to prove that both images $\iota(x_1x_2)$ and $\iota(x_1x_3)$ lie on the same great circle. But this is clear, since the points $\iota(x_2)$ and $\iota(x_3)$ are opposite on $\mathbb S^2$, while $\iota(x_1)$ is different from both points.  \end{proof}

Last lemma concerns semi-balanced triangles.

\begin{lemma}[Semibalanced triangles with one integral angle]\label{seminteger} 
Suppose $\Delta$ is a semi-balanced triangle with angles $\pi(\th_1,\th_2,\th_3)$.
\begin{itemize}
\item[(i)] If $\th_i$ is an integer, then $\th_1+\th_2+\th_3$ is an even integer $2m$ and $\th_j,\th_k$ are half-integers. 
\item[(ii)] If $\th_1+\th_2+\th_3=2m$, then one $\th_i$ is integer and the other two $\th_j,\th_k$ are half integer. 
 \end{itemize} 
\end{lemma}
\begin{proof}
Without loss of generality, we can assume that $\th_1=\th_2+\th_3$.
So certainly $\th_1+\th_2+\th_3$ cannot be odd integral. It follows from
\cite[Theorem 2]{eremenko:three} that $\th_1,\th_2,\th_3$ cannot be three integers.

(i) Note that $\th_2$ cannot be an integer, because the relation $\th_1-\th_3=\th_2$ would violate Theorem \ref{intriang} (a).
Similarly, $\th_3$ cannot be an integer. Hence, $\th_1$ is an integer and so  Theorem \ref{intriang} (a) implies that
$\th_2,\th_3$ are half-integers.

(ii) Our hypotheses imply that $\th_1=m$ is an integer. By (i) we obtain that $\th_2,\th_3$ are half-integers.
%
%
%This lemma is an exercise on application of inequalities from Theorem \ref{nonintriang}, Theorem \ref{intriang}, and Corollary \ref{tri2npi}, and we omit the proof.
%
\end{proof}

\subsubsection{Triangles with three integral angles}

We begin by giving a description of all triangles with integral angles.

\begin{proposition}[Triangles with three integral angles]\label{nonintside} 
For any spherical triangle $\Delta$ with integral angles $\pi\cdot (m_1,m_2,m_3)$ the following holds. 
\begin{itemize} 
\item[(i)]
There exists a unique triple $(n_1,n_2,n_3)$ of non-negative integers such that $m_1=n_2+n_3+1$, $m_2=n_3+n_1+1$, $m_3=n_1+n_2+1$. Moreover, there exist a unique triple of geodesic segments $\gamma_{12},\gamma_{23},\gamma_{13}\subset \Delta$ with $|\gamma_{12}|+|\gamma_{23}|+|\gamma_{13}|=2\pi$, that join points $x_i$ and cut $\Delta$ into the following four domains:
\begin{itemize}
\item
the \emph{central} disk $\Delta_0$ isometric to a half-sphere and bounded by segments $\gamma_{12},\,\gamma_{23},\,\gamma_{13}$;
\item
digons $B_1$, $B_2$, $B_3$ where each $B_i$ is bounded by segments $\gamma_{jk}$ and $x_jx_k$ and has angles $\pi n_i$.
\end{itemize}
\item[(ii)] 
The space of triangles with angles $\pi\cdot (n_1,n_2,n_3)$ can be identified with the set of triples of positive numbers $l_{12}, l_{13}, l_{23}$ satisfying $l_{12}+l_{23}+l_{13}=2\pi$ (where $l_{ij}$ are interpreted as  the lengths of the sides of $\Delta_0$).
\item[(iii)] 
All sides of $\Delta$ are shorter than $2\pi$. Moreover, there is at most one side of length $\pi$.
\end{itemize}
\end{proposition}

\begin{proof}
(i) Consider the developing map: $\iota: \Delta\to \mathbb S^2$. Since all the angles of  $\Delta$ are integral, all its sides are sent to one great circle on $\mathbb S^2$. The full preimages of this circle cuts $\Delta$ into a collection of hemispheres. It is easy to see that only one of these hemisphere contains all three conical points, this is the disk $\Delta_0$ in $\Delta$. The conical points cut the boundary of the disk into three geodesic segments $\gamma_{12},\gamma_{23}, \gamma_{13}$.
The complement to $\Delta_0$ in $\Delta$ is the union of the three digons $B_1,\,B_2,\,B_3$.

(ii) It is clear from (i) that $\Delta$ is uniquely defined by the three lengths $l_{ij}=|\gamma_{ij}|$ and $n_1,\,n_2,\,n_3$. Conversely, for each positive triple $l_{ij}$ with $l_{12}+l_{23}+l_{13}=2\pi$, and each integer triple $n_1,\,n_2,\,n_3$, one constructs a unique spherical triangle. 

(iii) Since $|\gamma_{12}|+|\gamma_{23}|+|\gamma_{31}|=2\pi$, then all $\gamma_{ij}$ are shorter than $2\pi$.
If $n_k=0$, then $x_ix_j=\gamma_{ij}$. If $n_k>0$, then $x_ix_j$ bounds a digon $B_k$ with angles $\pi n_k$. In both cases,
$x_ix_j$ has length $|\gamma_{ij}|$ (if $n_k$ is even) or $2\pi-|\gamma_{ij}|$ (if $n_k$ is odd).
Thus, $|x_i x_j|<2\pi$.
Moreover, suppose that one of the sides $x_ix_j$, say $x_2 x_3$, has length $\pi$.
It follows that $|\gamma_{23}|=\pi$ and so $|\gamma_{12}|,\,|\gamma_{13}|<\pi$.
As a consequence, $x_1x_2$ and $x_1x_3$ have length different from $\pi$.
\end{proof}

\begin{remark}[Existence of balanced triangles with integral angles]\label{bal-3int}
If $(m_1,m_2,m_3)$ is a triple of positive integers that satisfies the triangle inequality,
then there exist $n_1,n_2,n_3\geq 0$ integers such that $m_i=1+n_j+n_k$ for $\{i,j,k\}=\{1,2,3\}$.
Then the construction described in Proposition \ref{nonintside} (i) shows that there exists a balanced spherical triangle
with angles $\pi(m_1,m_2,m_3)$.
\end{remark}

We thus obtain a characterisation of such triangles (see also \cite{eremenko:three} and \cite{EGnew}).

\begin{corollary}[Balanced triangles of area $2m\pi$]\label{tri2npi}
Let $\Delta$ be a triangle.
\begin{itemize}
\item[(i)]
If $\Delta$ has integral angles $\pi\cdot (m_1,m_2,m_3)$, then
$\Delta$ is strictly balanced and it has area $2m\pi$ with $m=\frac{1}{2}(m_1+m_2+m_3-1)\in\ZZ$.
\item[(ii)]
If $\Delta$ has area $2m\pi$ for some integer $m>0$
and it is balanced, then $\Delta$ has integral angles $\pi\cdot (m_1,m_2,m_3)$,
with $m_1+m_2+m_3=2m+1$. 
\end{itemize}
\end{corollary}
\begin{proof}
(i) By Proposition \ref{nonintside}, the central disk $\Delta_0$ has angles $\pi(1,1,1)$ and so it is strictly balanced. Since $\Delta$ is obtained from $\Delta_0$ by gluing digons along its edges, $\Delta$ is strictly balanced.
The second claim is a consequence of \cite[Theorem 2]{eremenko:three}.

(ii) Suppose that $\Delta$ has angles $\pi(\th_1,\th_2,\th_3)$.
Since $\area(\Delta)=\pi(\th_1+\th_2+\th_3-1)$ we see that $\th_1+\th_2+\th_3=2m+1$.   It follows easily that $d_1((\th_1,\th_2,\th_3),\mathbb Z_{e}^3)=1$. Hence, from Theorem \ref{nonintriang} we conclude that at least one of the $\th_i$, say $\th_1$, is integer. 

Assume by contradiction that $\th_2$ and $\th_3$ are not integer, and so we are in the setting of Theorem \ref{intriang}. The possibility (b) can't hold because $\Delta$ is balanced. Assume that possibility (a) holds, in which case $\th_2-\th_3$ is an integer, and $\th_1+\th_2-\th_3$ is odd. But then, since $\th_1+\th_2+\th_3$ is also odd, we see that $\th_3$ is integer. This is a contradiction.

%Suppose, however, that $\th_2$ and $\th_3$ are not integer, and so we are in the setting of Theorem \ref{intriang}. Since $\th_2-\th_3$ can not be integer in this case, condition \ref{intriang} (b)  should be satisfies. But this is impossible because $\Delta$ is balanced. 

We conclude that all $\th_i$'s are integer. 
%Clearly, $(\th_1,\th_2,\th_3)$ strictly satisfies the triangle inequality, since $\Delta$ is balanced and the sum of $\th_i$ is odd.  
%Now  we have two cases, either $\th_2+\th_3$ is integer or $|\th_2-\th_3|$ is integer. The  first possibility is excluded, since by conditions of Theorem \ref{nonintriang}  one should have $\th_2+\th_3\le \th_1-1$ which contradicts the fact that $\Delta$ is balanced. Hence we are in the second case and $\th_2-\th_3$ is integer. It follows that both $\th_2$ and $\th_3$ are integer. Clearly, $(\th_1,\th_2,\th_3)$ satisfy strict triangle inequality since $\Delta$ is balanced and the sum of $\th_i$ is odd.  
\end{proof}

%3) We will present two proofs. First, consider the central disk $\Delta_0\subset \Delta$. Its boundary contains conical points $x_1,x_2,x_3$ which cut the boundary into three geodesic segments $\gamma_{12}, \gamma_{23}, \gamma_{31}$ . Since $|\gamma_{12}|+|\gamma_{23}|+|\gamma_{31}|=2\pi$, at most one of segments, say $\gamma_{23}$ can have length $\pi$. Note next that in $\Delta$ we have $|x_1x_2|=\pi+(-1)^{m_3}(|\gamma_{12}|-\pi)$. Hence the side $x_1x_2$ has non-integer length and the same holds for $x_1x_3$.

%Alternatively, one can take the sphere $S(\Delta)$ obtained by doubling of $\Delta$ and consider its developing map $\iota$ to $\mathbb S^2$. It is easy to see that three points $x_1,\, x_2,\,x_3$ will be sent to different points in $\mathbb S^2$. Now the statement can be deduced from the fact that for three distinct points on $\mathbb S^2$ at most one distance among them can be integer.  

%\begin{remark} We note that any triangle $\Delta$ with three integer angles must have some non-integer sides. 
%\end{remark}

%%%%%%%%%%%%%%%%%%%%%%%%%%%%%%

\subsubsection{Final considerations}

The last statement of the section can be derived in many ways.
Here we obtain it as a consequence of Theorem \ref{nonintriang}, Theorem \ref{intriang} and Proposition \ref{nonintside}.

\begin{corollary}[Triangles are determined the side lengths and angles]\label{uniquethl} 
Let $\Delta$ be a spherical triangle with angles $\pi\cdot(\th_1,\th_2,\th_3)$, and let $l_i$ be the length of the side opposite to the vertex $x_i$. Then $\Delta$ is uniquely determined by $\th_i$'s and $l_i$'s.
\end{corollary}
\begin{proof} If none of $\th_i$ is integer, then $\Delta$ is uniquely determined by $(\th_1,\th_2,\th_3)$ by Theorem \ref{nonintriang}.

If $\th_1$ is integer, while $\th_2$ and $\th_3$ are not integer, then the triangle $\Delta$ is uniquely determined by the angles $\th_i$ and the length $l_3$ by Theorem \ref{intriang}.

If $\th_1,\,\th_2\,\,\th_3$ are integer, then it follows from Proposition \ref{nonintside} that  all triangles with angles $\th_i$ are uniquely determined by the lengths of their sides.
\end{proof}

\subsection{The space of spherical triangles and its coordinates}\label{subsecsmooth}

%Let us denote by $\MP(n)$ the space of all spherical polygons with $n$ cyclically labelled vertices up to isometries that preserve the labelling. This space has a natural topology induced by Lipschitz distance. Since we are mainly interested in spherical triangles, we will use notation $\MT$ to denote $\MP(3)$. 
Let us denote by $\MT$ be space of all (unoriented) spherical triangles with vertices labelled by $x_1,x_2,x_3$, up to isometries that preserve the labelling. This space has a natural topology induced by the Lipschitz distance (see Section \ref{secLipschitz}).
We will denote by $\th_1,\,\th_2,\th_3,\, l_1,\,l_2,\,l_3$ the functions on $\MT$, defined by requiring that $\pi\th_i(\Delta)$ is the angle of the spherical triangle $\Delta$ at $x_i$ and $l_i(\Delta)$ is the length of the side of $\Delta$ opposite to $x_i$.

By Corollary \ref{uniquethl} the map $\Psi:\MT\to \mathbb R^6$, that associates to each triangle its angles and side lengths, is 
one-to-one onto its image. Moreover, we have the following result.

\begin{theorem}(Space of spherical triangles \cite[Theorem 1.2]{EGnew})\label{smoothKlein} Let $\MT$ be the space of spherical triangles. The image $\Psi(\MT)\subset \mathbb R^6$ is a smooth, connected, orientable real analytic $3$-dimensional submanifold of $\mathbb R^6$. 
\end{theorem}

This theorem says that the space $\MT$ has a structure of a smooth, connected, analytic manifold and moreover at each point $\Delta\in \MT$ one can choose three functions among $\th_i$ and $l_i$ as local analytic coordinates. It also follows from Theorem \ref{smoothKlein} that formulas of spherical trigonometry, that are usually stated for convex spherical triangles, hold for all spherical triangles. In particular, for any permutation $(i,j,k)$ of $(1,2,3)$ and any $\Delta\in \MT$ the cosine formula for lengths holds\footnote{Indeed, an analytic function vanishing on an open subset of an irreducible analytic variety vanishes identically.}:
\begin{equation}\label{cosformula}
\cos l_i\sin(\pi \th_j) \sin (\pi\th_k)=\cos(\pi\th_i)+\cos(\pi\th_j)  \cos(\pi \th_k).
\end{equation}

\begin{lemma}[Some coordinates on the space $\MT$]\label{coordinates}
Consider the functions $\th_1,\th_2,\th_3$ on $\MT$. 
%Let $\Delta\in \MT$ be a spherical triangle with angles $\pi(\th_1,\th_2,\th_3)$.
\begin{itemize}
\item[(i)] The functions $\th_1,\th_2,\th_3$ form global analytic coordinates on the (open dense) subset of $\MT$ consisting of triangles with non-integral angles.
\item[(ii)] Suppose $\Delta\in\MT$ is short-sided and the angle sum $\th_1(\Delta)+\th_2(\Delta)+\th_3(\Delta)$ is not an odd integer. Then the function $\th_1+\th_2+\th_3$ has non-zero differential at $\Delta$. 
\end{itemize}
\end{lemma}

\begin{proof} (i) Consider the projection map from $\Psi(\MT)$ to the angle space $\mathbb R^3$. According to Theorem \ref{nonintriang} this map is one-to-one over the subset of $(\th_1,\th_2,\th_3)$ in $\mathbb R_{>0}^3$, that satisfy Inequality  
(\ref{holconstr}). We need to show that this projection is in fact a diffeomorphism over this set. However, using the cosine formula (\ref{cosformula}) and the fact that none of $\th_i$ is integer, we see that the lengths $l_i$ depend analytically on the $\th_i$'s. 

(ii) In case $\th_i$ are all not integer the statement follows immediately from (i). Suppose that one of $\th_i$, say $\th_1$, is integer. Then, since $\Delta$ is short-sided, using exactly the same reasoning as in the proof of Proposition \ref{oneintconst} (i), we deduce that $l_i=\pi$. Now, for any $\theta>0$ we can glue to the side $x_2x_3$ of $\Delta$ the digon with two sides of length $\pi$ and the angles $\pi\theta$. The family of triangles thus constructed, that depends on $\theta$, determines a straight segment in $\Psi(\MT) $ starting from  $\Psi(\Delta)$ and the linear function $\th_1+\th_2+\th_3$ restricted to this segment has non-zero derivative.  
\end{proof}

%\begin{theorem}[The space of spherical triangles] The space $\mathcal X$ of all spherical triangles is a smooth, connected, orientable real analytic $3$-manifold. The angles and side lengths of triangles, which define the embedding of $\mathcal X$ to $\mathbb R^6$ are real analytic functions on $\mathcal X$.
%\end{theorem}

\begin{definition}[Spaces of triangles with assigned area]
For any $\th>1$ we denote by $\MT(\th)\subset \MT$ the surface consisting of triangles with $\th_1+\th_2+\th_3=\th$.
We denote by $\MT_{bal}(\th)$ and $\MT_{sh}(\th)$ the subsets of balanced and short-sided triangles correspondingly.
\end{definition}

The following statement is a corollary of Theorem \ref{smoothKlein} and Lemma \ref{coordinates}.

\begin{corollary}[Space of balanced triangles with assigned area]\label{smootharea}For any $\th>1$ the set $\MT_{bal}(\th)$  is a non-singular, real analytic, orientable bordered submanifold of the manifold $\MT$ of all spherical triangles. The boundary of $\MT_{bal}(\th)$ consists
of semi-balanced triangles.
\end{corollary}
\begin{proof} Suppose first $\th_1+\th_2+\th_3=2m+1$.
Balanced spherical triangles of area $2m\pi$ are classified in Lemma \ref{tri2npi} and Proposition \ref{nonintside}. They have integral angles and each connected component forms an open Euclidean triangle in $\mathbb R^6$. Clearly such a subset of $\mathcal \MT\subset \mathbb R^6$ is a smooth submanifold.

Assume now that $\th=\th_1+\th_2+\th_3$ is not an odd integer. Clearly, $\MT_{sh}$ is an open subset of $\MT$, and so we deduce from Lemma \ref{coordinates} (ii) that $\MT_{sh}(\th)$ is an open smooth $2$-dimensional submanifold of $\MT$. 
The set $\MT_{bal}(\th)$ is contained in $\MT_{sh}(\th)$ and its boundary is composed of semi-balanced triangles. We need to show that such triangles form a smooth curve in $\MT_{sh}(\th)$. 

Let $\Delta\in \MT_{sh}(\th)$ be a semi-balanced triangle, say $\th_1=\th_2+\th_3$. If $\th_1, \,\th_2,\,\th_3$ are not integer, from Lemma \ref{coordinates} (i) it follows immediately that the curve $\th_1-\th_2-\th_3=0$ is smooth in a neighbourhood of $\Delta$. Suppose that one of $\th_i$ is integer. Then we are in the setting of Lemma \ref{seminteger}. In particular by Lemma \ref{seminteger} (i) we have  $\th_1+\th_2+\th_3=2m$. But then, applying Lemma \ref{seminteger} (ii) we see that all semi-balanced triangles in $\MT_{bal}(2m)$ have one integral and two half-integral angles. Such triangles are governed by Proposition \ref{oneintconst} and their image under the map $\Psi$ forms a collection of straight segments in $\mathbb R^6$. It follows that semi-balanced triangles form  a smooth curve in $\MT_{sh}(2m)$.

Finally, let's show that $\MT_{bal}(\th)$ is orientable. This is clear if $\th$ is an odd integer, because a disjoint union of open triangles is orientable. In case $\th$ is not an odd integer, it suffices to show that $\MT_{bal}(\th)$ can be co-oriented, since $\MT$ is orientable. A co-orientation can indeed be chosen since the function $\th_1+\th_2+\th_3=\th$ has non zero differential along the surface $\MT_{bal}(\th)$ by Lemma \ref{coordinates} (ii).
\end{proof}

\subsection{Balanced spherical triangles of fixed area}\label{subsectbalancefix}

The goal of this section is to describe the topology of the moduli space  $\MT_{bal}(\th)$ of balanced triangles with marked vertices of fixed area $\pi(\th-1)$, where $\th>1$. To better visualize the structure of such space, we introduce the following object.

\begin{definition}[Angle carpet]
Take $\th>1$ such that $\th\notin 2\mathbb Z+1$. The {\it angle  carpet}, denoted $\mathrm{Crp}(\th)$ is the subset of the plane
$\Pi(\th):=\{(\th_1,\th_2,\th_3)\in\RR^3_{>0}\,|\,\th_1+\th_2+\th_3=\th\}$ 
consisting of points  such that there exists a spherical triangle  with angles $\pi(\th_1,\th_2,\th_3)$. 
Points in $\mathrm{Crp}(\th)$ with one integral coordinate are called {\it{nodes}}.
The {\it balanced angle carpet} is the subset  $\mathrm{Crp}_{bal}(\th):=\mathrm{Crp}(\th)\cap \B(\th)$,
where $\B(\th)=\{(\th_1,\th_2,\th_3)\,|\,\th_i\leq\th_j+\th_k\}$. A node in $\mathrm{Crp}_{bal}(\th)$ is {\it{internal}} if it does not lie on $\partial\B(\th)$.
\end{definition}

Now we separately treat the cases $\th$ not odd and $\th$ odd.

\subsubsection{Case $\th$ not odd}

Throughout the section, assume $\th\notin 2\mathbb Z+1$. We will denote by $\MT_{bal}^{\mathbb Z}(\th)$ its subset consisting of triangles with at least one integral angle. By Proposition \ref{oneintconst} this subset is a disjoint union of smooth open  intervals in $\MT_{bal}(\th)$. We will see that it  cuts $\MT_{bal}(\th)$ in a union of topological disks.  
This decomposition is very well reflected in the structure of the associated balanced carpet, as we will see below.

%To visualise this decomposition of $\MT_{bal}(\th)$ we introduce the following object.

\begin{figure}[!ht]
\vspace{-0.3cm}
\begin{center}
\includegraphics[scale=0.35]{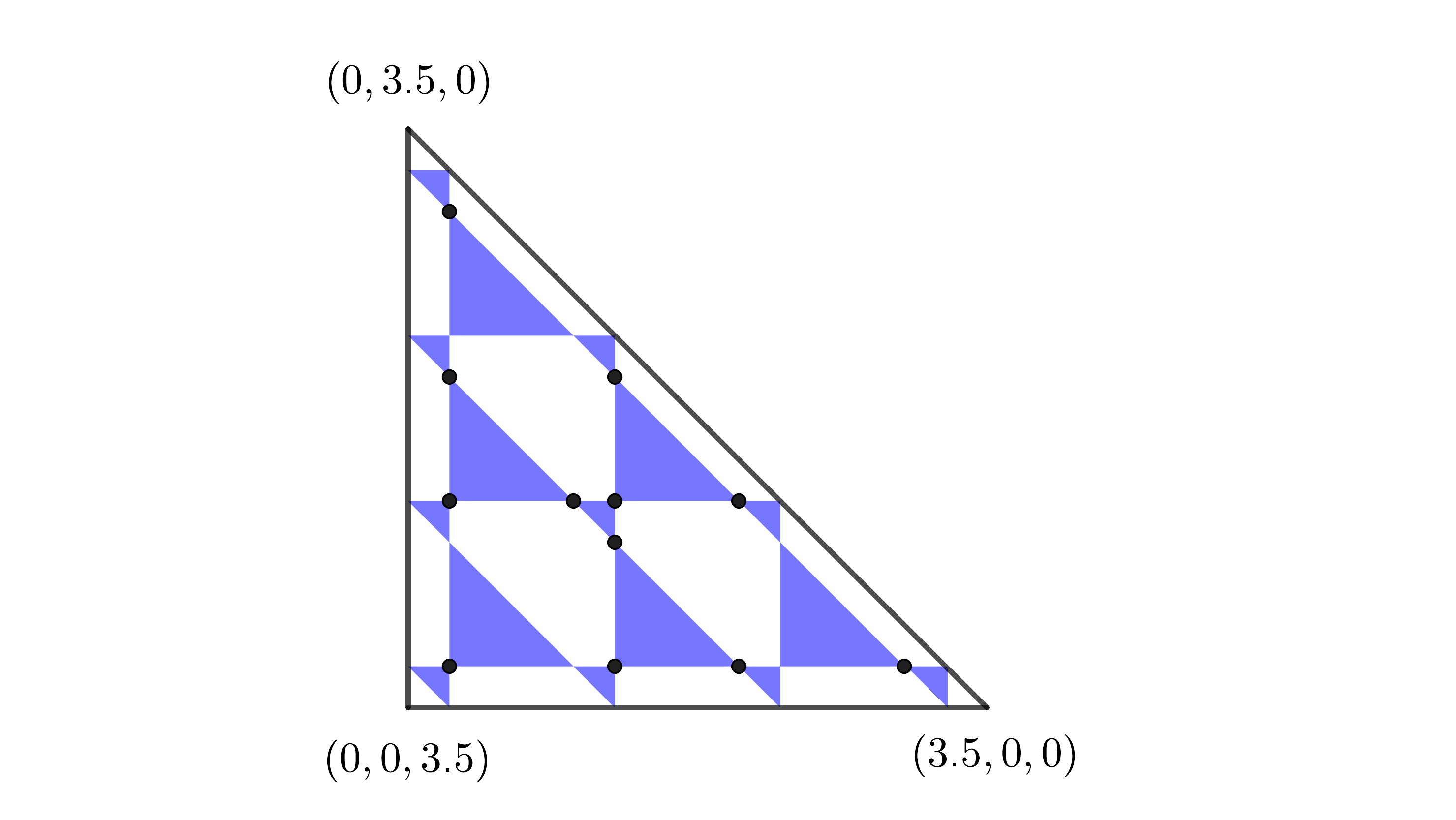}
\end{center}
\vspace{-0.5cm}
\caption{The angle carpet $\mathrm{Crp}(\frac{7}{2})$, composed of $16$ open triangles and $12$ nodes.}
\label{fig:carpet3.5}
\end{figure}

%\begin{definition}[Angle carpet]
%Take $\th>1$ such that $\th\notin 2\mathbb Z+1$. The {\it angle  carpet}, denoted $\mathrm{Crp}(\th)$ is the subset of the plane $\th_1+\th_2+\th_3=\th$ consisting of points  such that there exists a spherical triangle  with angles $\pi(\th_1,\th_2,\th_3)$. 
%Points in $\mathrm{Crp}(\th)$ with one integral coordinate are called {\it{nodes}}.
%The {\it balanced angle carpet} is the subset  $\mathrm{Crp}_{bal}(\th):=\mathrm{Crp}(\th)\cap \B(\th)$,
%where $\B(\th)=\{(\th_1,\th_2,\th_3)\,|\,\th_i\leq\th_j+\th_k\}$.
%\end{definition}

The carpet $\mathrm{Crp}(\th)$
is composed of a disjoint union of open triangles with a subset of their vertices (the nodes).
In order to better visualize such carpets,
we will often identify $\mathrm{Crp}(\th)$ with its projection to the horizontal $(\th_1,\th_2)$-plane.
Figure~\ref{fig:carpet3.5} shows how the projection of $\mathrm{Crp}(3.5)$ looks like:
it is a union of $16$ disjoint open triangles 
(singled out by Inequality (\ref{holconstr}) of Theorem \ref{nonintriang}) and
a subset of $12$ of nodes (governed by condition (a) of Theorem \ref{intriang})
marked as black dots.
Figure~\ref{fig:carpets} depicts the projection of balanced angle carpets for five different values of $\th$.

%We will denote by $\th_{12}$ the surjective map $\th_{12}: \MT(\th)\to \mathrm{Crp}(\th)$ that associates to each triangle $\Delta$ the point $(\th_1(\Delta),\th_2(\Delta))$.

%
%\begin{remark} In Figure~\ref{fig:carpet3.5} the angle carpet $\rm{Crp}(3.5)$ is depicted. It is a union of $16$ disjoint open triangles with a subset of $12$ of  nodes marked as black dots. The open triangles are singled out by Inequality (\ref{holconstr}) of Theorem \ref{nonintriang}. The nodes are governed by condition (a) of Theorem \ref{intriang}. 
%\end{remark}

%
%In Figure~\ref{fig:carpets} one can see the balanced carpets $\mathrm{Crp}_{bal}(\th)$ for parameters $\th=1.5,\,2,\,3.5,\,6,\,8$.

\begin{figure}[!ht]
\vspace{-0.3cm}
\begin{center}
\includegraphics[scale=0.5]{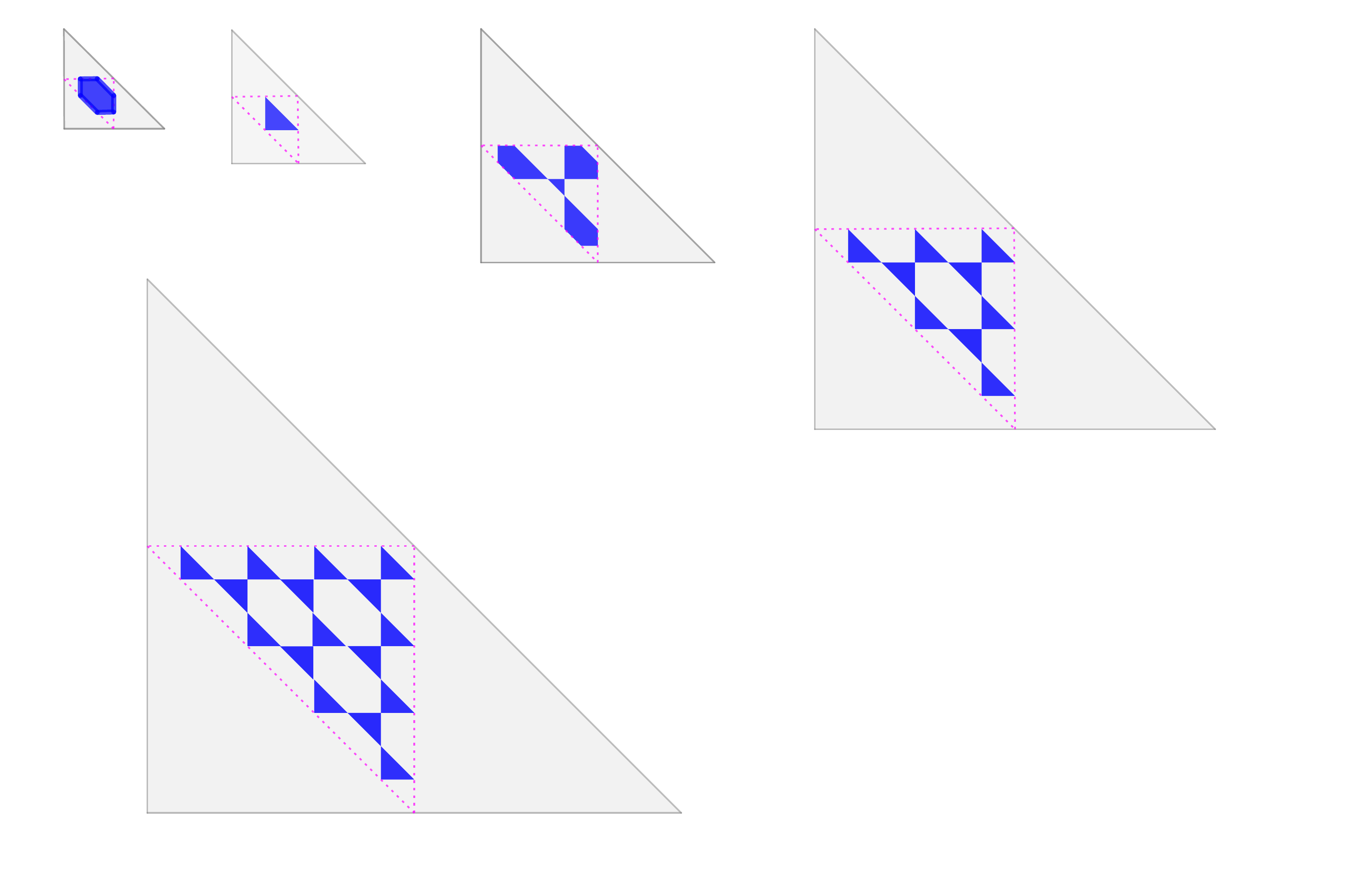}
\end{center}
\vspace{-0.5cm}
\caption{Balanced carpets for $\th=1.5,\ 2,\ 3.5,\ 6,\ 8$}
\label{fig:carpets}
\end{figure}

The following lemma  is a consequence of Theorems \ref{nonintriang} and \ref{intriang}.

\begin{lemma}[Description of the angle carpets]\label{connectedcarpet} Take $\th\in (1,\infty)\setminus \{2\mathbb Z+1\}$ and set $m=\floor{\frac{\th+1}{2}}$. 
\begin{itemize}
\item[(i)]
The carpet $\mathrm{Crp}(\th)$ is the union of  $4m^2$ open triangles with $3m^2$ nodes $(\th_1,\th_2,\th_3)$, such that the unique integer coordinate $\th_i$ of a node satisfies the inequality $\th_i\ge |\th_j-\th_k|+2l+1$ for some integer $l\leq 0$.  
\item[(ii)]
All points $(\th_1,\th_2,\th_3)\in \B(\th)$ with one positive integral coordinate
are nodes in $\mathrm{Crp}_{bal}(\th)$. Hence, the balanced carpet $\mathrm{Crp}_{bal}(\th)$ is a connected set.
\item[(iii)]
The balanced carpet $\mathrm{Crp}_{bal}(\th)$ intersects $E$ open triangles and
it contains $N$ internal nodes, where
\[
E=\begin{cases}
m^2 & \text{if $\th\leq 2m$}\\
m^2+3m & \text{if $\th>2m$}
\end{cases}
\qquad\qquad\qquad
N=\begin{cases}
3m(m-1)/2 & \text{if $\th\leq 2m$}\\
3m(m+1)/2 & \text{if $\th>2m$.}
\end{cases}
%
%\begin{array}{c|c|c|c}
%& \th<2m & \th=2m & \th>2m\\
%\hline
%E & m^2 & m^2 & m^2+3m\\
%N & \frac{3}{2}m(m-1) & \frac{3}{2}m(m-1) & \frac{3}{2}m(m+1)
%\end{array}
\]
Hence $E-N=-m(m-3)/2$.
\item[(iv)]
There exists a point in $\mathrm{Crp}_{bal}(\th)$ with non-integral
coordinates at which $\th_2=\th_3$.
\end{itemize}
\end{lemma}
\begin{proof}
(i) Let us split the carpet into two subsets. The first subset consists of points such that none of coordinates $\th_i$ is integer, and the second subset is where one of coordinates $\th_i$ is integer. 

Is is clear that the first subset is the union of open triangles given by  intersecting the plane $\th_1+\th_2+\th_3=\th$ with the open tetrahedra that are given by Inequality (\ref{holconstr}) of Theorem  \ref{nonintriang}. Since such plane does not pass through any vertex of the tetrahedra for $\th$ non-odd, it follows that the number of triangle only depends on $m$ and so we can compute it for $\th=2m$.
Look at the projection of $\mathrm{Crp}(2m)$ inside the $(\th_1,\th_2)$-plane and enumerate the open triangles as follows:
to points of type $(0,l+1/2)$ with $l\in\{0,1,\dots,2m-1\}$ we can associate a unique triangle, to points of type 
$(n,l+1/2)$ with $n\in\{1,\dots,2m-1\}$ and $l\in\{0,\dots,2m-n-1\}$ we can associate two triangles.
The number of such triangles is thus $4m^2$. 

The second subset is governed by Theorem  \ref{intriang}. Since $\th_1+\th_2+\th_3=\th$ is not an odd integer, only the nodes that satisfy condition (a) of Theorem \ref{intriang} lie in $\mathrm{Crp}(\th)$. Again it's enough to count the nodes for $\th=2m$. 
Suppose first $\th_1$ integer. We must have $|2\th_2+\th_1-2m|=|\th_2-\th_3|=\th_1-1-2l$ for some integer $l$.
If $\th_1\in\{1,2,\dots,m\}$, then $\th_2\in  \frac{1}{2}+\{m-\th_1,\dots,m-1\}$ and so we have $m(m+1)/2$ nodes.
If $\th_1\in\{m+1,\dots,2m-1\}$, then $\th_2\in \frac{1}{2}+\{0,\dots,2m-1-\th_1\}$ and so we have $m(m-1)/2$ nodes.
Thus, we have $m^2$ nodes with integral $\th_1$, and we conclude that we have $3m^2$ nodes in total.

(ii) Again, it is enough to consider the case $\th=2m$. In the balanced carpet $\th_i\leq m$ for all $i$ and so the first claim follows from
the above enumeration of the nodes. Hence, $\mathrm{Crp}_{bal}(\th)$ is connected.

(iii) Let us first consider $N$. For $\th=2m$ the enumeration in part (i) shows that $N=3m(m-1)/2$. 
If $\th<2m$, such $N$ does not change. If $\th>2m$, then $N=3m(m-1)/2+3m$, such extra $3m$ is exactly the number of 
nodes sitting in $\partial\B(2m)$.

As for $E$, the enumeration in (i) for $\th=2m$ shows that $4E=4m^2$ and so $E=m^2$.
For $\th<2m$, the value of $E$ does not change. For $\th>2m$, there $3m$ extra triangles intersected by $\B(\th)$,
which is exactly the number of nodes sitting in $\partial\B(2m)$.
%
%, it suffice to notice that by Theorem \ref{intriang} for any open triangle of $\mathrm{Crp}(\th)$ that intersects the balanced set all its vertices belong to $\mathrm{Crp}(\th)$ as well. 

(iv) The point $\th_1=(c+3)/4$ and $\th_2=\th_3=m-3(1-c)/8$ belongs to the interior of $\mathrm{Crp}_{bal}(\th)$
and it is not a node.
\end{proof}

In order to understand the topology of $\MT_{bal}(\th)$ we consider
the natural projection map $\Theta:\MT_{bal}(\th)\rightarrow\mathrm{Crp}_{bal}(\th)$
that sends $\Delta$ to $(\th_1(\Delta),\th_2(\Delta),\th_3(\Delta))$. 
%We will say that a node is strictly balanced (resp.~semi-balanced) if it belongs to the image
%of a strictly balanced (resp.~semi-balanced) triangle via $\Theta$.

{\bf{Analysis of the map $\Theta$.}}
By Proposition \ref{connectedcarpet} the balanced carpet $\mathrm{Crp}_{bal}(\th)$
consists of $E$ polygons $\{P_l\}$, bounded by some {\it{semi-balanced edges}} that sit in $\pa \B(\th)$
and some nodes. Note that we are considering $P_l$ as closed subsets of $\mathrm{Crp}_{bal}(\th)$.
The closure of $P_l$ inside the plane $\Pi(\th)$ is obtained from $P_l$ by adding
some {\it{ideal edges}}, that have equations of type $\th_i=a+(c+1)/2$ with $i\in\{1,2,3\}$ and $a\in\{0,1,\dots,m-1\}$.
 
%Denote by $\ol{P}_l$ the closure of $P_l$ inside the plane $\mathrm{Crp}_{bal}(\th)$.

For each polygon $P_l$, the real blow-up $\widehat{P}_l$ of $P_l$ at its nodes is obtained from $P_l$
by replacing each node by an open interval ({\it{nodal edge}}): the natural projection $\widehat{P}_l\rar P_l$ contracts
each nodal edge to the corresponding node. 
For every $l$ we can fix a realization of $\widehat{P}_l$ inside $\RR^2$ as the union of an open convex polygon 
with some of its open edges (nodal edges and semi-balanced edges); the missing edges (ideal edges) correspond to the ideal edges of $P_l$.

We recall that $\MT_{bal}(\th)$ is a surface by Corollary \ref{smootharea} and its boundary consists of semi-balanced triangles,
and that the map $\Theta$ contracts each open interval in $\MT_{bal}^{\mathbb Z}(\th)$ to a node
by Proposition \ref{oneintconst} and it is a homeomorphism elsewhere by Lemma \ref{coordinates} (i).

It is easy then to see that $\Theta^{-1}(P_l)$ is homeomorphic to $\widehat{P}_l$.
Suppose now that two distinct polygons $P_l$ and $P_h$ intersect in a node $\bm{\bar{\th}}$.
The preimage $\Theta^{-1}(\bm{\bar{\th}})$ is an open segment and 
$\Theta^{-1}(P_l\cup P_h)$
is homeomorphic to the space obtained from  $\widehat{P}_l\sqcup \widehat{P}_h$ by identifying the nodal edges that correspond to $\bm{\bar{\th}}$.

In order to understand such identification, choose an orientation of $\MT_{bal}(\th)$
in a neighbourhood of $\Theta^{-1}(\bm{\bar{\th}})$ and an orientation of the plane $\Pi(\th)$,
so that $P_l$ and $\widehat{P}_l$ inherit an orientation from $\Pi(\th)$,
and each nodal edge of $\widehat{P}_l$ is induced an orientation from $\widehat{P}_l$.
Together with Corollary \ref{smootharea}, the last paragraph of the proof of \cite[Proposition 4.7]{EGnew} shows that $\Theta$ is orientation-preserving
on one of the two polygons $P_l$ or $P_h$ and orientation-reversing on the other. Hence, the two nodal edges
corresponding to $\bm{\bar{\th}}$
are identified through a map that preserves their orientation; we can also prescribe that such identification is a homothety in the chosen
realizations of $\widehat{P}_l$ and $\widehat{P}_h$.

Part of the above analysis can be rephrased as follows.

\begin{lemma}\label{blowup}
The space $\MT_{bal}(\th)$ is homeomorphic to  the real blow-up of $\mathrm{Crp}_{bal}(\th)$ at its nodes.
\end{lemma}

A further step in describing the topology of $\MT_{bal}(\th)$ is to study its ends.

\begin{fundamental}[The strips $\Strip_{i,a}(\th)$]\label{construction:strips}
As remarked above, every ideal edge of $P_l$ has equation $\th_i=a+(c+1)/2$ for some $a\in\{0,\dots,m-1\}$ and $i\in\{1,2,3\}$.
Viewing $\widehat{P}_l$ inside $\RR^2$, an open thickening of the corresponding ideal edge intersects $\widehat{P}_l$
in a region $\Strip^l_{i,a}(\th)$ homeomorphic to $[0,1]\times\RR$, where $\{0,1\}\times\RR$ correspond to portions of nodal or semi-balanced segments.
In every $\widehat{P}_l$ such thickenings can be chosen so that the corresponding regions are disjoint
and their ends $\{0\}\times\RR$ and $\{1\}\times\RR$ cover $1/3$ of the corresponding nodal or semi-balanced segment.
The complement inside $\widehat{P}_l$ of such strips is clearly compact.

It follows that, for fixed $i\in\{1,2,3\}$ and $a\in\{0,1,\dots,m-1\}$, 
the regions $\{\Strip^l_{i,a}(\th)\}$ glue to give a strip $\Strip_{i,a}(\th)$
homeomorphic to $[0,1]\times\RR$, with $\{0,1\}\times\RR$ corresponding to
semi-balanced triangles.
Thus there are $3m$ disjoint such strips, each one associated to a pair $(i,a)$.
\end{fundamental}

We are now ready to completely determine the topology of the space $\MT_{bal}(\th)$.

\begin{proposition}[Topology of the space of balanced triangles with assigned area]\label{baltopo} 
Suppose that $\th=2m+c$ where $c\in (-1,1)$. 
\begin{itemize}

\item[(i)] $\MT_{bal}(\th)$ is a connected, orientable, smooth bordered surface of finite type, whose boundary is the set of semi-balanced triangles.

\item[(ii)] The boundary of $\MT_{bal}(\th)$  is a union of $3m$ disjoint open intervals.

\item[(iii)] 
The surface $\MT_{bal}(\th)$ has $3m$ ends, namely the strips $\Strip_{i,a}(\th)$.
Each strip corresponds in  $\mathrm{Crp}_{bal}(\th)$ to a line 
$\th_i=a+(c+1)/2$ for some $a\in\{0,1,\dots,m-1\}$ and $i\in\{1,2,3\}$.
Moreover, each $\Strip_{i,a}(\th)$ is homeomorphic to $[0,1]\times\RR$
and $\{0,1\}\times\RR$ corresponds to semi-balanced triangles.
\item[(iv)] 
The Euler characteristic of $\MT_{bal}(\th)$ is $\chi(\MT_{bal}(\th))=-m(m-3)/2$.
%$\chi(\MT_{bal}(\th))=-m(m+3)/2$. Moreover, if $c\le 0$, $\MT_{bal}(\th)$ has a natural decomposition into $3m(m+1)/2$ open intervals and $m^2$ open disks. If  $c>0$, 
%$\MT_{bal}(\th)$ has a natural decomposition into $3m(m+3)/2$ open intervals and $m^2+3m$ open disks.
%
%
%There exists $\varepsilon(c)>0$ such that for any $\varepsilon\in (0,\varepsilon(c))$ the subset $\MT_{bal}(\th,\varepsilon)$ of $\MT_{bal}(\th)$ consisting of triangles with $|x_ix_j|\in(0,\varepsilon)\cup (2\pi-\varepsilon, 2\pi)$ is a disjoint union of $3m$ disks.
\end{itemize} 
\end{proposition}

\begin{proof}
(i) Thanks to Corollary \ref{smootharea} we only need to prove that $\MT_{bal}(\th)$ is connected and of finite type.
Since the balanced carpet $\mathrm{Crp}_{bal}(\th)$ is connected by Lemma \ref{connectedcarpet} (ii) and
it consists of finitely many nodes and polygons, both claims follow from Lemma \ref{blowup}.
%By Lemma \ref{coordinates} (i) and Proposition \ref{oneintconst}, the surjective map $\Theta: \MT_{bal}(\th)\to  \mathrm{Crp}_{bal}(\th)$  consists of contracting each open interval in $\MT_{bal}^{\mathbb Z}(\th)$ to a single point. 
%Since the balanced carpet $\mathrm{Crp}_{bal}(\th)$ is connected by Lemma \ref{connectedcarpet} (ii),
%we conclude that $\MT_{bal}(\th)$ is connected as well.
%Moreover, $\MT_{bal}(\th)$ is of finite type because $\mathrm{Crp}_{bal}(\th)$ consists of finitely many nodes and polygons.
%Recall that $\MT(\th)$ is a smooth surface by Proposition \ref{smoothKlein}. At the same time, $\MT_{bal}(\th)$ is a subset of $\MT(\th)$ given by three inequalities $\th_i\le \th_j+\th_k$. By Lemma \ref{localcoord}  the function  $\th_i-\th_j+\th_k$ is a local coordinate on $\MT(\th)$ in the neighbourhood of its zero level. Hence the curves $\th_i= \th_j+\th_k$ are smooth on $\MT(\th)$, and so $\MT_{bal}(\th)$ is a smooth bordered surface.
%Thanks to Corollary \ref{smootharea} we only need to prove that $\MT_{bal}(\th)$ is connected, since it is clearly of finite type.
%We note that the surjective map $\Theta: \MT_{bal}(\th)\to  \mathrm{Crp}_{bal}(\th)$  consists of contracting each open interval in $\MT_{bal}^{\mathbb Z}(\th)$ to a single point, see Lemma \ref{coordinates} (i) and Proposition \ref{oneintconst}. Since by Lemma \ref{connectedcarpet} (ii) the balanced carpet $\mathrm{Crp}_{bal}(\th)$ is  connected,
%%and the intervals of $\MT_{bal}^{\mathbb Z}(\th)$ are disjoint, 
%we conclude that $\MT_{bal}(\th)$ is connected as well.

(ii) It will be enough to show that the set of semi-balanced triangles with angles $\th_1,\,\th_2,\,\th_3$, satisfying $\th_1=\th_2+\th_3$ and $\th_1+\th_2+\th_3=\th$, is a union of $m$ open intervals. In case $c=0$ these $m$ intervals correspond to $m$ types of triangles with angles $\pi(m,\frac{1}{2}+l, \frac{1}{2}+m-l-1)$ where $l\in[0,m-1]$ is an integer number. In case $c\ne 0$ these intervals correspond to the intersection of the line $\th_1=\th_2+\th_3$ with $m$ open triangles of the carpet $\mathrm{Crp}(\th)$.

(iii) follows from Construction \ref{construction:strips}.

(iv) The internal part of $\MT_{bal}(\th)$ is an orientable surface without boundary
and so the Euler characteristic of its cohomology with compact support coincides with its Euler characteristic by Poincar\'e duality.
Decompose the interior of $\MT_{bal}(\th)$ into a finite union  of open $1$-cells $\MT_{bal}^{\mathbb Z}(\th)$ (corresponding
to internal nodes in the balanced carpet) and open $2$-cells (corresponding to the intersection of $\B(\th)$ with open triangles in the carpet).
By Lemma \ref{connectedcarpet} (iii) the space $\MT_{bal}(\th)$ is a union of $E$ open $2$-cells and $N$ open $1$-cells.
Thus, its Euler characteristic is $E-N=-m(m-3)/2$.
\end{proof}

%\begin{remark}\label{remark2pieps}  
%By Corollary \ref{balancedsystole} any balanced triangle with one side $x_1x_2$ satisfying $|x_1x_2|>2\pi-\varepsilon$ contains a geodesic segment of length less than $\varepsilon$ that joins $x_1$ with $x_2$.
%\end{remark}
%
%\begin{remark}\label{thinstrip} It is instructive to see what is the image in $\mathbb R^2$  under the map $\th_{12}$ of a connected component of $\MT_{bal}(\th, \varepsilon)$. Namely, this is a thin strip lying in a neighbourhood of a segment of the balanced set, given by equation $\th_i=(c+1)/2+l$, where $l\in \{0,\ldots, m-1\}$.
%\end{remark}

Let us now consider balanced triangles (with labelled vertices, as usual) endowed with an orientation. We stress that the orientation and the labelling of the vertices are unrelated.
Let $\MT^+_{bal}(\th)$ be the set of oriented balanced triangles of area $\pi(\th-1)$
in which the vertices are labelled anti-clockwise,
and let $\MT^-_{bal}(\th)$ be the analogous space in which the vertices are labelled clockwise.
Both sets can be given the topology induced by the identification with $\MT_{bal}(\th)$.
The space of oriented balanced triangles is then $\MT^+_{bal}(\th)\sqcup\MT^-_{bal}(\th)$.

\begin{definition}[Doubled space of balanced triangles]\label{def-oriented}
The {\it{doubled space of balanced triangles}}  of area $\pi(\th-1)$ is the space $\MT^{\pm}_{bal}(\th)$ obtained
from $\MT^+_{bal}(\th)\sqcup\MT^-_{bal}(\th)$ by identifying an oriented semi-balanced triangle
$\Delta$ to the triangle obtained from $\Delta$ by reversing its orientation.
\end{definition}

It follows that $\MT^{\pm}_{bal}(\th)$ is homeomorphic to the double of $\MT_{bal}(\th)$.

\begin{proposition}[The doubled space of balanced triangles of assigned area]\label{or-bal-tri}
Let $\th>1$ a non-odd real number and let $m=\floor{\frac{\th+1}{2}}$.
\begin{itemize}
\item[(i)] 
$\MT_{bal}^{\pm}(\th)$ is a connected, orientable surface of finite type, without boundary.
\item[(ii)] 
$\MT_{bal}^{\pm}(\th)$ has Euler characteristic $-m^2$, genus $(m-1)(m-2)/2$ and $3m$ punctures.
\item[(iii)] 
The action of $\Sy_3$ by relabelling the vertices of the triangles consists of
orientation-preserving homeomorphisms of $\MT_{bal}^{\pm}(\th)$.
\item[(iv)]
The action of $\Sy_3$ on the 
set of punctures of $\MT_{bal}^{\pm}(\th)$ has $m$ orbits of length $3$.
\end{itemize}
\end{proposition}
\begin{proof}
(i) is a consequence Proposition \ref{baltopo} (i), since $\MT_{bal}^{\pm}(\th)$ is the double of $\MT_{bal}(\th)$.

(ii) Since $\MT_{bal}^{\pm}(\th)$ is an orientable surface without boundary, the Euler characteristic agrees
with the Euler characteristic with compact support.
By Proposition \ref{baltopo} (ii) the surface $\MT_{bal}(\th)$ has boundary consisting of $3m$ open segments.
Hence $\chi(\MT_{bal}^{\pm}(\th))=2\chi(\MT_{bal}(\th))-3m=-m(m-3)-3m=-m^2$.

By Proposition \ref{baltopo} (iii) each end of $\MT_{bal}(\th)$
is associated to a strip $\Strip^i_a(\th)$ with $a\in\{0,1,\dots,m\}$ and $i\in\{1,2,3\}$,
and it is homeomorphic to $[0,1]\times\RR$, and so it doubles to 
punctured disk $S^1\times\RR$
inside $\MT^{\pm}_{bal}(\th)$, that will be denoted by $\End^i_a(\th)$. Hence, we obtain $3m$ punctures.
The genus of $g(\MT^{\pm}_{bal}(\th))=1-\frac{3m}{2}-\frac{1}{2}\chi(\MT^{\pm}_{bal}(\th))$ is then easily computed.

(iii) Choose an arbitrary orientation of $\MT^{\pm}_{bal}(\th)$.
We want to show that every transposition $(i\, j)\in \Sy_3$ acts on $\MT^{\pm}_{bal}(\th)$
through an orientation-preserving homeomorphism. Consider for instance the transposition $(2\, 3)$,
that sends a triangle in $\MT^+_{bal}(\th)$ with nonintegral angles $(\th_1,\th_2,\th_3)$
to the triangle in $\MT^-_{bal}(\th)$ with nonintegral angles $(\th_1,\th_3,\th_2)$.
Since $\MT^+_{bal}(\th)$ and $\MT^-_{bal}(\th)$ have opposite orientations when viewed as subsets of $\MT^{\pm}_{bal}(\th)$, it is enough to show that $(2\, 3)$ acts on $\MT_{bal}(\th)$ by
reversing its orientation. 

By Lemma \ref{connectedcarpet} (iv), there exists a point in $\mathrm{Crp}_{bal}(\th)$ with non-integral
coordinates $(\th_1,\th_2,\th_2)$, and so a corresponding
balanced triangle $\Delta$ in $\MT^{\pm}_{bal}(\th)$.
It is clear that
the transformation $(\th_1,\th_2,\th_3)\mapsto (\th_1,\th_3,\th_2)$ of $\mathrm{Crp}_{bal}(\th)$
reverses the orientation at $(\th_1,\th_2,\th_2)$. 
Hence, $(2\, 3)$ acts on $\MT_{bal}(\th)$ by reversing its orientation.

(iv) Each orbit of the $\Sy_3$-action on the ends $\End^i_a(\th)$
is of type $\{\End^1_a(\th),\,\End^2_a(\th),\,\End^3_a(\th)\}$.
Since $a\in\{0,1,\dots,m-1\}$, there are $m$ orbits of length $3$.
\end{proof}

\subsubsection{Case $\th$ odd}

The case of $\th=2m+1$ for some integer $m\geq 0$ is much easier to handle.

\begin{lemma}[Description of the balanced carpet]\label{bal-crp-odd}
The balanced carpet $\mathrm{Crp}_{bal}(2m+1)$ consists of
$m(m+1)/2$ internal nodes.
\end{lemma}
\begin{proof}
Triangles in $\MT_{bal}(2m+1)$ have area $2m\pi$ by Gauss-Bonnet.
By Lemma \ref{tri2npi} and Remark \ref{bal-3int}, the balanced carpet $\mathrm{Crp}_{bal}(2m+1)$
consists just of triples $(\th_1,\th_2,\th_3)\in\ZZ^3$ such that 
$\th_1+\th_2+\th_3=2m+1$ and $1\leq \th_i\leq m$ for all $i$.
It is easy to see that such points are $m(m+1)/2$ internal nodes.
\end{proof}

This easily lead to the description of the moduli space of balanced triangles.

\begin{proposition}[Topology of the space of balanced triangles]\label{bal-tri-odd}
The space $\MT_{bal}(2m+1)$ is diffeomorphic to the disjoint union of $m(m+1)/2$ copies
of the open $2$-simplex $\twosimpl$.
\end{proposition}
\begin{proof}
Fix $(\th_1,\th_2,\th_3)\in\mathrm{Crp}_{bal}(2m+1)$.
By Proposition \ref{nonintside} (ii) and Proposition \ref{nonintside}
the locus of triangles $\Delta$ in $\MT_{bal}(2m+1)$ with $\th_i(\Delta)=\th_i$ for $i=1,2,3$
is real-analytically diffeomorphic to the set of triples $(l_1,l_2,l_3)\in (0,2\pi)^3$ such that $l_1+l_2+l_3=2\pi$,
which is clearly homothetic to $\twosimpl$.
The conclusion then follows from Lemma \ref{bal-crp-odd}.
\end{proof}

Let $\mathrm{Crp}_{bal}^{\pm}(2m+1)$ be the disjoint union of two copies of $\mathrm{Crp}_{bal}(2m+1)$,
namely its elements are of type $(\bm{\th},\epsilon)$, where $\bm{\th}\in\mathrm{Crp}_{bal}^{\pm}(2m+1)$
and $\epsilon=\pm 1$.
We denote by $\MT_{bal}^{\pm}(2m+1)$ the doubled space of spherical triangles of area $2m\pi$
and by $\Theta^{\pm}:\MT_{bal}^{\pm}(2m+1)\rightarrow\mathrm{Crp}_{bal}^{\pm}(2m+1)$ the map that sends
an oriented triangle $\Delta$ to $(\bm{\th}(\Delta),\epsilon(\Delta))$, where 
$\epsilon(\Delta)=1$ if the vertices of $\Delta$ are numbered anti-clockwise,
and $\epsilon(\Delta)=-1$ otherwise.

\begin{proposition}[Topology of the doubled space of balanced triangles]\label{int-or-bal}
The space $\MT_{bal}^{\pm}(2m+1)$ is diffeomorphic to $\mathrm{Crp}_{bal}^{\pm}(2m+1)\times\twosimpl$,
namely to the disjoint union of $m(m+1)$ open $2$-simplices.
The permutation group $\Sy_3$ that relabels the vertices of a triangle in $\MT_{bal}^{\pm}(2m+1)$
acts on an element $(\bm{\th},\epsilon,\bm{y})$ of $\mathrm{Crp}_{bal}^{\pm}(2m+1)\times\twosimpl$
permuting the coordinates of $\bm{\th}$ and $\bm{y}$, and through its sign on $\epsilon$.
\end{proposition}
\begin{proof}
The first claim relies on Proposition \ref{bal-tri-odd}. The remaining ones are straightforward.
\end{proof}

%%%%%%%%%%%%%%%%%%%%%%%%%%%%%%%%%%%%%%%%%%%%%%%%%%%%%
%%%%%%%%%%%%%%%%%%%%%%%%%%%%%%%%%%%%%%%%%%%%%%%%%%%%%
%%%%%%%%%%%%%%%%%%%%%%%%%%%%%%%%%%%%%%%%%%%%%%%%%%%%%
%%%%%%%%%%%%%%%%%%%%%%%%%%%%%%%%%%%%%%%%%%%%%%%%%%%%%
%%%%%%%%%%%%%%%%%%%%%%%%%%%%%%%%%%%%%%%%%%%%%%%%%%%%%

\section{Moduli spaces of spherical tori}\label{sec:moduli}

The goal of this section is to describe the topology of the moduli space $\MSPH_{1,1}(\th)$
and so to prove Theorem \ref{mainonodd} (case $\th$ non-odd) and Theorems \ref{mainodd2}-\ref{mainodd} (case $\th$ odd).

We recall that, by {\it{isomorphism}} between two spherical tori, we mean an orientation-preserving isometry.
We refer to Section \ref{secLipschitz} for the definition of Lipschitz distance and topology on $\MSPH_{1,1}(\th)$
and $\MSPH_{1,1}^{(2)}(\th)$ needed below.

The object of our interest is the following.

\begin{definition}[$\MSPH_{1,1}(\th)$ as a topological space]\label{topspaces}
The space $\MSPH_{1,1}(\th)$ is the set of isomorphism classes of spherical tori with one conical point of angle $2\pi\th$,
endowed with the Lipschitz topology. 
\end{definition}

In order to prove Theorem \ref{mainonodd} it will be convenient to introduce the notion of $2$-marking.

\begin{definition}[$2$-marking]\label{def:2marking}
A {\it{$2$-marking}} of a spherical torus $T$ with one conical point $x$ is a labelling of its nontrivial $2$-torsion points
or, equivalently, an isomorphism $H_1(T;\ZZ_2)\cong (\ZZ_2)^2$.
\end{definition}

There is a bijective correspondendence between isomorphisms $\mu:(\ZZ_2)^2\rar H_1(T;\ZZ_2)$
and orderings of the three non-trivial elements of $H_1(T;\ZZ_2)$:
it just sends $\mu$ to the triple $(\mu(e_1),\mu(e_2),\mu(e_1+e_2))$.
In fact, the action of $\mathrm{SL}(2,\ZZ_2)$ on $2$-markings corresponds to the $\Sy_3$-action
that permutes the orderings.
If the torus $T$ has a spherical metric with conical point $x$,
the nontrivial conformal involution $\sigma$ fixes $x$ and its three non-trivial $2$-torsion points:
the above ordering is then equivalent to the labelling of such three points.
In this case, an isomorphism of between two $2$-marked spherical tori is an orientation-preserving isometry
compatible with the $2$-markings.

\begin{definition}[$\MSPH_{1,1}^{(2)}$ as a topological space]
The space $\MSPH_{1,1}^{(2)}(\th)$ is the set isomorphisms classes of $2$-marked spherical tori with one conical point of angle $2\pi\th$, 
endowed with the Lipschitz topology.
\end{definition}

In Remark \ref{orbifold-correct} we show that $\MSPH_{1,1}(\th)$ and $\MSPH_{1,1}^{(2)}(\th)$ can be endowed with the structure of orbifolds
in such a way that the map $\MSPH_{1,1}(\th)\rar\MSPH_{1,1}^{(2)}(\th)$ that forgets the $2$-marking is
a Galois cover with group $\Sy_3$ (which is unramified in the orbifold sense).

%
%One can similarly define the topological space $\MSPH_{0,4}(\th_1,\dots,\th_4)$.
%%
%%
%%for us to consider the degree $6$ cover $\MSPH_{1,1}^{(2)}(\th)$ of $\MSPH_{1,1}(\th)$ consisting of spherical tori for which all three points of order $2$ are labelled. The symmetric group  $\Sy_3$ is acting on $\MSPH_{1,1}^{(2)}(\th)$ by relabelling the order two points, and obviously we have  the  isomorphism $\MSPH_{1,1}(\th)\cong \MSPH_{1,1}^{(2)}(\th)/\Sy_3$.   
%%
%Observe that one can define the topological space  $\MSPH_{0,4}(\th_1,\dots,\th_4)$ along the line of Definition \ref{topspaces}, and endow it
%with a manifold structure (in particular, with the trivial orbifold structure, since surfaces of genus $0$ with $4$ marked points have no nontrivial conformal automorphisms).
%As anticipated in Remark \ref{rmk:orbi}, 
%for $\th$ not odd
%$\MSPH_{1,1}^{(2)}(\th)$ and $\MSPH_{0,4}\left(\frac{\th}{2},\frac{1}{2},\frac{1}{2},\frac{1}{2} \right)$ are homeomorphic as topological spaces but they are not isomorphic as orbifolds,
%since every point in $\MSPH_{1,1}^{(2)}(\th)$ has orbifold order $2$ (see Proposition \ref{toriwithsym}), whereas
% $\MSPH_{0,4}\left(\frac{\th}{2},\frac{1}{2},\frac{1}{2},\frac{1}{2} \right)$ has trivial orbifold structure. 
%The same considerations hold for the moduli spaces
%$\MSPH_{1,1}^{(2)}(2m+1)^\sigma$ of $\sigma$-invariant tori
%and $\MSPH_{0,4}\left(m+\frac{1}{2},\frac{1}{2},\frac{1}{2},\frac{1}{2} \right)$
%
%
%
%
%In what follows an automorphism is an orientation-preserving isometry. 

\subsection{The case $\th$ not odd integer}\label{secmainodd}

%The orbifold points on  $\MSPH_{1,1}(\th)$ correspond to surfaces with automorphisms of order $6$ and $4$. At the same time the automorphism group of any spherical torus with labelled  order $2$ points is $\mathbb Z_2$.

Because of the relevance for the orbifold structure of the moduli spaces we are interested in,
we first classify all possible automorphisms of spherical tori with one conical point.

\begin{proposition}[Automorphisms group of a spherical torus ($\th$ non-odd)]\label{toriwithsym} 
Suppose that $\th\notin 2\mathbb Z+1$. 
For any spherical torus $(T,x)$ of area $2\pi(\th-1)$ the group of automorphisms $G_T$ is isomorphic
either to $\mathbb Z_2$, or to $\ZZ_4$, or to $\ZZ_6$. 
\begin{itemize} 
\item[(i)] 
A torus with automorphism group $\ZZ_6$ exists if and only if $d_1(\th, 6\mathbb Z)>1$.
\item[(ii)] 
A torus with automorphism group $\ZZ_4$ exists if and only if
$d_1(\th, 4\mathbb Z)>1$.
\item[(iii)] 
For each $\th$ there can be at most one torus with automorphism $\ZZ_4$ 
and one torus with automorphism $\ZZ_6$.
\item[(iv)]
The subgroup of $G_T$ of automorphisms that fix the $2$-torsion points of $T$
is isomorphic to $\ZZ_2$, generated by the conformal involution.
\end{itemize}
%In all cases, the unique automorphism of order $2$ is the conformal $(-1)$-involution.
\end{proposition}
\begin{proof} 
Recall that by Proposition \ref{conformaliso} each torus has an automorphism of order $2$, namely the conformal involution.
Clearly such involution fixes the $2$-torsion points of the torus. This implies (iv) and it proves that $|G_T|$ is even.

To bound the automorphism group we note that  the action of $G_T$ fixes $x$ and preserves the conformal structure on $T$. Hence, in case  $|G_T|>2$ the torus $T$ is biholomorphic  to  either $T_4$ or $T_6$, and its automorphisms group is $\ZZ_4$ or $\ZZ_6$ correspondingly.

Let us now prove the existence part of (i) and (ii).

(i) Suppose that $d_1(\th, 6\mathbb Z)>1$. According to Theorem \ref{nonintriang}, this condition is equivalent to existence of a spherical triangle $\Delta$  with angles $\pi\th/3$. Such a triangle has a rotational $\mathbb Z_3$-symmetry. It follows that the torus $T(\Delta)$ has an automorphism of order $6$. 

(ii) Suppose that  $d_1(\th, 4\mathbb Z)>1$. According to Theorem \ref{nonintriang}, this condition is equivalent to existence of a spherical triangle $\Delta$  with angles $\pi(\th/2,\th/4,\th/4)$. This triangle has a reflection, i.e. an anti-conformal isometry that exchanges two vertices of angles $\pi \th/4$. Gluing two copies of $\Delta$ along the edge that faces the angle $\pi\th/2$, we obtain a quadrilateral with
four edges of the same length and four angles $\pi\th/2$. It is easy to see that such quadrilateral has a rotational $\mathbb{Z}_4$-symmetry,
and so that $T(\Delta)$ has an order $4$ automorphism.

Let now $(T,x,\th)$ be any spherical torus with $|G_T|>2$ and let us show that it has to be one of two tori constructed above. 
Consider two cases. 

First, suppose that the Voronoi graph $\Gamma(T)$ is a trefoil. In this case by Proposition \ref{threegeodesics} and Theorem \ref{twotrianglesTH} there is a unique collection of three geodesic loops $\gamma_1,\,\gamma_2,\,\gamma_3$ based at $x$ that cut $T$ into two isometric strictly balanced triangles $\Delta$ and $\Delta'$. This collection is sent by $G_T$ to itself, and so  $|G_T|$ is divisible by three, hence $|G_T|=6$. It is easy to see then that the subgroup $\mathbb Z_3\subset G_T$ sends $\Delta$ to itself and permutes its vertices. So $\Delta$ has angles $\pi\th/3$ and so we are in case (i).
Since $\th/3$ cannot be integer, this also proves the uniqueness of a torus with automorphism group $\ZZ_6$.

Suppose now that the Voronoi graph $\Gamma(T)$ is an eight graph. Then again by Proposition \ref{threegeodesics} and Theorem \ref{twotrianglesTH} there is a canonical collection of four geodesic loops $\gamma_1,\,\gamma_2,\,\eta_1,\,\eta_2$. Since $G_T$ sends the couple $\eta_1,\eta_2$ to itself, we see that  geodesics $\eta_1$ and $\eta_2$ cut a neighbourhood of $x$ into four sectors of angles $\pi\th/2$. The same holds for the couple of loops $\gamma_1$ and $\gamma_2$. Since by Remark \ref{bysector} each $\gamma_i$ bisects two sectors formed by $\eta_1$ and $\eta_2$ we see that, taken together, the geodesics $\gamma_1,\, \gamma_2,\, \eta_1,\, \eta_2$ cut a neighbourhood of $x$ into four eight sectors of angles $\pi\th/4$. Hence $\gamma_1,\,\gamma_2,\,\eta_1$ cut $\Delta$ into two semi-balanced triangles with angles $\pi(\th/2, \th/4,\th/4)$, and so we are in case (ii).
The uniqueness of a torus with automorphism group $\ZZ_4$ follows from the uniqueness of an isosceles triangle with angles
$\pi(\th/2,\th/4,\th/4)$.
\end{proof}

We recall in more detail the construction mentioned in the introduction.

\begin{fundamental}\label{fund2}
Consider the following
maps of sets 
\[
\xymatrix{
\MT_{bal}^{\pm}(\th)   \ar@/^1pc/[rr]^{T^{(2)} }
&&  \MSPH^{(2)}_{1,1}(\th)     \ar@/^1pc/[ll]^{\Delta^{(2)}} 
}
\]
%
%$T^{(2)}:\MT_{bal}^{\pm}(\th)\rightarrow \MSPH^{(2)}_{1,1}(\th)$
%and $\Delta^{(2)}:\MSPH^{(2)}_{1,1}(\th)\rightarrow\MT_{bal}^{\pm}(\th)$.
The map $T^{(2)}$ is defined
by sending an oriented triangle $\Delta$ to the torus $T(\Delta)$,
where we mark by $p_i$ the midpoint of the side opposite to the vertex  $x_i$ of $\Delta$. 

As for $\Delta^{(2)}$, we proceed as follows.
Let $(T,x,\bm{p})$ be a torus with its order $2$ points  marked by $p_1,\, p_2,\,p_3$.

Suppose first that $T$ does not have a rectangular involution. By Theorem \ref{twotrianglesTH} there is  a unique collection of three geodesics loops $\gamma_i$ that cuts $T$ into two congruent strictly balanced triangles $\Delta$ and $\Delta'$. We enumerate the geodesics so that each $p_i$ is the midpoint of $\gamma_i$. Next, we label the vertices of $\Delta$ by $x_1,\,x_2,\,x_3$ so that $x_i$ is opposite to $\gamma_i$. Hence, we associate to $T$ a unique strictly balanced triangle with enumerated vertices. In case the vertices of $\Delta$ go in anti-clockwise order, we associate to $\Delta$  the corresponding point in the interior of $\MT_{bal}^+(\th)$, otherwise we associate to $\Delta$ a point in the interior of $\MT_{bal}^-(\th)$.

Suppose now that $T$ has a rectangular involution. Then by Theorem \ref{twotrianglesTH} the torus $T$ can be cut into two isomorphic semi-balanced triangles in two different ways. At the same time the rectangular involution sends one pair to the other by reversing the orientation and fixing the labelling of the vertices. This means that the two points associated to $T$ in the boundaries of $\MT^+_{bal}(\th)$ and $\MT^{-}_{bal}(\th)$ are identified in $\MT_{bal}^{\pm}(\th)$.
\end{fundamental}

At this point we have the tools to prove the following preliminary fact.

\begin{lemma}[$T^{(2)}$ is bijective]\label{T-bijective}
The map $T^{(2)}:\MT_{bal}^{\pm}(\th)\rightarrow \MSPH^{(2)}_{1,1}(\th)$ is a bijection and $\Delta^{(2)}$ is its inverse.
\end{lemma}
\begin{proof}
It is very easy to see that $T^{(2)}\circ\Delta^{(2)}$ is the identity of $\MSPH^{(2)}_{1,1}(\th)$.
Vice versa, $\Delta^{(2)}\circ T^{(2)}$ is the identify of $\MT_{bal}^{\pm}(\th)$
by Theorem \ref{twotrianglesTH}.
\end{proof}

\begin{remark}[Orbifold Euler characteristic]\label{defchi}
We recall from the introduction that we are using the definition of orbifold Euler characteritic
given at page 29 of \cite{orbifold}. We are particularly interested in two properties enjoyed by the
orbifold Euler characteristic:
\begin{itemize}
\item[(a)]
if $\mathcal{Y}\rightarrow\mathcal{Z}$ is an orbifold cover of degree $d$,
then $\chi(\mathcal{Y})=d\cdot\chi(\mathcal{Z})$;
\item[(b)]
if $\mathcal{Y}$  is a connected, orientable, two-dimensional orbifold 
with underlying topological space $Y$, then
\[
\chi(\mathcal{Y})=\frac{1}{\mathrm{ord}(\mathcal{Y})}\chi(Y)-
\sum_{y} \left(\frac{1}{\mathrm{ord}(\mathcal{Y})}-\frac{1}{\mathrm{ord}(y)}\right),
\]
where $\mathrm{ord}(\mathcal{Y})$ is the orbifold order of a general point of $Y$
and $\mathrm{ord}(y)$ is the orbifold order of $y\in Y$, and the sum is ranging
over points $y\in Y$ that have orbifold order strictly greater than $\mathrm{ord}(\mathcal{Y})$.
\end{itemize}

% an equivalent definition of orbifold Euler characteristic tailored to our needs.
%
%%\begin{remark}[Orbifold Euler characteristic]\label{orbichi}
%In the two-dimensional oriented case, we can associate to an orbifold $\mathcal{Y}$
%its underlying topological surface $Y$ and an order function $\mathrm{ord}:Y\rightarrow\ZZ_{\geq 1}$. Such function will be almost constant
%over each connected component of $Y$, but it can jump at isolated points: in particular,
%if $(y_n)\rar y$ in $Y$ and $\mathrm{ord}(y_n)$ is constant, then $\mathrm{ord}(y)$
%will be a multiple of $\mathrm{ord}(y_n)$.
%In this case, an equivalent definition of 
%
%\begin{definition}[Orbifold Euler characteristic]\label{defchi}
%The {\it{orbifold Euler characteristic}} of a
%connected, orientable, two-dimensional orbifold $\mathcal{Y}$ 
%with underlying topological space $Y$ is
%\[
%\chi(\mathcal{Y})=\frac{1}{\mathrm{ord}(\mathcal{Y})}\chi(Y)-
%\sum_{y} \left(\frac{1}{\mathrm{ord}(\mathcal{Y})}-\frac{1}{\mathrm{ord}(y)}\right) 
%\]
%where $\mathrm{ord}(\mathcal{Y})$ is the orbifold order of a general point of $Y$
%and $\mathrm{ord}(y)$ is the orbifold order of $y\in Y$, and the sum is ranging
%over points $y\in Y$ that have orbifold order strictly greater than $\mathrm{ord}(\mathcal{Y})$.
%\end{definition}
%
%With this definition, if $\mathcal{Y}\rightarrow\mathcal{Z}$ is an orbifold cover of degree $d$,
%the orbifold Euler characteristics satisfy $\chi(\mathcal{Y})=d\cdot\chi(\mathcal{Z})$.
%
Since we will only compute $\chi$ for two-dimensional, connected, orientable orbifolds,
property (b) could even be taken as a definition.
\end{remark}

%In order to study the continuity properties of $T^{(2)}$,
%recall that we endow the set $\MSPH_{1,1}^{(2)}(\th)$ with a topology
%induced by the Lipschitz distance (see Section \ref{secLipschitz}).
The main ingredient for the proof of Theorem \ref{mainonodd} is to show
that the map $T^{(2)}$ is a homeomorphism and so that, as a topological space,
$\MSPH_{1,1}^{(2)}(\th)$ is a surface.
As a consequence, we can endow $\MSPH_{1,1}^{(2)}(\th)$ with an orbifold structure (as done in Remark \ref{orbifold-correct})
in such a way that every point has orbifold order $2$, which is coherent with Proposition \ref{toriwithsym} (iv).

%In particular,$\MSPH_{1,1}^{(2)}(\th)$
%of the orbifold structure given by taking the quotient of such topological surface
%by the trivial $\ZZ_2$-action (which is the only orbifold structure for which every point
%has orbifold order $2$).
%
%In view of Proposition \ref{toriwithsym} (iv), we will endow $\MSPH_{1,1}^{(2)}(\th)$
%of the orbifold structure given by taking the quotient of such topological surface
%by the trivial $\ZZ_2$-action (which is the only orbifold structure for which every point
%has orbifold order $2$).

\begin{theorem}[Moduli space of spherical tori with $2$-marking]\label{markedtori} 
Let $\th>1$ be a real number such that $\th\notin 2\mathbb Z+1$ and let $m=\floor{\frac{\th+1}{2}}$. 
As a topological space, $\MSPH_{1,1}^{(2)}(\th)$ 
has the following properties.
%is a smooth, connected, orientable surface of finite type with the following properties.
\begin{itemize}
\item[(i)] 
The map $T:\MT_{bal}^{\pm}(\th)\rightarrow \MSPH_{1,1}^{(2)}(\th)$ is a homeomorphism and so
$\MSPH_{1,1}^{(2)}(\th)$ is
a connected, orientable surface of finite type without boundary.
%As an orbifold, $\MSPH_{1,1}^{(2)}(\th)$ is isomorphic to the quotient of $\MT_{bal}^{\pm}(\th)$ by the trivial $\ZZ_2$-action.
\item[(ii)] 
%$\MSPH_{1,1}^{(2)}(\th)$ has orbifold Euler characteristic $-m^2/2$,
It has genus $(m-1)(m-2)/2$ and $3m$ punctures.
\item[(iii)]
The group $\Sy_3$ that permuted the $2$-torsion points of a torus
acts on $\MSPH_{1,1}^{(2)}(\th)$ by orientation-preserving homeomorphisms.
\item[(iv)]
The action of $\Sy_3$ on the set of punctures 
of $\MSPH_{1,1}^{(2)}(\th)$
has $m$ orbits of length $3$.
\end{itemize}
As an orbifold, $\MSPH_{1,1}^{(2)}(\th)$ is isomorphic to the quotient of its underlying topological
space by the trivial $\ZZ_2$-action and its orbifold Euler characteristic is $-m^2/2$.
\end{theorem}
\begin{proof} 
%In (i) below it is shown that the topological space underlying $\MSPH_{1,1}^{(2)}(\th)$
%is a smooth connected orientable surface.
%Proposition \ref{toriwithsym} (iv) the automorphism group of each marked torus
%in $\MSPH_{1,1}^{(2)}(\th)$ is isomorphic to $\ZZ_2$, generated by the conformal involution.
%Since the group $\ZZ_2$ has no nontrivial automorphisms,
%Proposition \ref{conformaliso} (i)
%By what we have recalled at the beginning of Section \ref{secmainodd},
%the moduli space $\MSPH_{1,1}^{(2)}(\th)$ is an orientable, connected two-dimensional orbifold,
%which is diffeomorphic to the quotient of its underlying topological space by the trivial $\ZZ_2$-action.
%In fact, the orbifold order of each point of $\MSPH_{1,1}^{(2)}(\th)$ is two by Proposition \ref{toriwithsym} (iv).
%
%Orientability, connectedness and finite-typeness of $\MSPH_{1,1}^{(2)}(\th)$ follow
%from those of $\MT_{bal}^{\pm}(\th)$ and from part (i) here below.
%
The map $T^{(2)}$ is bijective by Lemma \ref{T-bijective} and in fact
a homeomorphism by Theorem \ref{torusbalancehomeo}.
Hence, (i-iv) follow from Proposition \ref{or-bal-tri} (i-iv).
The orbifold structure was described just above the statement of the theorem: 
the involution $\sigma$ is the only nontrivial automorphism of a point in $\MSPH_{1,1}^{(2)}(\th)$
by Proposition \ref{toriwithsym} (iv), and it acts trivially on $\MT_{bal}^{\pm}(\th)$.
Hence, $\MSPH_{1,1}^{(2)}(\th)$ is isomorphic to the quotient of $\MT_{bal}^{\pm}(\th)$ by the trivial $\ZZ_2$-action.
As a consequence,
the orbifold Euler characteristic satisfies $\chi(\MSPH_{1,1}^{(2)}(\th))=\chi(\MT_{bal}^{\pm}(\th))/2$.
\end{proof}

%
%
%We recall that, by results of \cite{luo:monodromy} (see  \cite[Proposition 6.4]{MP:systole}), the moduli spaces $\MSPH_{1,1}(\th)$ 
%and $\MSPH_{1,1}^{(2)}(\th)$ are two-dimensional real analytic orbifolds if $\th\notin 2\mathbb Z+1$.
As above, we can endow $\MSPH_{1,1}(\th)$ with an orbifold structure 
as in Remark \ref{orbifold-correct},
in such a way that the orbifold order of a point in $\MSPH_{1,1}(\th)$
agrees with the number of automorphisms of the corresponding spherical torus.
%It also follows that the map $\MSPH_{1,1}^{(2)}(\th)\rar\MSPH_{1,1}(\th)$ that forgets the $2$-marking
%is an unramified Galois cover of group $\mathrm{SL}(2,\ZZ_2)\cong\Sy_3$, where $\Sy_3$ acts
%by permuting the three $2$-torsion points.
%%
%
%the set $\MSPH_{1,1}(\th)$ is given the Lipschitz topology.
%We endow it with the structure of orbifold by viewing $\MSPH_{1,1}(\th)$ as
%$\MSPH_{1,1}^{(2)}(\th)/\Sy_3$, where $\Sy_3$ permutes the $2$-torsion points of each torus.
%As a consequence, the orbifold order of a point in $\MSPH_{1,1}(\th)$ is the order of the automorphism group of the associated spherical torus.
%
%%the orbifold order of a point in $\MSPH_{1,1}^{(2)}(\th)$ is the order of the group of automorphisms
%%of the associated torus that fix its $2$-torsion points. 

Let us finally prove Theorem \ref{mainonodd}.

\begin{proof}[Proof of Theorem \ref{mainonodd}] 
By Remark \ref{orbifold-correct}, the map $\MSPH_{1,1}^{(2)}(\th)\rar\MSPH_{1,1}(\th)$ 
that forgets the $2$-marking is an unramified $\Sy_3$-cover of orbifolds.
Hence, $\MSPH_{1,1}(\th)$ is a smooth, connected two-dimensional orbifold of finite type by Theorem \ref{markedtori} (i),
and orientability follows from Proposition \ref{toriwithsym}.

(ii-iii-iv) Clearly $\chi(\MSPH_{1,1}(\th))=\chi(\MSPH_{1,1}^{(2)}(\th))/|\Sy_3|=-m^2/12$
by Theorem \ref{markedtori}. Also, (iii-iv) and remaining claim of (ii)
are established in Proposition \ref{toriwithsym}.

(i) The space $\MSPH_{1,1}(\th)$ has $m$ punctures by Theorem \ref{markedtori} (ii,iv).
Moreover its (non-orbifold) Euler characteristic of is $2(-m^2/12+\epsilon)$,
where $\epsilon\in\{ 0,\, \frac{1}{4},\, \frac{1}{3},\, \frac{7}{12}=\frac{1}{4}+\frac{1}{3}\}$.
Indeed, a point of order $4$ in $\MSPH_{1,1}(\th)$ contributes to $\epsilon$
with $\frac{1}{4}=\frac{1}{2}-\frac{1}{4}$ and a point of order $6$ contributes 
with $\frac{1}{3}=\frac{1}{2}-\frac{1}{6}$.
Hence, the genus of $\MSPH_{1,1}(\th)$ is $1-\frac{1}{2}(m+2(-m^2/12+\epsilon))=
\floor{\frac{m^2-6m+12}{6}}$.
\end{proof}

Let us finish this subsection with a simple corollary of Theorem \ref{markedtori}.
As a topological space,
we denote by $\overline{\MSPH}_{1,1}^{(2)}(\th)$ the unique smooth compactification
of the surface $\MSPH_{1,1}^{(2)}(\th)$ obtained by filling in the $3m$ punctures.
As above, we endow $\overline \MSPH_{1,1}^{(2)}(\th)$ with the orbifold strucure
given by taking the quotient of its underlying topological space by the trivial $\ZZ_2$-action.

%by filling in the $3m$ punctures, and making the added points into orbifold points of order $2$.

\begin{corollary}[A cell decomposition of $\overline{\MSPH}_{1,1}^{(2)}(\th)$]\label{compactcell}  
Suppose that $\th=2m+c$, where $c\in (-1,1)$.
As a topological space, $\overline{\MSPH}_{1,1}^{(2)}(\th)$ has the following properties.
\begin{itemize}
\item[(i)]
It is a compact, connected, orientable surface of genus $(m-1)(m-2)/2$. 
\item[(ii)]
It has a natural structure of a CW complex, where 
\begin{itemize}
\item
its $0$-cells are the $3m$ added points, 
\item
its $1$-cells are formed by tori $T$ such that $\Delta(T)$ is ether a semi-balanced triangle, or a triangle with one integral angle, 
\item
its $2$-cells are the complement  to the union of $0$-cells and $1$-cells. 
\end{itemize}
Moreover, for $c\le 0$ the cell decomposition is a triangulation into $2m^2$ triangles.
\end{itemize}
\end{corollary} 

\begin{proof}  
Let us comment on the last claim, since the other claims are rather immediate after Theorem \ref{markedtori}.
Recall, that in the proof of Proposition \ref{baltopo} (iv) for $c\le 0$ we constructed a decomposition of $\MT_{bal}(\th)$ in the union of $3m(m+1)/2$ one-cells and  $m^2$ two-cells. One can check that each of theses $m^2$ cells has exactly three $1$-cells in its boundary. Hence, we get a triangulation of
the topological space $\overline{\MSPH}_{1,1}^{(2)}(\th)$.  

Note however, that for $c>0$ the total number of $2$-cells is $2m^2+6m$ and the   additional $6m$ cells are digons rather than triangles.
\end{proof}

%\input torus-one-point-triangles.tex 

%%%%%%%%%%%%%%%%%%%%%%%

\subsection{The case $\th$ odd integer}\label{sec:moduliodd}

In this subsection we prove Theorems \ref{mainodd2}-\ref{mainodd}.
Our first step will be to prove Theorem \ref{twointiranglesTH}, from which part (a) of Theorem \ref{mainodd2} will be easily obtained.

\begin{proof}[Proof of Theorem \ref{twointiranglesTH}] 
According to Proposition \ref{conformaliso}, there is a unique curvature $1$ metric  on $T$ with angle $2\pi(2m+1)$ in given
projective equivalence class, which is invariant under the conformal involution $\sigma$ of $T$.  
Hence we can apply Proposition \ref{threegeodesics} to $T$ endowed with such $\sigma$-invariant metric. According to such proposition, there exist three geodesic loops based at the conical point $x$ that cut $T$ into two isometric balanced triangles $\Delta$ and $\Delta'$. By Gauss-Bonnet formula we have $\area(\Delta)=2\pi m$  and so we can apply Lemma \ref{tri2npi}. According to such lemma, $\Delta$ is a balanced triangle with angles $2\pi\cdot (m_1,m_2,m_3)$ where $m_1+m_2+m_3=2m+1$.   This finishes the proof of the theorem. 
\end{proof}

Such a result already allows us to describe $\MSPH_{1,1}(2m+1)^\sigma$ as a topological space.

\begin{proof}[Proof of part (a) of Theorem \ref{mainodd2}] 
As in the proof of Theorem \ref{twointiranglesTH},
we can  associate to each torus with a $\sigma$-invariant metric
a unique oriented balanced spherical triangle with integral angles
and unmarked vertices. Clearly, an orientation on a triangle is equivalent
 to a numbering of its vertices up to cyclic permutations.
Such correspondence determines a bijective map
\[
T:\MT_{bal}(2m+1)/\ASy_3\lra \MSPH_{1,1}(2m+1)^{\sigma}
\]
where the 
alternating group $\ASy_3$ acts by relabelling the vertices of the triangle. 
Arguments entirely analogous to the ones used in Theorem \ref{torusbalancehomeo} (ii)
show that $T$ is continuous and proper, hence a homeomorphism of topological spaces.

By Proposition \ref{bal-tri-odd} the space $\MT_{bal}(2m+1)$ is homeomorphic to the disjoint union of
$m(m+1)/2$ copies of the open standard simplex $\mathring{\Delta}^2$. Each component
represents triangles of angles $\pi(m_1,m_2,m_3)$ with $m_1+m_2+m_3=2m+1$,
where $(m_1,m_2,m_3)$ is a triple of positive integers that satisfy the three triangle inequalities.

Consider now two cases.
\begin{itemize}
\item[(i)]
Suppose that $m\not\equiv 1\pmod 3$. 
In this case, the integer $2m+1$ is not divisible by $3$ and so 
neither of spherical triangles in $\MT_{bal}(2m+1)$ have all equal angles.
It follows that the action of $\ASy_3$ does not send any component to itself. So the number of components of $\MSPH_{1,1}(2m+1)^\sigma$ 
is $\frac{m(m+1)}{6}$ and each one is homeomorphic to the open $2$-disk $\mathring{\Delta}^2$.
\item[(ii)]
Suppose that $m\equiv 1\pmod 3$.
Then the component corresponding to triangles with angles $m_1=m_2=m_3=(2m+1)/3$ is the only one that is sent to itself. It contains a unique point fixed by $\ASy_3$, namely the equilateral spherical triangle, and the quotient of such component by $\ASy_3$ is homeomorphic to an open $2$-disk.
All the other $\frac{m(m+1)-2}{2}$ components of $\MT_{bal}(2m+1)$
are non-trivially permuted by $\ASy_3$: hence, they give $\frac{m(m+1)-2}{6}$ components of $\MSPH_{1,1}(2m+1)^\sigma$ homeomorphic to $\mathring{\Delta}^2$.
Hence, the total number of connected components of $\MSPH_{1,1}(2m+1)^\sigma$ is $\frac{m(m+1)+4}{6}$.
\end{itemize}
\end{proof}

The rest of the subsection is devoted to a careful analysis of the orbifold structures
on our moduli spaces and to the proof of part (b) of Theorem \ref{mainodd2}
and of Thoerem \ref{mainodd}.

\subsubsection{Voronoi graph and decorations}\label{sec:vor-integral}

The orbifold structure on our moduli spaces
is defined in Remark \ref{orbifold-correct},
but a more explicit interpretation of such structure for moduli spaces
of tori of area $4m\pi$ relies on the notion of decoration.

We begin with a simple lemma.

\begin{lemma}[Voronoi graphs of tori of area $4m\pi$]\label{lem:Vor-4mpi}
The Voronoi graph $\Gamma(T)$ of a spherical torus $T$ of area $4m\pi$ has two vertices
and three edges  of lengths $(2m_i+1)\pi$ for integers $m_i\geq 0$.
The two vertices are exchanged by the conformal involution $\sigma$.
Moreover, projectively equivalent spherical metrics on a torus have the same Voronoi graph.
%$\Gamma(T)$ is invariant under a projective deformation of the metric on $T$.
\end{lemma}
\begin{proof}
Consider first the case $m=1$.
A spherical triangle $\Delta_0$ with vertices $x_1,x_2,x_3$ of angles $(\pi,\pi,\pi)$ is isometric to a hemisphere
and its circumcenter $O$ is at distance $\pi/2$ from the boundary of such hemisphere.
So the rotations of the hemisphere that take $x_i$ to $x_j$ fix $O$.
A torus $T_0$ with a $\sigma$-invariant metric $h$ of area $4\pi$ is isometric to $T(\Delta_0)$
and so it has three edges and two vertices. Since $\sigma$ fixes the Voronoi graph $\Gamma(T_0)$
and pointwise fixes the conical point and the midpoints of the three edges of $\Gamma(T_0)$, it does not fix
any other point. In particular, $\sigma$ exchanges the two vertices of $\Gamma(T_0)$.
Moreover, the vertices of $\Gamma(T_0)$ are at distance $\pi/2$ from $\pa\Delta_0$, and so the edges
of $\Gamma(T_0)$ have length $\pi$.
It follows that a (multi-valued) developing map for $T_0$ sends the vertices of $\Gamma(T_0)$
to the two fixed points $O,O'$ for the monodromy, and the edges of $\Gamma(T_0)$ to meridians
running between $O$ and $O'$.
Note that another spherical metric on $T_0$ projectively equivalent to $h$ is obtained by post-composing
the developing map of $h$ by a M\"obius transformation that fixes $O$ and $O'$. Since such
transformations preserve the meridians between $O$ and $O'$, the two metrics have the same Voronoi graph.

Suppose now $m>1$.
By Theorem \ref{twointiranglesTH} and Proposition \ref{nonintside},
a torus $T$ with $\sigma$-invariant metric of area $4m\pi$
is obtained from a torus $T_0=T(\Delta_0)$ of area $4\pi$ as above
by gluing a sphere $S_i$ with two conical points of angles $2m\pi$ at distance
$|x_jx_k|$ along the geodesic segment $x_jx_k$ of $T_0$.
The conclusion then follows from the analysis of the case $m=1$.
%
%has a unique triple of loops $\gamma_{12},\gamma_{13},\gamma_{23}$ based at $x$
%that cut it into a couple $\Delta_0,\Delta'_0$ of congruent triangles with integral angles and 
%three digons $B_i$ of area $4m_i\pi$.
%Hence, $T$ 
%is obtained from a $\sigma$-invariant torus $T_0$ of area $4\pi$ as above
%by 
%
%There are three 
%
%A torus $T$ with $\sigma$-invariant metric 
%Theorem \ref{twointiranglesTH}
%is obtained as $T(\Delta)$
%and $\Delta$ is obtained from some $\Delta_0$
%a torus $T$ of area
%
%
%
%{\bf{TO BE REVIEWED}}\\
%
%
%
%All the claims except for the last one follow from Theorem \ref{twointiranglesTH}
%for tori with $\sigma$-invariant metrics.
%
%Now, the monodromy of such $T$ is coaxial
%and a developing map $\iota$
%sends the two vertices of $\Gamma$ to the two antipodal poles $O,O'$ of $\Sph$
%fixed by the monodromy. Hence the edges of $\Gamma$ are mapped to unions of meridians running
%from $O$ to $O'$. As a consequence, $\Gamma(T)$ is invariant under projective deformations,
%namely post-composing $\iota$ with a projective transformation of $\Sph$ that fix $O$ and $O'$.
%Hence all the claims hold even if the metric is not $\sigma$-invariant.
\end{proof}

In order to make the role of the conformal involution $\sigma$ in the below 
Constructions \ref{fund3}-\ref{fund4}
more transparent, we will need the following.

\begin{definition}[Decorations on strictly balanced tori]
%A {\it{$2$-marking}} $\bm{p}=(p_1,p_2,p_3)$ on a spherical (or just conformal) torus $(T,x)$ is
%a marking by $\bm{p}$ of its points of order $2$.
A {\it{decoration}} $v$ of a spherical torus $(T,x)$ is
a vertex $v$ of its Voronoi graph $\Gamma(T)$.
\end{definition}

The main reason for introducing decorations relies on the following fact.

\begin{lemma}[Rigidity of $2$-marked decorated spherical tori]\label{rigidity}
Decorated $2$-marked spherical tori of area $4m\pi$ 
have non-trivial automorphisms.
\end{lemma}
\begin{proof}
Being an isometry, an automorphism is in particular biholomorphic.
It is a classical fact that the only nontrivial biholomorphisms
a $2$-marked conformal torus $(T,x)$ is the involution $\sigma$. 
By Lemma \ref{lem:Vor-4mpi} the Voronoi graph $\Gamma(T)$ has two vertices
and they are exchanged by $\sigma$. 
\end{proof}

As a consequence, we obtain the following modular
interpretration of $\MSPH^{(2)}_{1,1}(2m+1)$ as a topological space.

\begin{remark}[The topological space $\MSPH^{(2)}_{1,1}(2m+1)$]
The topological space $\MSPH^{(2)}_{1,1}(2m+1)$ is the moduli space of
decorated $2$-marked spherical tori of area $4m\pi$.
\end{remark}

By Lemma \ref{lem:Vor-4mpi},
the non-trivial conformal involution $\sigma$ induces
an action $\sigma^*$ on $\MSPH^{(2)}_{1,1}(2m+1)$
by sending $(T,\bm{p},v,h)$ to $(T,\bm{p},v,\sigma^*h)$.
Since $\sigma:(T,\bm{p},v,\sigma^*h)\rar(T,\bm{p},\sigma(v),h)$ is an isomorphism,
we also have $\sigma^*(T,\bm{p},v,h)=(T,\bm{p},\sigma(v),h)$.

\subsubsection{Moduli spaces of $\sigma$-invariant spherical metrics of area $4m\pi$}

Similarly to what we did in Section \ref{secmainodd},
we first discuss the space of decorated $2$-marked tori.

\begin{fundamental}[Tori with $\sigma$-invariant metrics]\label{fund3}
If $\sigma$ is the unique (nontrivial) conformal involution of a conformal torus,
denote by $\MSPH_{1,1}^{(2)}(2m+1)^\sigma$ the set of $2$-marked decorated tori $(T,x,\bm{p})$
with a $\sigma$-invariant spherical metric of angle $(2m+1)2\pi$ at $x$.
We recall that triangles in
$\MT_{bal}^{\pm}(2m+1)$ have area $2m\pi$, integral angles 
and they are strictly balanced.
We then define the maps
\[
\xymatrix{
\MT_{bal}^{\pm}(2m+1)   \ar@/^1pc/[rr]^{T^{(2)} }
&&  \MSPH^{(2)}_{1,1}(2m+1)^\sigma    \ar@/^1pc/[ll]^{\Delta^{(2)}} 
}
\]
quite as in Construction \ref{fund2}.
In particular, $T^{(2)}$ sends an oriented triangle
$\Delta$ to the $2$-marked torus $T(\Delta)$ obtained as a union of $\Delta$ and
$\Delta'$, with the decoration given by the vertex of $\Gamma(T(\Delta))$ that sits inside $\Delta$.
\end{fundamental}

We easily have the following preliminary result.

\begin{theorem}[Moduli space of $2$-marked $\sigma$-invariant tori of area 
$4m\pi$]\label{int-MS-2mark}
For $m>0$ integer, the space $\MSPH^{(2)}_{1,1}(2m+1)^\sigma$ of decorated $2$-marked tori
with a $\sigma$-invariant spherical metric has the following properties.
\begin{itemize}
\item[(i)]
The map $T^{(2)}:\MT_{bal}^{\pm}(2m+1)\rightarrow \MSPH^{(2)}_{1,1}(2m+1)^\sigma$ is a homeomorphism, with inverse $\Delta^{(2)}$.
\item[(ii)] 
$\MSPH^{(2)}_{1,1}(2m+1)^\sigma$ is a disjoint union of $m(m+1)$ open $2$-disks $\mathring{\Delta}^2$.
\item[(iii)]
The group $\Sy_3$ acts on $\MSPH^{(2)}_{1,1}(2m+1)^\sigma$ permuting its components.
If $m\not\equiv 1\pmod3$, then all orbits have length $6$.
If $m\equiv 1\pmod 3$, then one orbit has length $2$ and all the other have length $6$.
\item[(iv)]
The action of $\sigma^*$ on the topological space $\MSPH^{(2)}_{1,1}(2m+1)^\sigma$ is trivial.
\end{itemize}
As an orbifold, the moduli space of $2$-marked tori with a $\sigma$-invariant spherical
metric is isomorphic to the quotient of $\MT_{bal}^{\pm}(2m+1)$ by the trivial $\ZZ_2$-action.
\end{theorem}
\begin{proof}
(i) It is very easy to see that $T^{(2)}\circ\Delta^{(2)}$ is the identity of $\MSPH^{(2)}_{1,1}(2m+1)^\sigma$. Vice versa, $\Delta^{(2)}\circ T^{(2)}$ is the identify of $\MT_{bal}^{\pm}(2m+1)$
by Theorem \ref{twointiranglesTH}. Hence, $T^{(2)}$ is bijective.
Moreover $T^{(2)}$ is a homeomorphism by Theorem \ref{torusbalancehomeo}.

(ii-iii) follow from Proposition \ref{int-or-bal} and Proposition \ref{bal-tri-odd}.

(iv) is clear, since the $2$-marked decorated spherical tori $(T,\bm{p},v,h)$ and $(T,\bm{p},v,\sigma^*h)$ 
are isomorphic via the map $\sigma$.

In view of Remark \ref{orbifold-correct}, the final claim follows from (iv).
\end{proof}

%For the bijectivity of $\Pi$ it is enough to work in a fixed projective class of spherical metric:
%the conclusion then follows from Proposition \ref{conformaliso} (ii).
%\end{proof}

Now we discuss the moduli space $\MSPH_{1,1}(2m+1)^\sigma$ of $\sigma$-invariant spherical tori.
%can prove Theorem \ref{mainodd2}.

\begin{proof}[Proof of part (b) of Theorem \ref{mainodd2}] 
Recall that $\MSPH_{1,1}(2m+1)^\sigma$ is endowed with a $2$-dimensional orbifold
structure by Remark \ref{orbifold-correct}.
By Theorem \ref{int-MS-2mark} (i), the space $\MT^{\pm}_{bal}(2m+1)$
is isomorphic to the moduli space of decorated $2$-marked tori with a $\sigma$-invariant metric.
On a fixed such torus, $2$-markings are permuted by $\Sy_3$ and the decorations are exchanged by $\sigma$.
Hence, the moduli space $\MSPH_{1,1}(2m+1)^\sigma$ is isomorphic
(as an orbifold) to the quotient of 
$\MT^{\pm}_{bal}(2m+1)$ by $\Sy_3\times \langle 1,\sigma^*\rangle$.
By Proposition \ref{int-or-bal}, such quotient can be identified to
$\MT_{bal}(2m+1)/\ASy_3\times\ZZ_2$, where
the alternating group $\ASy_3$ acts
by cyclically relabelling the vertices of the triangles
and $\ZZ_2$ acts trivially by Theorem \ref{int-MS-2mark} (iv).
By Proposition \ref{bal-tri-odd} the space $\MT_{bal}(2m+1)$ consists
of $m(m+1)/2$ connected component and it is diffeomorphic to $\mathrm{Crp}_{bal}(2m+1)\times \twosimpl$.

Consider now two cases. 

(b-i) Suppose $2m+1$ not divisible by $3$. 
In this case neither of spherical triangles in $\MT_{bal}(2m+1)$ have all equal angles, so the action of $\ASy_3$ does not send any component to itself. So the number of components of 
$\MSPH_{1,1}(2m+1)^\sigma$ is $\frac{m(m+1)}{6}$, each one is homeomorphic to 
the quotient $\Dodd$ of $\twosimpl$ by the trivial $\ZZ_2$-action 
and so all points have orbifold order $2$.

(b-ii) Suppose $2m+1$ divisible by $3$.
Then the component corresponding to triangles with angles $m_1=m_2=m_3=(2m+1)/3$ is the only one that is sent to itself. 
It contains a unique point fixed by $\ASy_3$, namely the equilateral spherical triangle. This point gives rise to an orbifold point of order $6$ on $\MSPH_{1,1}(2m+1)^\sigma$, which belongs to a component homeomorphic to
the quotient $\Dodd'$ of $\mathring{\Delta}^2$ by $\ZZ_2\times\ASy_3$, where $\ZZ_2$ acts trivially.
All the other $m(m+1)/2-1$
components are non-trivially permuted by $\ASy_3$, and they are all homeomorphic to $\Dodd$.
Hence, there are$\ceil{\frac{m(m+1)}{6}}$ connected component, and
all points except the equilateral spherical triangle
have orbifold order $2$.
\end{proof}

\subsubsection{Moduli spaces of spherical metrics of area $4m\pi$}

In order to treat spherical metrics that are not $\sigma$-invariant,
we need a further construction.

\begin{fundamental}\label{fund4}
Given a point $O\in\Sph$, let $R\in\mathfrak{su}(2)$ be the unique element with $\tr(R^2)=-1/2$
that generates anti-clockwise rotations of $\Sph$ at $O$.

We view the topological space $\MSPH^{(2)}_{1,1}(2m+1)$ as a moduli space of decorated, $2$-marked tori
and we define the following couple of maps
\[
\xymatrix{
\MT_{bal}^{\pm}(2m+1)\times\RR \ar@/^1pc/[rr]^\Xi && \MSPH^{(2)}_{1,1}(2m+1)
\ar@/^1pc/[ll]^{\nu} 
}
\]
as follows. 

In order to define $\Xi$, let $\Delta$ be an oriented triangle in 
$\MT_{bal}^{\pm}(2m+1)$,
and fix a developing map $\iota$ for $\Delta$ that sends its circumcentre $v$ to $O\in\Sph$.
Extend $\iota$ to the universal cover of the torus $T(\Delta)$, which has 
a $\sigma$-invariant metric $h$, and is given a $2$-marking as in Construction \ref{fund2}.
For every $t\in\RR$ the map 
$e^{itR}\circ \iota:\widetilde{T}\rightarrow\Sph$ has the same equivariance
of $\iota$, and so the pull-back of the metric of $\Sph$ via such map descends to a spherical
metric $h_t$ on $T$.
We then define $\Xi(T,x,\bm{p},v,h,t):=(T,x,\bm{p},v,h_t)$.

In order to define $\nu$, consider
a $2$-marked decorated spherical torus $(T,x,\bm{p},v,\hat{h})$,
whose metric 
$\hat{h}$ is not necessarily invariant under the conformal involution $\sigma$.
Its developing map $\iota:\widetilde{T}\rightarrow\Sph$ has monodromy
contained in a $1$-parameter subgroup that fixes $O=\iota(\tilde{v})$, where $\tilde{v}$ is a lift of $v$,
and a maximal circle $E$.
Note that points in $e^{itR}\cdot E$ sit at constant distance $\arctan(2e^{-t})$ from $O$
and that the distance from $O$ corresponds to the distance function $d_v:T\rightarrow[0,\pi]$ from
the vertex $v$.
Thus we also have the function $t=-\log\tan(d_v/2):T\rightarrow [-\infty,\infty]$.
We remark that a developing map of $\sigma^*(T,x,\bm{p},v,\hat{h})$
can be obtained by post-composing $\iota$ with an isometry of $\Sph$ that exchanges $O$ with $-O$.
Hence, $t\circ\sigma^*=-t$.
It follows that $\hat{h}$ is $\sigma$-invariant if and only if $t(x)=0$,
namely $\iota(\tilde{x})\in E$ for any lift $\tilde{x}$ of $x$.
It is easy to see that 
modified developing map $e^{-it(x)R}\circ\iota$ has the same invariance as $\iota$ and
sends $\tilde{x}$ to $E$. Hence, the round metric
on $\Sph$ pulls back and descends to a $\sigma$-invariant metric $h$ on $T$.
We define $\nu(T,x,\bm{p},v,\hat{h}):=\Delta^{(2)}(T,x,\bm{p},v,h,t(x))$.
\end{fundamental}

Before proceeding, we need a very simple lemma.

\begin{lemma}[Lipschitz constant of projective transformations]\label{proj-Lip}
For every $t\in\RR$, the transformation $e^{itR}$ of $\Sph$ has (bi)Lipschitz constant $\cosh(t)$.
Moreover, along the maximal circle $E$ 
it has Lipschitz constant $1/\cosh(t)$.
%achieved at $E$ and at $e^{-itR}\cdot E$.
\end{lemma}
\begin{proof}
If $O$ is the origin of $\mathbb{C}$ and $\frac{2|dz|}{1+|z|^2}$ is the spherical line element,
then the transformation $e^{itR}$ can be written as $z\mapsto e^{-t}z$.
Through the map $e^{itR}$ the metric decreases the most at $E=\{|z|=1\}$,
where the Lipschitz constant is exactly $1/\cosh(t)$.
\end{proof}

The first fact about Construction \ref{fund4} is the following.

\begin{proposition}[The homeomorphism $\Xi$]\label{Xi-homeo}
The map $\Xi$ is a homeomorphism and $\nu$ is its inverse.
\end{proposition}
\begin{proof}
It is routine to check that the maps $\Xi$ and $\nu$ are set-theoretic inverse of each other.
Note that the restriction of $\Xi$ to $\MT_{bal}^{\pm}(2m+1)\times\{0\}$
is a homeomorphism by Theorem \ref{int-MS-2mark} (i).
Hence, the continuity of $\Xi$ follows from Lemma \ref{proj-Lip}.

To show that $\Xi$ is proper, consider a diverging sequence in
$\MT_{bal}^{\pm}(2m+1)\times\RR$, which we can assume to be contained in a fixed
connected component. 
By Proposition \ref{int-or-bal}
an element of such component can be identified by a quadruple  $(l_1,l_2,l_3,t)$
with $0<l_i<2\pi$ and $l_1+l_2+l_3=2\pi$.
A sequence of quadruples 
diverges if and only if some $\bar{l}_i\rar 0$ or if $|t|\rar\infty$ (up to subsequences).
Since the systole of the triangle corresponding to $(l_1,l_2,l_3)$ is
$\min\{\bar{l}_i\}$ by Lemma \ref{trianglesystole},
the systole of the torus $\Xi(l_1,l_2,l_3,t)$ is at most
$\min\{\bar{l}_i\}/\log\cosh(t)\rar 0$ by Lemma \ref{proj-Lip}.
It follows that $\Xi$ sends diverging sequences to diverging sequences
by Theorem \ref{propersys}.
\end{proof}

Since $\MSPH_{1,1}^{(2)}(2m+1)$ is a manifold by Proposition \ref{Xi-homeo},
it can be endowed with an orbifold structure as in Remark \ref{orbifold-correct}.
We then have the following preliminary result.

\begin{theorem}[Moduli space of $2$-marked tori of area $4m\pi$]\label{int-MS-2mark'}
For $m>0$ integer, the moduli space $\MSPH^{(2)}_{1,1}(2m+1)$ of $2$-marked tori
with spherical metric of area $4m\pi$ has the following properties.
\begin{itemize}
\item[(i)]
As an orbifold, it is isomorphic to the quotient of $\MT_{bal}^{\pm}(2m+1)\times\RR$
by the action of the involution $\sigma^*$ that flips the sign of the $\RR$-factor. 
Hence it consists of $m(m+1)$ components isomorphic to $\twosimpl\times(\RR/\{\pm 1\})$.
\item[(ii)]
The locus in $\MSPH_{1,1}^{(2)}(2m+1)$ of metrics that are invariant under the conformal involution $\sigma$ corresponds
to $\MT_{bal}^{\pm}(2m+1)\times\{0\}$.
\item[(iii)]
The group $\Sy_3$
that permutes the $2$-torsion points of the torus acts trivially on $\RR$
and as in Proposition \ref{int-or-bal} on $\MT_{bal}^{\pm}(2m+1)$.
\end{itemize}
\end{theorem}
\begin{proof}
(i) The action of $\sigma$ is described in Construction \ref{fund4}. The claim
follows from Theorem \ref{int-MS-2mark} (i) and Proposition \ref{Xi-homeo}.

(ii) is also clear by Construction \ref{fund4}.

(iii) follows by noting that relabelling the $2$-torsion points does not affect the decoration.
\end{proof}

Finally we can prove Theorem \ref{mainodd}.

\begin{proof}[Proof of Theorem \ref{mainodd}] 
The forgetful map $\MSPH_{1,1}^{(2)}(2m+1)\rar\MSPH_{1,1}(2m+1)$
is an unramified $\Sy_3$-cover of orbifolds.
By Proposition \ref{int-MS-2mark'} such quotient can be identified to
$(\MT_{bal}(2m+1)\times\RR)/(\ASy_3\times\ZZ_2)$, where
$\ZZ_2$ acts by flipping the sign of the $\RR$-factor
and the alternating group $\ASy_3$ acts by cyclically relabelling the vertices of the triangles. 

The rest of argument is entirely analogous to the one used in the proof of Theorem \ref{mainodd2}.
\end{proof}

%\begin{remark} Take $S^1$ of length $2\pi$ and consider the space $M$ of ordered triples $x_1,x_2,x_3\in S^1$ with $x_i\ne x_j$. Then $M/S^1$ is an open disk. One can compactify it to a disk with three marked points. These marked three points correspond to the situation when $x_1,x_2,x_3$ coincide. 
%\end{remark}

%%%%%%%%%%%%%%%%%%%%%%%%%%%%%%%%%%%%%%%%%%%%

\section{$\MSPH_{1,1}(2m)$ and $\MSPH_{1,1}^{(2)}(2m)$ as Belyi curves}\label{sec:Belyi}

The goal of this section is to identify the moduli spaces $\MSPH_{1,1}(2m)$ and $\MSPH_{1,1}^{(2)}(2m)$ with Belyi curves and relate their cell decompositions constructed in Corollary \ref{compactcell} with the corresponding dessins. 
We recall \cite[Section 2]{Lin2} and \cite{ET} that these two spaces have a {\it canonical complex structure}. This structure is the unique one with respect to which the forgetful map to $\Mcal_{1,1}$ and $\Mcal_{1,1}^{(2)}$ are holomorphic. 
We also recall that the compactification $\overline{\MSPH}_{1,1}^{(2)}(2m)$ 
obtained from $\MSPH_{1,1}^{(2)}(2m)$ by filling it the $3m$ punctures
has the orbifold structure
that makes it isomorphic to the quotient of its underlying topological space (which is
in fact a Riemann surface) by the trivial $\ZZ_2$-action.
The respective forgetful maps extend to the smooth compactifications of all the four orbifolds.

%Let us now recall some definitions.
The following definition slightly differs from the usual definition of dessin d'enfant, though it is very similar in spirit.

\begin{definition}[Belyi functions and dessins]
A \emph{Belyi function} is a holomorphic map 
$\psi:S\rightarrow\mathbb{C}P^1$ from a compact Riemann surface $S$ to the complex projective line $\mathbb CP^1$, ramified only over points $0, 1, \infty$. 
The {\it{dessin}} associated to $\psi$ is the $3$-partite graph embedded in $S$
obtained as the preimage
of the real line $\mathbb RP^1\subset \mathbb CP^1$ under $\psi$.
% (or a \emph{dessin d'enfant}). 
\end{definition}

The dessin of $\psi$ can also be seen as the $1$-skeleton of the triangulation of $S$ whose open cells are the preimages through $\psi$ of the two open disks in which  $\mathbb RP^1$ cuts $ \mathbb CP^1$.

The main result of this section is the following theorem, which concerns
the underlying Riemann surface $\overline{\MSPH}_{1,1}^{(2)}(2m)$.

\begin{theorem}[The topological space $\overline{\MSPH}_{1,1}^{(2)}(2m)$ as a Belyi curve]\label{belyitheorem} Let $m$ be a positive integer. 
%Consider the smooth compactification $\overline{\MSPH}_{1,1}^{(2)}(2m)$ of $\MSPH_{1,1}^{(2)}(2m)$ with its canonical complex structure. 
Then
there is a holomorphic Belyi map $\BEL:\overline{\MSPH}_{1,1}^{(2)}(2m)\rightarrow\mathbb CP^1$ of degree $m^2$ from the Riemann surface underlying $\overline{\MSPH}_{1,1}^{(2)}(2m)$ with the following properties.
\begin{itemize}
\item[(i)] The preimage of $\mathbb CP^1\setminus \{0,\,1,\,\infty\}$ under $\BEL$ coincides with  the Riemann surface $\MSPH_{1,1}^{(2)}(2m)$.

\item[(ii)] The cycle type of ramification of $\BEL$ over points $\{0,\,1,\,\infty\}$ is $(1,\,3,\,\ldots,\,2m-1)$.

\item[(iii)] The dessin of $\BEL$ is composed of tori $T$ such that the  triangle $\Delta(T)$ has one integral angle. In particular, the triangulation given by this dessin is the one described in Corollary \ref{compactcell}. 
\end{itemize}
\end{theorem}

\begin{definition}[Klein group and Klein sphere]
The {\it{Klein group $K_4$}} is the subgroup of diagonal matrices in $\mathrm{SO}(3,\RR)$. 
The {\it{Klein sphere}} $S_{Kl}$ is the sphere with three conical points $y_1,\,y_2,\,y_3$ of angles $(\pi,\pi,\pi)$ obtained by taking the quotient  
of the unit sphere $\Sph$ by the action of $K_4\cong\ZZ_2\oplus\ZZ_2$.
We denote by $S_{Kl}(\RR)$ the circle in $S_{Kl}$
which is invariant under the unique anti-holomorphic isometric involution of $S_{Kl}$.
\end{definition}

Using the conformal structure on $S_{Kl}$ given by the spherical metric, 
we can view $S_{Kl}$ as $\mathbb CP^1$, where $y_1=0,\,y_2=1,\,y_3=\infty$, and
$S_{Kl}(\RR)$ as $\RR P^1$.  

\begin{remark}[Klein sphere as a doubled triangle]\label{simplepillowprop} We note that $S_{Kl}$ can also be obtained by doubling of the spherical triangle $\Delta$ with angles $\pi,\,\pi,\,\pi$ across its boundary. This way $\partial \Delta$ corresponds to the circle $S_{Kl}(\mathbb R)$ in $S_{Kl}$. 
Recall that, in the triangle with three angles $\pi$, each vertex is at distance exactly $\frac{\pi}{2}$ from each point of the opposite side. For this reason, points of $S_{Kl}(\mathbb R)$ are exactly the points on $S_{Kl}$ that are at distance $\frac{\pi}{2}$ from one conical point.
\end{remark}

The key result to parametrize spherical tori using a Hurwitz space is the following.

\begin{proposition}[Tori of area $(2m-1)\pi$ cover the Klein sphere]\label{coverconstruct}
Let $(T,x)$ be a spherical torus with a conical point of angle  $4\pi m$ and with points of order $2$ marked by $p_1,\, p_2,\, p_3$. There exists a unique branched cover map $\varphi_{Kl}:T\to S_{Kl}$
of degree $4m-2$, which is a local isometry outside of branching points, and such that $\varphi_{Kl}(p_i)=y_i$. Moreover $\varphi_{Kl}(x)\neq y_i$ for $i=1,2,3$.
\end{proposition}
\begin{proof} We will first construct the map and then will count its degree. Recall \cite[Proposition 1.5.1]{Lin2}, that the image of the monodromy map $\rho: \pi_1(T,x)\to \mathrm{SO}(3,\RR)$ is the Klein group (see also Corollary \ref{evenklein}).
Consider the developing map $\iota:\widetilde T\to \mathbb S^2$ from the universal cover $\widetilde T$ of $T$. This map is equivariant with respect to the action of $\pi_1(T,x)$ on $\widetilde T$ by deck transformation and on $\mathbb S^2$ by the monodromy representation. Hence, by taking the quotient we get a  map $\varphi_{Kl}:T\to S_{Kl}\cong \mathbb S^2/K_4$.

We will now prove that the constructed map $\varphi_{Kl}$ sends points $p_i$ to the three distinct orbifold points of $S_{Kl}$. This will permit us to label these three points so that $\varphi_{Kl}(p_i)=y_i$. In order to do this, consider the order two automorphims $\sigma$ of $T$ and denote by $S$ the quotient $T/\sigma$. The surface $S$ is a sphere with three conical points of angle $\pi$, that are the images of the points $p_i$, and one conical point of angle $2\pi m$. Let us take 
a lift $\tilde{x}\in\widetilde{T}$ of $x$ and let $\widetilde\sigma$  be the lift of $\sigma$ to $\widetilde T$ that fixes $\tilde{x}$.
Since the conical angle at $x$ is an even multiple of $2\pi$, the maps
$\iota$ and $\iota\circ\widetilde{\sigma}$ coincide in a neighbourhood
of $\tilde{x}$. 
%
%It is not hard to see, that the maps $\iota$ and $\iota\widetilde\sigma$ from $\widetilde T$ to $\mathbb S^2$ coincide. Indeed, this follows from the fact that the conical angle at $x$ is an even multiple of $2\pi$. 
It follows that $\iota$ is $\widetilde{\sigma}$-invariant
and the map $\varphi_{Kl}$ descends to a map $\varphi_{Kl}':S\to S_{Kl}$. Now, by construction, the map $\varphi_{Kl}'$ is a local isometry outside of ramification points. This  implies that all three conical points of angle $\pi$ on $S$ are sent by $\varphi_{Kl}'$ to  conical points of angle $\pi$ on $S_{Kl}$. Finally, to see that the images of the three conical points are distinct, we use the fact that the monodromy of $S$ is generated by three loops winding simply around these points, and it is isomorphic to $K_4$. Hence, we proved that points $\varphi_{Kl}(p_i)$ in $S_{Kl}$ are the three distinct conical points of $S_{Kl}$, and so we can label each  $\varphi_{Kl}(p_i)$ by $y_i$. This finishes the construction of the map. Its uniqueness is clear.

To prove that $\deg(\varphi_{Kl})=4m-2$, we use the fact that $\varphi_{Kl}$ is a local isometry outside of branching points. So we have 
$\deg(\varphi_{Kl})=\area(T)/\area(S_{Kl})=2\pi (2m-1)/\pi=4m-2$.
Finally, if $\varphi_{Kl}$ mapped $x$ to some $y_i$, its local degree at $x$
would be $(4m\pi)/\pi=4m>\deg(\varphi_{Kl})$. This contradiction proves the last claim.
\end{proof}

Here is already a first corollary of the above proposition.

\begin{corollary}[Moduli space of $2$-marked tori as a Hurwitz space]\label{hurwitz}
As an orbifold, the moduli space $\MSPH_{1,1}^{(2)}(2m)$ is isomorphic to the Hurwitz space 
$\Hur_m$ of connected degree $4m-2$ covers, ramified over points $0,1,\infty$
with cyclic type $(2,\ldots, 2)$ and over  $\lambda\ne 0, \,1,\, \infty$
 with cyclic type $(1,\ldots, 1,\, 2m)$.
\end{corollary}
\begin{proof} To construct the map $\MSPH_{1,1}^{(2)}(2m)\rar \Hur_m$, we use Proposition \ref{coverconstruct}, that associates to each spherical torus $(T,x)$ with a $2$-marking
the branched cover $\varphi_{Kl}: T\to S_{Kl}$. 
Using the conformal structure on $S_{Kl}$ given by the spherical metric, we view it as $\mathbb CP^1$, where $y_1=0,\,y_2=1,\,y_3=\infty$. By Proposition \ref{coverconstruct} we know that $\lambda=\varphi_{Kl}(x)\ne 0,\,1,\, \infty$. To find the cyclic type of ramification over points $(0,\,1,\,\infty,\,\lambda)$, we recall that the map $\varphi_{Kl}$ is a local isometry outside of branching locus, and so for each preimage of the points $0$, $1$, $\infty$ the map has branching of order $2$. Finally, there is only one conical point in the preimage of $\lambda$, hence the cyclic type over $\lambda$ is $(1,\ldots, 1,\, 2m)$.

The inverse map $\Hur_m\rightarrow\MSPH_{1,1}^{(2)}(2m)$ is the following. For each ramified cover $T\rightarrow \mathbb CP^1\cong S_{Kl}$ with the prescribed  cyclic type, we pull-back  the spherical metric of $S_{Kl}$ to $T$. By Proposition \ref{coverconstruct} the $2$-torsion points of $T$ are mapped
to $y_1,y_2,y_3$ and we call $p_i$ the unique $2$-torsion point of $T$ that is sent to $y_i$. 
\end{proof}

This corollary has the following immediate consequence.

\begin{corollary}[$\MSPH_{1,1}^{(2)}(2m)$ covers the $3$-punctured sphere]\label{ramifiedconstr} 
There is an unramified cover $\BEL: \MSPH_{1,1}^{(2)}(2m)\rightarrow\mathbb CP^1\setminus \{0,1,\infty\}$ of Riemann surfaces
that is holomorphic with respect to the canonical holomorphic structure on 
the topological space $\MSPH_{1,1}^{(2)}(2m)$.
\end{corollary}
\begin{proof} 
Let us first construct the cover. 
As in the proof of Proposition \ref{coverconstruct},
to each torus $T$ with conical angle $4\pi m$ we associate the ramified cover $\varphi_{Kl}: T\to S_{Kl}\cong \mathbb CP^1$. We know that $\varphi_{Kl}(x)\ne 0,\,1,\infty$. So we set $\BEL(T)=\varphi_{Kl}(x)$. 
It is easy to see that this map is an unramified cover of $\mathbb CP^1\setminus \{0,1,\infty\}$.

To justify the holomorphicity statement, let us pull-back the complex structure on $\mathbb CP^1\setminus \{0,1,\infty\}$ to the surface $\MSPH_{1,1}^{(2)}(2m)$. We claim that this pulled-back structure coincides with the canonical one. Indeed the canonical structure is the unique structure with respect to which the forgetful map to the Riemann surface $\Mcal^{(2)}_{1,1}$ is holomorphic. However, it is also clear that the
modulus of the torus $\BEL^{-1}(\lambda )$ is a holomorphic function of the local parameter $\lambda$ on  $\mathbb CP^1\setminus \{0,1,\infty\}$.
\end{proof}

We need one last lemma.

\begin{lemma}[Dessin of $\BEL$]\label{cellequivalence} 
A torus $T$ in the topological space $\MSPH_{1,1}^{(2)}(2m)$ belongs to the dessin of $\BEL$ if and only if the balanced triangle $\Delta(T)$ has one integral angle. 
\end{lemma}

\begin{proof} Let us prove the ``if'' direction. Suppose that $\Delta$ has an integral angle. Then it has one side of length $\pi$. This means that for some $i$ the distance on $T$ from $x$ to $p_i$ is $\pi/2$. This means that the distance on $S_{Kl}$ between $y_i$ and  $\varphi_{Kl}(x)$ is $\pi/2$. Using Remark \ref{simplepillowprop}, we deduce  that $\varphi_{Kl}(x)$ belongs to $S_{Kl}(\mathbb R)$. By the definition of the dessin of $\BEL$ we see that $T$ belongs to the dessin. 

Let us now prove the  ``only if'' direction. Suppose that $\varphi_{Kl}(x)$ belongs to $S_{Kl}(\mathbb R)$. For example, assume $\varphi_{Kl}(x)\in y_1y_2$. Let $\gamma_3$ be the geodesic loop on $T$ based at $x$, whose midpoint is $p_3$. Since half of this geodesic is projected by  $\varphi_{Kl}$ to the segment that joins $y_3$ with the segment $y_1y_2$, we see that $|\gamma_3|=\pi$. From Lemma \ref{sidepi}, it follows that the angle of $\Delta$ opposite to $\gamma_3$ is integral.
\end{proof}

\begin{proof}[Proof of Theorem \ref{belyitheorem}] 
(i) The ramified cover is the extension
of the cover constructed in Corollary  \ref{ramifiedconstr} to the compactified spaces.

(iii) This is proven in Lemma \ref{cellequivalence}.

(ii) Recall that the topological space underlying $\MSPH_{1,1}^{(2)}(2m)$ is a complex curve with  exactly $3m$ punctures. 
Let $q$ be one of such punctures.
Since $S_{Kl}$ is obtained by gluing two spherical triangles along $S_{Kl}(\RR)$,
each $y_i$ is adjacent to two triangles.
It follows that the  ramification degree of $\BEL$ at $q$ is equal to half the number of triangles of the dessin adjacent to $q$. Since $\MSPH_{1,1}^{(2)}(2m)$ is glued from  two copies of $\MT_{bal}(\th)$ it is easy to check that the corresponding number can be any odd number from $1$ to $2m-1$. 
\end{proof}

Finally, we can prove Theorem \ref{eventheta}.

\begin{proof}[Proof of Theorem \ref{eventheta}] 
To prove this result we will realise $\MSPH_{1,1}(2m)$ is an unramified orbifold cover of the modular curve 
$\mathbb H^2/\mathrm{SL}(2,\mathbb Z)$. Recall that in Theorem \ref{belyitheorem} we constructed the unramified covering map 
$\BEL$ of degree $m^2$ from the topological space
$\MSPH_{1,1}^{(2)}(2m)$ to $\mathbb CP^1\setminus \{0,\,1,\,\infty\}$. 
Note that the quotient of $\mathbb{C}P^1\setminus\{0,1,\infty\}$ by the trivial $\ZZ_2$-action
is an orbifold isomorphic to $\mathbb{H}^2/\Gamma(2)$, where $\Gamma(2)=\{A\in\mathrm{SL}(2,\ZZ)\,|\,A\equiv I \pmod 2\}$. 
Hence, the above cover can be promoted to an unramified cover of orbifolds
$\MSPH_{1,1}^{(2)}(2m)\rightarrow \mathbb{H}^2/\Gamma(2)$ of the same degree.

The symmetric group $\Sy_3$ is acting on $\MSPH_{1,1}^{(2)}(2m)$ by relabelling the $2$-torsion points of the tori and it acts on $\mathbb{H}^2/\Gamma(2)$ through the isomorphism
$\Sy_3\cong \mathrm{SL}(2,\ZZ)/\Gamma(2)$.

% and it is acting on $\mathbb CP^1\setminus \{0,\,1,\,\infty\}$. It is easy to see that the covering map $\BEL$ is equivariant with respect to this action and 

Since $\MSPH_{1,1}^{(2)}(2m)/S_3=\MSPH_{1,1}(2m)$ as orbifolds,
the covering map then descends to an unramified orbifold covering 
$\MSPH_{1,1}(2m)\rightarrow \mathbb{H}^2/\mathrm{SL}(2,\ZZ)$ of degree $m^2$.
Note that the cycle type ramification of such cover at infinity 
is $(1,\, 3,\,\dots,\, 2m-1)$ by Theorem \ref{belyitheorem} (ii). It follows that, for $m>1$, such cover
is not Galois and so $G_m$ is not a normal subgroup.

The last claim follows from Theorem \ref{belyitheorem} (iii) after noting that
the real locus $\RR P^1\setminus\{0,1,\infty\}$ inside
$\mathbb{C}P^1\setminus\{0,1,\infty\}\cong \mathbb{H}^2/\Gamma(2)$
descends to $\{[it]\ \text{with $t\geq 1$}\}$ inside $\mathbb{H}^2/\mathrm{SL}(2,\ZZ)$.
%It remains to note that the latter space is isomorphic to the modular curve.
\end{proof}

\section{Lipschitz topology on $\MSPH_{g,n}$}\label{secLipschitz}

In this section we define a natural topology on the set of spherical surfaces with conical singularities and establish some of its basic properties. We choose the approach using Lipschitz distance, described, for example in  \cite[Example, page 71]{Gro}.

Let us first recall the definition of Lipschitz distance between two marked metric spaces.

\begin{definition}\label{def:Lipschitz} Let $(X, x_1,\ldots,x_n; d_X)$ and $(Y,y_1,\ldots, y_n; d_Y)$ be two  metric spaces with distinct marked points $x_i, y_i$.
The \emph{Lipschitz distance} between them is defined by

$$d_{\mathcal L}((X,\bm x), (Y,\bm y)) = \inf_f \log(\max\{\mathrm{dil}(f), \mathrm{dil}(f^{-1})\}),$$
where
$$\mathrm{dil}(f) = \sup_{p_1\ne p_2 \in X} \frac{ d_Y(f(p_1), f(p_2)) } {d_X(p_1, p_2)}$$
and the infimum runs over bi-Lipschitz homeomorphisms between $X$ and $Y$ that send each $x_i$ to $y_i$. The value $\max\{\mathrm{dil}(f), \mathrm{dil}(f^{-1})\}$ is called the \emph{bi-Lipschitz constant} of the map $f$.

Furthermore, we say that a map $f:X\to Y$ is a \emph{bi-Lipschitz embedding} with constant $c\ge 1$ if for any two points $x_1,x_2$ we have 
$$c^{-1}\cdot d_Y(f(x_1),f(x_2))\le d_X(x_1,x_2)\le c\cdot d_Y(f(x_1),f(x_2)).$$
\end{definition}

We will denote by $\MSPH_{g,n}$ the space of genus $g$ surfaces with $n$ marked conical points up to a marked marked isometry. By $\MSPH_{g,n}\lA$ we denote the subspace of surfaces with area bounded by $A>0$. To state the main two results of this section we recall the notion of the systole of a spherical surface.
%\begin{theorem}\label{locompact}  $\MSPH_{g,n}$ is a locally complete metric space with respect to Lipschitz distance. 
%\end{theorem}

\begin{definition}[Systole]\label{sysdefinition}   The \emph{systole} $\sys(S)$ of a spherical surface $S$ is the half length of the shortest geodesic segment or geodesic loop on $S$ whose end points are conical points of $S$. 

The \emph{systole} $\sys(P)$ of a spherical polygon $P$ is the minimum  of half-distances between all vertices of $P$ and the distances between a vertex of $P$ with  the unions of edges not adjacent to the vertex. Such a systole is clearly equal to the systole of the sphere obtained by doubling $P$ along its boundary.
%Tthe systole of the sphere with conical points that is obtained by taking the double of $P$. 
\end{definition}

Let $\MSc\lA$ be the subspace of $\MSPH_{g,n}\lA$ of surfaces with systole at least $s$.

\begin{theorem}\label{propersys} $\MSPH_{g,n}$ is a complete metric space with respect to Lipschitz distance.  The function $\sys(S)^{-1}$ is proper on $\MSPH_{g,n}\lA$ in Lipschitz topology. 

\end{theorem}

Let us denote by $\MP_n$ the space of all spherical polygons with $n$ cyclically labelled vertices up to isometries that preserve the labelling. We have the following similar result.

\begin{corollary}\label{Lippolygon} The space $\MP_n$ of spherical polygons with $n$ vertices is complete with respect to Lipschitz distance. For any positive $A>0$ the function $\sys^{-1}(P)$ is proper on the subset $\MP_n$ of polygons with area at most $A$.
\end{corollary}

%\begin{definition} Let $(S,x_1,\ldots, x_n)$ be a genus $g$ surface with $n$ marked points. Let $G$ be a finite group acting on $(S,x_1,\ldots, x_n)$, preserving labelling and preserving an orientation. 
%\end{definition}
%
%\begin{corollary}\label{symmetryLip} Let $G$ be a finite group with 
%\end{corollary}

%We prove Theorem \ref{locompact}  in Section \ref{seclocalcompl}. 

To prove Theorem \ref{propersys}, we show that surfaces from  $\MSc\lA$ admit triangulations into a finite number of {\it relatively large} triangles.  
This is done in Theorem \ref{boundedtriangulation} which itself relies on Delaunay triangulations, constructed in Theorem \ref{Delaunay}. The proof of Corollary \ref{Lippolygon} is similar.

As an application of Theorem \ref{propersys} and Corollary \ref{Lippolygon}, we will get the following result on topology of the space $\MSPH_{1,1}^{(2)}(\th)$ of $2$-marked tori, induced by Lipschitz metric $\cal L$. 

\begin{theorem}\label{torusbalancehomeo} 
\begin{itemize}
\item[(i)] Suppose $\th$ is not odd. The map $ T^{(2)}:\MT_{bal}^{\pm}(\th) \to (\MSPH_{1,1}^{(2)}(\th),{\cal L})$ is a homeomorphism of surfaces.
\item[(ii)] Let $m$ be a positive integer. The map $ T^{(2)}:\MT_{bal}^{\pm}(2m+1) \to (\MSPH_{1,1}^{(2)}(2m+1)^{\sigma},{\cal L})$ is a homeomorphism of surfaces.
 \end{itemize}
\end{theorem}
Recall, that the bijective map $T^{(2)}$ was defined in Construction \ref{fund2}, whereas the Lipschitz distance between two $2$-marked tori is measured among maps that preserve $2$-marking.

\subsection{Lipschitz metric and its basic properties}\label{lipschitzprop}

Here we collect basic results concerning Lipschitz metric with an emphasis on spherical surfaces.
 
\begin{lemma} Lipschitz distance defines a metric on the space $\MSPH_{g,n}$ of spherical surfaces of genus $g$ with $n$ conical points.
\end{lemma}
\begin{proof}  Let   $S_1$ and $S_2$ be genus $g$ spherical surfaces with $n$ conical points. Let's show that $d_{\mathcal L}(S_1,S_2)<\infty$, i.e., that there is a bi-Lipschitz map $\varphi: S_1\to S_2$. Take a map $\varphi: S_1\to S_2$ that is a diffeomorphism from $\dot S_1$ to $\dot S_2$, sends each conical point $x_i\in S_1$ to the corresponding point $x_i'\in S_2$, and is {\it radial} \footnote{The map is called radial if it sends geodesics through a conical point isometrically to geodesics and scales the angles between geodesics by a constant factor.} in a neighbourhood of each $x_i$. Such a map is clearly bi-Lipschitz.

Next note that $d_{\mathcal L}(S_1,S_2)=0$ if and only if $S_1$ and $S_2$ are isometric by \cite[Theorem 7.2.4]{BBI}. 

All the other properties of the metric are obvious.
\end{proof}

\begin{definition} 
The {\it{Lipschitz topology}} on the moduli space $\MSPH_{g,n}$ of spherical surfaces is the topology induced by the Lipschitz metric.
\end{definition}

The next lemma explains how difference in the values of conical angles of two surfaces affects the Lipschitz distance between them. 

\begin{lemma}[Continuity of angle functions]\label{contangles} Let $U_1, U_2$ be neighbourhoods of conical points $x_1, x_2$ with conical angles $\th_1,\th_2$. Suppose $f: U_1\to U_2$ is a bi-Lipschitz homeomorphism. Then  
\begin{equation}\label{inenquality:dilatation}
\max\{\mathrm{dil}(f), \mathrm{dil}(f^{-1})\}\ge \max\left (\frac{\th_1}{\th_2},\frac{\th_2}{\th_1}\right)^{\frac{1}{2}}. 
\end{equation} 

In particular, functions $\th_i$ are continuous on $\MSPH_{g,n}$ in  Lipschitz topology.
\end{lemma}
\begin{proof} After scaling by a large constant and passing to the limit, we can assume that the metrics on $U_1$ and $U_2$ are flat,  moreover both $U_1$ and $U_2$  are flat cones with conical angles $2\pi \th_1, 2\pi\th_2$ correspondingly. Note, that as a result the limit quantity $\max\{\mathrm{dil}(f), \mathrm{dil}(f^{-1})\}$   can only decrease. Replacing $f$ by $f^{-1}$ if necessary, we can assume that $\th_1\le \th_2$. 
 
Let us now reason by contradiction. Assume that Inequality (\ref{inenquality:dilatation}) is not satisfied. Consider the radius  $1$ circle $S^1$ centred at $x_1$ on $U_1$. Since $\mathrm{dil}(f^{-1})<(\th_2/\th_1)^{\frac{1}{2}}$, 
%Since (\ref{inenquality:dilatation}) is not satisfied, 
the image $f(S^1)$ lies at distance $c$ from $x_2$, where $c>(\th_1/\th_2)^\frac{1}{2}$. Hence, $l(f(S^1))\ge 2\pi c\th_2$. At the same time we have 
$$\mathrm{dil}(f)\ge \frac{l(f(S^1))}{l(S^1)}=\frac{l(f(S^1))}{2\pi\th_1}\ge \frac{2\pi c\th_2 }{2\pi\th_1}>\left(\frac{\th_2}{\th_1}\right)^{\frac{1}{2}}.$$ 
This contradicts our assumption.
\end{proof}

%\begin{corollary} Let $S_1$ and $S_2$ be two spherical surfaces from $\MSPH_{g,n}(\th)$. Suppose that $d=d_{\cal L}(S_1,S_2)<\frac{1}{2}\min_i |\log(\th_i)|$. Then there is a bi-Lipschitz  map $f: S_1\to S_2$ with bi-Lipschitz constant at $e^d$ that sends conical points to conical points.
%\end{corollary}

\begin{lemma}[Continuity of systole function]\label{systoleratio} Let $(S_1,h_1)$, $(S_2,h_2)$ be two spherical surfaces from $\MSPH_{g,n}$ such that $d_{\cal L}(S_1,S_2)=d$. Then
$$e^d\sys(S_1,h_1)\ge \sys(S_2,h_2)\ge e^{-d}\sys(S_1,h_1).$$
In particular, the function $\sys(S,h)$ is continuous on $\MSPH_{g,n}$ in  Lipschitz topology.
\end{lemma}
\begin{proof} Let $S$ be a spherical surface with conical points $x_1,\ldots,x_n$. According to \cite{MP:systole}, $\sys(S)$ is equal to the minimum of  half distances between conical points, and half-lengths of all (rectifiable) simple  loops based at some conical point $x_i$, contained in $\dot S\cup x_i$ and non-contractible in $\dot S\cup x_i$.
Any bi-Lipschitz homeomorphsim $f$ from $S_1$ to $S_2$ that sends points $x_i$ to points $x_i$ also sends rectifiable loops based at $x_i$ to rectifiable loops based at $x_i$. By definition there for any $\varepsilon>0$ exists a homemorphism $f_{\varepsilon}: S_1\to S_2$ with bi-Lipschitz constant $e^{d+\varepsilon}$. This clearly explains the above inequalities. 
\end{proof}

\subsection{Injectivity radius}\label{seclocalcompl}
%\subsection{Local completeness of $\MSPH_{g,n}$, proof of Theorem \ref{locompact}}\label{seclocalcompl}

%In this section we prove Theorem \ref{locompact}. To do this we need some control on the {\it injectivity radius} of points on spherical surfaces in terms of the value of the Voronoi function and the systole of the surface. This is worked out in Proposition \ref{injcontrol}.
Here we prove Proposition \ref{injcontrol}
that gives an estimate on the injectivity radius of points on spherical surfaces in terms of the value of the Voronoi function and the systole of the surface.

\begin{definition} Let $S$ be a spherical surface and $y\in \dot S$ be a non-conical point.
The {\it{injectivity radius}} $\inj(y)$ is the supremum of $r$ such that $S$ contains an isometric copy of a spherical disk of radius $r$, embedded in $S$ and centred at $y$. 

For a conical point $x_i\in S$ the injectivity radius is defined to be the minimum of all distances from $x_i$ to other conical points, and half lengths of geodesic loops based at $x_i$.
\end{definition} 

\begin{proposition}\label{injcontrol} Let $S$ be a spherical surface with conical angles $2\pi(\th_1,\ldots, \th_n)$. Then for any $y\in \dot S$ we have 
\begin{equation}\label{injinequality}
\inj(y)\ge\min(\sys(S), \Vor_S(y),\min_i \th_i\cdot \Vor_S(y))
\end{equation}
Moreover, the following statements hold.
\begin{itemize}

\item[(i)] If $\inj(y)<\Vor_S(y)$, then there exists a closed geodesic loop $\gamma\subset \dot S$  of length $2\inj(y)$, based at $y$. Moreover $l(\gamma)=2\inj(y)<\pi$. 
\item[(ii)] If $\Vor_S(y)>\frac{\pi}{2}$ then $\inj(y)=\Vor_S(y)$.  
\item[(iii)] Suppose $\inj(y)<\Vor_S(y)$ and so $\Vor_S(y)\le \frac{\pi}{2}$. Then, at least one of the following holds
\begin{itemize}
\item[(a)] $\inj(y)>\sys(S)$.
\item[(b)] There exists $i$ such that $\th_i<\frac{1}{2}$ and  $\inj(y)>\min_i \th_i\cdot \Vor_S(y)$.
\end{itemize}
\end{itemize}

\end{proposition}

We will need one  lemma to prove this result.

\begin{lemma}\label{involDisk} Let $D$ be a spherical disk with one conical point $x$ in its interior. Suppose that the boundary $\gamma$ of $D$ satisfies $\ell(\gamma)<2\pi$, and $\gamma$ is a geodesic loop
with a unique non-smooth point $y$. Then there exits an orientation reversing, isometric involution $\tau$ on $D$. 
\end{lemma}
\begin{proof} Note first, that the angle at $x$ is non-integer, otherwise the univalent developing map from $D$ to $\mathbb S^2$ would send $\gamma$ onto a great circle.
Consider the sphere $S$ obtained from $D$ by doubling along $\gamma$, and denote by $\tau_{\gamma}$ the corresponding isometric involution. Since not all conical angles of $S$ are integer, there exists a unique anti-conformal isometry $\tau$ of $S$ fixing its conical points. Clearly, $\tau$ commutes with $\tau_{\gamma}$ and so $\tau$ lives $\gamma\subset S$ invariant. Hence $\tau$ induces the desired involution on $D\subset S$ .
\end{proof}

\begin{proof}[Proof of Proposition \ref{injcontrol}] Since clearly $\inj(y)\leq \Vor_S(y)$, Inequality (\ref{injinequality}) immediately follows from Claim (iii). So, we need to prove Claims (i-iii).

(i) Since $\inj(y)<\Vor_S(y)$, the existence of a geodesic loop of length $2\inj(y)$, based at $y$ is straightforward. Indeed, the midpoint of such a loop is a point at distance $\inj(y)$ from $y$, where disk centred at $y$ of radius $\inj(y)$ touches itself. One can check that $l(\gamma)\le \pi$, since otherwise there will be points close to the midpoint of $\gamma$ that can be joined with $y$ by two distinct geodesic segments of length less than $\inj(y)$. To see that $\inj(y)<\frac{\pi}{2}$ we note that in case $\inj(y)=\frac{\pi}{2}$ the boundary of the open disk centred at $y$ of radius $\pi/2$ is a closed geodesic to which the disk is adjacent twice. This means that $S$ is a standard $\mathbb RP^2$, which is impossible since $S$ is orientable.

(ii)  Assume $\Vor_S(y)>\frac{\pi}{2}$, and suppose by contradiction that $\inj(y)<\Vor_S(y)$. Let $\gamma$ be a geodesic constructed in (i). Let $2\pi\theta$ and $2\pi(1-\theta)$  be the angles in which $\gamma$ cuts the neighbourhood of $y$, and assume without loss of generality that $\theta\le \frac{1}{2}$. 

Take now a point $O\in \mathbb S^2$ and consider a spherical kite $OP_1QP_2$ in $\mathbb S^2$ with $\angle O=2\pi\theta$, $\angle P_1=\angle P_2=\pi/2$ and $l([OP_1])=l([OP_2])=l(\gamma)/2$. Since $\theta\le \frac{1}{2}$ and $l([OP_1])\le\frac{\pi}{2}$, one can check that $l([OQ])\le \frac{\pi}{2}$. In particular the kite lies in the interior of a disk $\mathbb D_r$ centred at $O$ for any  $r\in (\pi/2, \Vor_S(y))$. Since $\Vor_S(y)>r$, there exists a locally isometric immersion $\iota: \mathbb D_{r}\to \dot S$ such that $\iota(O)=y$. By pre-composing $\iota$ with rotation, we can arrange so that $\iota$ sends the sides $OP_1$ and $OP_2$ to $\gamma$, and $\iota(P_1)=\iota(P_2)$ is the mid-point of $\gamma$. It is clear then that the segments $P_1Q$ and $P_2Q$ are sent by $\iota$ to the same geodesic segment in $\dot S$. It follows that $\iota$ is not a locally isometric immersion in any neighbourhood of $Q$. This is a contraction.
%, since $\iota:\mathbb D_{r}\to \dot S$ is an immersion.

(iii) Since $\inj(y)<\Vor_S(y)$, by (i) there is a simple geodesic  loop $\gamma$ on $\dot S$ based at $y$ of length $2\inj(y)<\pi$.
%Since $\inj(y)<\Vor_S(y)$, by (ii) we know that $\Vor_S(y)\le \frac{\pi}{2}$.  Hence, by (i), there is a simple geodesic  loop $\gamma$ on $\dot S$ based at $y$ of length $2\inj(y)$. 
We will consider  separately two possibilities, depending on whether $\gamma$ is {\it essential}\footnote{I.e. it doesn't bound on $\dot S$ a disk with at most one puncture.} on $\dot S$, or non-essential.

If $\gamma$ is essential on $\dot S$, it follows  from \cite{MP:systole} that $\inj(y)=l(\gamma)/2>\sys(S)$, and so we are in case (a). 

Let's assume now that $\gamma$ is non-essential on $\dot S$. Then $\gamma$ encircles on $S$ a disk $D$  with at most one conical point in its interior. Since  $l(\gamma)<\pi$ by (i), the disk $D$ should contain exactly one conical point, which we denote $x_i$. Denote by $2\pi\theta$ the angle that $\gamma$ forms at $y$ in $D$. 

Suppose first that $\theta\ge \frac{1}{2}$. In this case $\gamma$ forms a convex boundary of the surface $S\setminus D$. Thanks to this, using exactly the same method as in \cite[Corollary 3.11]{MP:systole} one proves that $l(\gamma)>2\sys(S)$, and we are in case (a).

%Suppose now $\theta<\frac{1}{2}$. Note first, that $\th_i$ is non-integer, otherwise the univalent developing map from $D$ to $\mathbb S^2$ would send $\gamma$ to a smooth geodesic loop. Denote by $p$ the mid-point of $\gamma$. We will prove now that there exist two geodesic segments $x_ip$ and $yx_i$ that cut $D$ into two isometric right-angled triangles. To do this, consider the sphere $S(y,x_i,p)$ obtained by gluing segments into which $p$ divides $\gamma$, and denote by $yp$ the image of $\gamma$ in $S(y,x_i,p)$. The sphere has conical angles $\pi(\theta, 2\th_i, \frac{1}{2})$ so it has a unique orientation reversing involution $\tau$. We claim that $\tau$ fixes $yp$, indeed

Suppose now $\theta<\frac{1}{2}$. Since $\ell(\gamma)<\pi$ we can apply Lemma \ref{involDisk} to $D$ to get its isometric involution $\tau$. This involution fixed the midpoint $p$ of $\gamma$, and fixes two geodesic segments $yx_i$ and $px_i$ that cut $D$ into two isometric right-angled spherical triangles.
%Let $p$ be the midpoint of $\gamma$. Let $yx_i$ be the shortest geodesic segment in $D$ that joins $y$ with $x_i$, let $px_i$ be the shortest geodesic segment that joins $p$ with $x_i$, and
Let  $yp$ be one of two halves of $\gamma$. The segments $yx_i$,  $px_i$ and $yp$ border a triangle $x_iyp$ in $D$ with $\angle x_i=\pi\th_i$, $\angle y=\frac{\pi\theta}{2}$, $\angle p=\frac{\pi}{2}$. Since the side $yp$ of the triangle is shorter than $\pi$ and two adjacent angles are less than $\pi$, the triangle is convex. Since $|yx_i|>|yp|$, we have $\theta_i<\frac{1}{2}$.
Applying the sine rule to the triangle $x_iyp$ we get $\sin(|yp|)=\sin(\pi\th_i)\sin(|x_iy|)$. Hence 
$$\inj(y)=|yp|>\sin(\pi\th_i)\sin(|x_iy|)>2\th_i\sin(\Vor_S(y))>\frac{4}{\pi}\th_i\cdot \Vor_S(y),$$
which proves that we are in case (b).
%; we have $\th_i<\frac{1}{2}$ since $yp$ is the longest side in the triangle $x_iyp$.
\end{proof}

\subsection{Equivalence of Lipschitz and analytic topologies on $\MT$}

In this section we prove that Lipschitz distance between triangles induces the same topology on $\MT$ as the topology induced by the embedding in $\mathbb R^6$, described in Theorem \ref{smoothKlein}.

\begin{definition} The {\it relative Lipschitz distance} $d_{\overline{\mathcal{L}}}$ (or $\overline{\mathcal{L}}$-distance)  between two spherical triangles
%Define {\it relative Lipschitz distance} or $\overline{\cal L}$-{\it distance}  $d_{\overline{\cal L}}$   between spherical triangles $\Delta_1$ and $\Delta_2$ as 
is the infimum of $\log(\max(\mathrm{dil}(f), \mathrm{dil}(f^{-1})))$ over all the marked bi-Lipschitz homeomorphisms $f:\Delta_1\to \Delta_2$ that restrict to a homothety on each edge of $\Delta_1$.
\end{definition}

The $\overline{\cal L}$-distance defines a metric on the space $\MT$ of spherical triangles, which we call the  $\overline{\cal L}$-metric. We have the following natural  statement.

\begin{proposition}\label{threetop} The topologies defined on $\MT$ by the $\cal L$-metric and the $\overline{\cal L}$-metric coincide with the analytic topology given by the angle-side-length embedding $\Psi:\MT\to \mathbb R^6$.  
\end{proposition}
\begin{proof} Note that the side-lengths of $\Delta$ are clearly continuous functions in both $\overline{\cal L}$ and ${\cal L}$ topologies. The angles of $\Delta$ are continuous in these topologies thanks to Lemma \ref{contangles}, applied to the double of $\Delta$. Furthermore the $\overline {\cal L}$-distance is greater of equal to the $\cal L$-distance.  Hence, the $\overline{\mathcal{L}}$-topology is finer than the $\mathcal{L}$-topology, which is finer than the analytic topology.
For this reason, we only need to show that for any spherical triangle $\Delta$ and a sequence of triangles $\Delta_i$ converging  to $\Delta$ in $\mathbb R^6$ (i.e., in the analytic topology), we have $\lim d_{\overline{\cal L}}(\Delta_i,\Delta)=0$.
This claim can be proven by exhibiting explicit bi-Lipschitz maps between spherical triangles. We will only treat the case when $\Delta$ is short-sided, since this is the only case needed for the purposes of the paper.

Following \cite[Lemma 4.1]{EGnew} denote by $U$ the open subset of $\MT$ consisting of triangles with angles $\pi\th_i$, where $\th_i<2$. This subset consists of spherical triangles that admit an isometric embedding into $\mathbb S^2$. In particular $U$ lies in ${\MT}_{sh}$, the space of all short-sided triangles. Let's first prove that the $\overline{\cal L}$-topology coincides with the analytic topology on $U$.

For two spherical triangles $\Delta=x_1x_2x_3$ and $\Delta'=x_1'x_2'x_3'$ embedded into $\mathbb S^2$, with incentres $I_\Delta$ and $I_{\Delta'}$ respectively, define the {\it incentric} map $\Phi:\Delta\to \Delta'$ as the unique map satisfying the following properties. 
\begin{itemize}
\item $\Phi(x_i)=x_i'$, $\Phi(I_{\Delta})=\Phi(I_{\Delta'})$.
\item $\Phi$ is a homothety on each edge $x_ix_j$. 
\item For any point $p\in \partial \Delta$,  $\Phi$ sends the geodesic segment $pI_{\Delta'}$ to a geodesic segment and restricts to a homothety on it. 
\end{itemize}

Suppose now we have a sequence of embedded triangles $\Delta_i\in U$ whose angles and side-lengths converge to that of $\Delta\in U$. Then it not hard to see that the bi-Lipschitz constant of the incentric maps $\Phi: \Delta_i\to \Delta$ tends to $1$. Hence $\Delta_i$ converges to $\Delta$ in $\overline{\cal L}$-topology as well. This proves the statement for $U$.

Let us denote by $U_{klm}\subset \MT_{sh}$ the subspace of triangles which can be obtained from an embedded triangle $\Delta$ by repeated gluing of correspondingly $(k-1)$, $(l-1)$ and $(m-1)$ hemispheres to the sides $x_1x_2$, $x_2x_3$ and $x_3x_1$ of $\Delta$.  From Theorem 4.7 and Lemma 5.2 from \cite{EGnew} it follows that the sets $U_{klm}$ give an open cover of $\MT_{sh}$. At the same time, the incentric map  $\Phi$ between any two  triangles $\Delta$ and $\Delta'$ from $U$ can be  naturally extended  to a map $\tilde\Phi: \tilde \Delta\to \tilde \Delta'$ between triangles with attached hemispheres. Namely a radius of each hemisphere is sent isometrically to a radius and the restriction of $\tilde \Phi$ to both sides of each hemisphere are homotheties. Since the Lipschitz constants of $\Phi$ and $\tilde\Phi$ clearly coincide, the statement about the topologies is proven for each $U_{klm}$ and so for the whole space $\MT_{sh}$.
\end{proof}

\subsection{Delaunay triangulations}

We now turn to triangulations of spherical surfaces into convex spherical triangles. 
We will not require the triangulation to induce on the surface the structure of a simplicial complex. In particular, a triangle can be adjacent to a vertex up to $3$ times, and to an edge up to $2$ times.

The first result is a variation of the famous Delaunay triangulations of the plane \cite{Del} (see also \cite[Section 14]{Pak} for a modern exposition and references within).   

\begin{proposition}[Delaunay triangulations]\label{Delaunay} Let $S$ be a spherical surface with conical points $x_1,\ldots, x_n$, some of which might have angle $2\pi$. Suppose that the Voronoi function $\Vor_S$ is bounded by $\frac{\pi}{2}$. Then there exists a triangulation of $S$ into convex spherical triangles  with the following \emph{``empty circle"} property: for each triangle $x_ix_jx_k$ of the triangulation there exists a vertex $v\in \Gamma(S)$ at equal distance $r$ from $x_i, x_j, x_k$, such that $d(x_l,v)\ge r$ for all $l\in\{1,\ldots, n\}$.
\end{proposition}
The proof will follow the proof by Thurston of a similar result \cite[Proposition 3.1]{Th} concerning triangulations of surfaces with flat metric and conical singularities. We will need the following elementary lemma.

\begin{lemma}\label{segmentin} Let $D, D'\subset \mathbb S^2$ be two disks of radius less than $\frac{\pi}{2}$. Let $x_1,x_2\in \partial D$ and $x_1',x_2'\in \partial D'$ be four distinct points. Suppose $x_1,x_2$ don't lie in the interior of $D'$ and $x_1',x_2'$ don't lie in the interior of $D$. Then
the geodesic segments $x_1x_2\subset D$ and $x_1'x_2'\subset D'$ are disjoint in $\mathbb S^2$.
\end{lemma}
\begin{proof} If  $D$ and $D'$ are disjoint, there is nothing to prove. Suppose $D$ and $D'$ intersect, and let $y_1,y_2$ be the two points of intersection of the boundary circles $\partial D, \partial D'$. Let $\gamma$ be the unique great circle on $\mathbb S^2$ passing through $y_1$ and $y_2$.
It is now easy to see that the complements $D\setminus D'$ and $D'\setminus D$ lie in different hemispheres of $\mathbb S^2$ with respect to $\gamma$. It follows that the segments  $x_1x_2$ and $x_1'x_2'$ also lie in different hemispheres, and so they can intersect only in their endpoints. However, points $x_i$ and $x_i'$ are distinct, so $x_1x_2$ and $x_1'x_2'$ are disjoint.
\end{proof}

\begin{proof}[Proof of Proposition \ref{Delaunay}] 
The proof follows very closely the proof of \cite[Proposition 3.1]{Th}. 
Let $\Gamma(S)$ be the Voronoi graph of $S$. Let us first explain how to associate to each edge $e\subset \Gamma(S)$ a {\it dual} geodesic segment $\check{e}$ with conical endpoints. 
 
Let $p\in \Gamma(S)$ be a point in the interior of an edge $e\subset \Gamma(S)$, and set $r=\Vor_S(p)$. Then there exists a locally isometric immersion $\iota_p:\mathbb D_r\to S$ from a radius $r<\frac{\pi}{2}$ spherical disk, that sends the centre of $\mathbb D_r$ to $p$. Exactly two of the boundary points of $\mathbb D_r$, say $y, z$, are sent to two conical points $x_i, x_j$ of $S$. Denote by $\check{e}$ the image $\iota_p(yz)$. It is easy to see that the segment $\check{e}$ is independent of the choice of $p\in e$.   

Let us now deduce from Lemma \ref{segmentin} that for any two edges $e,e'\subset \Gamma_S$ their dual edges $\check{e}, \check{e}'$ do not intersect in their interior points. This is  similar the proof of \cite[Proposition 3.1]{Th}. 
Let $D,\, D'$ be the disks immersed in $S$, that correspond to $e$ and $e'$. 
Assume by contradiction that $\check{e}, \check{e}'$ intersect in their interior point $p$. Consider the (multivalued) developing map $\iota: S\to \mathbb S^2$. The images of $D$ and $D'$ under this map are embedded disks, and the images of  $\check{e}, \check{e}'$ are cords of these disks, intersecting in $\iota(p)$. This contradicts Lemma \ref{segmentin}. Indeed, the endpoints of  $\check{e}$ are conical points that belong to $\partial D\setminus D'$, and the endpoints of  $\check{e}'$ are conical points that belong to  $\partial D'\setminus D$. Hence, Lemma  \ref{segmentin} is applicable to the $4$-tuple $\iota(D,\check{e}, D',  \check{e}')$. 

Next, we associate to each vertex $v$ of $\Gamma(S)$ a convex polygon embedded in $S$, whose edges $\check{e}_1,\ldots, \check{e}_k$ are dual to the half-edges of $\Gamma(S)$ adjacent to $v$. To do so, consider the immersion $\iota_v:\mathbb D_r\to S$ of a disk of radius $r=\Vor_S(v)$, that sends the centre of $\mathbb D_r$ to $v$. There will be exactly $k$ points, say $y_1,\ldots, y_k$, on $\partial \mathbb D_r$ whose images in $S$ are conical points. Let $P_v$ we the convex hull of the points $y_i$ in $\mathbb D_r$. Then the map $\iota_v$ is an embedding on the interior $\mathring{P}_v$ of the polygon $P_v$, it may identify some vertices and it may identify an edge to at most one other edge of $P_v$.

Our last observation is that the union of the $\iota_v(\mathring{P}_v)$
over all vertices $v$ of $\Gamma(S)$ coincides with the complement in $S$ to the union of edges $\check{e}$. Indeed, since the edges $\check{e}$ can only intersect at endpoints, each $\iota_v(\mathring{P}_v)$ is a connected component of the complement to edges $\check{e}$. At the same time, each edge $\check{e}$ is adjacent to one or two open polygons $\iota_v(\mathring{P}_v)$ corresponding to the vertices of the edge $e$ dual to $\check{e}$. It follows that polygons $\iota_v(P_v)$ cover the whole $S$.

Finally, in case some of convex polygons $\iota_v(P_v)$ are not triangles, we subdivide them by diagonals into a collection of triangles. This gives us the desired triangulation of $S$, where  for each triangle $x_ix_jx_k$ the point $v$ is the corresponding vertex of the Voronoi graph.  
\end{proof}

\begin{remark}\label{ciremark} Let $\Delta$ be a triangle from Delaunay triangulation with vertices $x_i,\,x_j,\,x_k$ and let $v$ be the corresponding vertex of $\Gamma(S)$. Then the circumscribed radius of $\Delta$ is equal to $\Vor_S(v)=d(v,x_i)$. 
\end{remark}

\subsubsection{Compact subsets of $\MSPH_{g,n}\lA$}

In this subsection we prove Proposition \ref{trianglecompact} that singles out a class of compact subsets of $\MSPH_{g,n}\lA$, consisting of surfaces that admit triangulations into triangles of {\it bounded shapes}.

\begin{definition}[Convex $(l,r)$-bounded triangles]  Fix constants $l\in (0,\pi)$, $r\in (0,\frac{\pi}{2})$. We say that  a convex spherical triangle is  $(l,r)$-\emph{bounded} if all its sides have length at least $l$ and its circumscribed circle has radius at most $r$. We will denote  by $\MT_{l,r}$ the subset of $\MT$ consisting of $(l,r)$-bounded triangles.

A spherical surface is $(l,r)$-{\it bounded} if it admits a triangulation into $(l,r)$-bounded spherical triangles.
\end{definition} 

\begin{remark}[Compactness of $\MT_{l,r}$]\label{compactMT} The set $\MT_{l,r}$ is compact in analytic topology of $\MT$. Indeed, let $\Delta_i\subset \mathbb S^2$ be a sequence of convex triangles from $\MT_{l,r}$ with vertices $(x_1^i,x_2^i,x_3^i)$. Passing to a subsequence we can assume that the sequences of vertices converge to $x_1,x_2,x_3$. We have $|x_ix_j|\ge l$ and the circle on $\mathbb S^2$ containing $x_1,x_2,x_3$ has radius at most $r$. Hence $x_1x_2x_3$ is a triangle from $\MT_{l,r}$.
\end{remark}

\begin{definition}[Space of $(l,r)$-triangulated surfaces] Let $\mu$ be a combinatorial type of triangulations of a genus $g$ surface with $n$ marked points, such that the marked points are vertices of the triangulation.  Denote by $Y_{l,r}^{\mu}\lA$ the set of all spherical surfaces of area at most $A$ with a chosen triangulation of  type $\mu$, consisting of $(l,r)$-bounded triangles. The $\overline{\cal L}$-distance between two triangulated surfaces from $Y_{l,r}^{\mu}\lA$ is the Lipschitz distance with respect to all the maps that send the triangulation to the triangulation and restrict to homotheties on the edges.
\end{definition}

\begin{lemma} The set $Y_{l,r}^{\mu}\lA$ is compact in the $\overline{\cal L}$-metric.
\end{lemma}
\begin{proof} From Remark \ref{compactMT} and Proposition \ref{threetop} it follows that the subset $\MT_{l,r}\subset \MT$ of $(l,r)$-bounded triangles is compact in the $\overline {\cal L}$ metric. At the same time $Y_{l,r}^{\mu}\lA$ can be identified with a closed subset of the set of  
$(\MT_{l,r})^{|\mu|}$, where $|\mu|$ is the number of triangles in $\mu$.
\end{proof}

\begin{proposition}[$\mathcal{L}$-compactness of $(l,r)$-bounded surfaces]\label{trianglecompact} Fix $A>0$, $l\in (0,\pi)$ and $r\in (0,\frac{\pi}{2})$. Then the subset $X_{l,r}\lA$ of $\MSPH_{g,n}\lA$ consisting of $(l,r)$-bounded surfaces is compact in Lipschitz topology. The analogous statement holds for $\MP_n\lA$.
\end{proposition}

\begin{proof} Since  the area of an $(l,r)$-bounded triangle is bounded from below, there exists only finite number of combinatorial triangulations $\mu$ of surfaces from $\MSPH_{g,n}\lA$. Note that for each $\mu$  the  natural map $Y_{l,r}^{\mu}\lA\to \MSPH_{g,n}\lA$, that forgets the triangulation, is  continuous, since it contracts the metric. Hence $X_{l,r}\lA$ is a finite union of images of compact sets under continuous maps.
\end{proof}

\subsection{Properness of the function $\sys(S)^{-1}$ on $\MSPH_{g,n}\lA$}\label{seclocalprop}

In this section we deduce Theorem \ref{propersys} and Corollary \ref{Lippolygon} from the following result.

\begin{theorem}[Bounded Delaunay triangulations]\label{boundedtriangulation} For any $s>0$ we have
\begin{itemize}
\item[(i)] Any spherical surface from $\MSc$ can be triangulated into  $(\frac{s}{2},\frac{\pi}{4})$-\emph{bounded} spherical triangles.
\item[(ii)]  Any spherical polygon $P$ with $\sys(P)\ge s$ can be triangulated into  $(f(s),\frac{\pi}{4})$-\emph{bounded} spherical triangles, where $f$ is a positive and continuous function.
\end{itemize}
\end{theorem}

\begin{proof}[Proof of Theorem \ref{propersys}] By Proposition \ref{trianglecompact} the subset  $X_{\frac{s}{2},\frac{\pi}{4}}\lA$ of $\MSPH_{g,n}\lA$ consisting of $(\frac{s}{2},\frac{\pi}{4})$-bounded surfaces is compact. By Theorem \ref{boundedtriangulation} (i) any surface from $\MSc$ belongs to this subset. Since function $\sys$ is continuous on $\MSPH_{g,n}$ by Lemma \ref{systoleratio}, $\MSc$ is closed and so compact.
\end{proof}

\begin{proof}[Proof of Corollary \ref{Lippolygon}]
The proof is identical to the proof of Theorem \ref{propersys}, where instead of using Theorem \ref{boundedtriangulation} (i) one applies  Theorem \ref{boundedtriangulation} (ii).
\end{proof}

\begin{proof}[Proof of Theorem \ref{boundedtriangulation}]
(i) We will prove that for any $S\in \MSc$, there exists a collection of regular points $x_{n+1},\ldots, x_{n+m}\in S$, such that the surface $(S,x_1,\ldots, x_{n+m})$ has the following three properties. 

1) For any $i\ne j$, $d(x_i,x_j)\ge \frac{s}{2}$ for all $i\ne j\in \{1,\ldots, n+m\}$.

2) For each $i$ the injectivity radius of $x_{i}$ on $S$ is  at least $\frac{s}{4}$.

3) For any $x\in S$ there is a point $x_i$, such that $d(x,x_i)\le \frac{\pi}{4}$.

Before proving this claim, let us explain why the statement of the theorem follows from it. Indeed, suppose that we have such a collection of points. Then let us consider  the Delaunay triangulation of $S$ with respect to points $x_1,\ldots, x_{n+m}$ that exists thanks to Proposition \ref{Delaunay}. We claim that all the triangles of the triangulation are $(\frac{s}{2},\frac{\pi}{4})$-bounded. 
Indeed, by condition 3) and Remark \ref{ciremark} each such triangle is isometric to a triangle that can be inscribed in a circle of radius at most $\frac{\pi}{4}$. At the same time by conditions 1), 2) all sides of the triangle have length at least $\frac{s}{2}$.

Let us now show how to find such a collection of points $x_{n+1},\ldots,x_{n+m}\in S$. We will add points $x_{n+1},\ldots, x_{n+m}$ by induction. Note first that points $x_1,\ldots, x_n$ satisfy conditions 1) and 2). Suppose that there is a point $x\in S$ at distance more than $\frac{\pi}{4}$ from points $x_1,\ldots, x_n$. Let us denote such $x$ by $x_{n+1}$ and let us show that $(S,x_1,\ldots, x_{n+1})$ satisfies conditions 1) and 2) for $m=1$. Note that by  \cite[Lemma 3.10]{MP:systole} we have $\sys(S)\le \frac{\pi}{2}$, which means $\frac{\pi}{4}\ge \frac{s}{2}$, and so when we add $x_{n+1}$ we don't violate 1). It remains to show that the injectivity radius of $x_{n+1}$ of $S$ is at least $\frac{s}{4}$.
Let's apply Inequality (\ref{injinequality}) of Proposition \ref{injcontrol}. We get
$${\rm inj}(x_{n+1})\ge \min\left(s, \frac{\pi}{4},\min_i \th_i\cdot \frac{\pi}{4}\right).$$
However, by Lemma \cite[Lemma 3.13]{MP:systole} we know that $\sys(S)\le\min_i \th_i\cdot \pi$. So we get $\inj(x_{n+1})\ge \frac{s}{4}$. Hence, condition 2) is  satisfied for $x_1,\ldots, x_{n+1}$. In this way we can go on adding points $x_{n+i}$ until condition 3) is satisfied. Indeed, the process must terminate since the $\frac{s}{8}$-neighbourhoods of points $x_{n+i}$ are disjoint disks on $S$ and the area of $S$ is finite.

(ii) To prove the second part of the theorem, we will work with the double $S(P)$ of $P$. We will construct a collection of regular points $x_{n+1},\ldots, x_{n+m}\in S(P)$, such that the surface $(S(P),x_1,\ldots, x_{n+m})$ has the following three properties. 

0) The set of points $x_i$ is invariant under the isometric involution $\tau$ of $S(P)$ 

1) For any $i\ne j$, $d(x_i,x_j)\ge \frac{s}{4}$ for all $i\ne j\in \{1,\ldots, n+m\}$.

2) For each $i$ the injectivity radius of $x_{i}$ on $S$ is  at least $\frac{s}{8}$.

3) For any $x\in S$ there is a point $x_i$, such that $d(x,x_i)\le \frac{\pi}{4}$.

Let us explain how to make the first step. Consider $P$ and $\partial P$ as subsets of $S(P)$. Suppose there is a point $y\in S(P)$ at distance greater than $\frac{\pi}{4}$ from $x_1,\ldots, x_n$. In case its distance from $\partial P$ is more than $\frac{\pi}{8}$, we set $x_{n+1}=y$, $x_{n+2}=\tau(y)$. In such case conditions 0)-2) are still satisfied for points $x_1,\ldots, x_{n+2}$, since $d(x_{n+1},x_{n+2})\ge \frac{\pi}{4}$. Suppose now that $d(y,\partial P)<\frac{\pi}{8}$. Let $y'$ be a point on $\partial P$ closest to $y$ and set $x_{n+1}=y'$. Clearly, the distance from $x_{n+1}$ to $x_1,\ldots, x_n$ is at least $\frac{\pi}{8}$. For this reason, as in (i) conditions 2) and 3) are still satisfied. This finishes the first step.

Now, we repeat the above step until we get a collection of points $x_1,\ldots, x_{n+m}$ in $S(P)$, that satisfy all four conditions 0)-3). As in the proof of Proposition \ref{Delaunay}, we get a canonical decomposition of $S(P)$ into convex spherical polygons, invariant under the action of $\tau$, and such that each polygon has side lengths at least $\frac{s}{4}$ and can be inscribed in a circle of radius at most $\frac{\pi}{4}$. Those polygons whose interior doesn't intersect $\partial P$ should be further cut into triangles by diagonals. Suppose that the interior of a polygon $Q$ intersects $\partial P$. Then $\tau(Q)=Q$ and using a $\tau$-invariant subset of  diagonals of $Q$, one can cut it into a union of triangles exchanged by $\tau$ and either a triangle or a trapezoid $Q'$, satisfying $\tau(Q')=Q'$. If $Q'$ is a triangle we take $Q'\cap P$ as one of the triangles of the triangulation of $P$. If $Q'$ is a trapezoid, we subdivide further $Q'\cap P$ into two triangles along a diagonal.
It is not hard to see that the resulting triangles are $(f(s),\frac{\pi}{4})$-bounded for certain positive function $f(s)$. That concludes the decomposition of $P$ into triangles.
\end{proof}

\subsection{Systole of balanced triangles}

In this section we calculate the systole of a balanced triangle, show that for a balanced triangle $\Delta$ we have $\sys(\Delta)=\sys(T(\Delta))$. 

\begin{lemma}\label{trianglesystole}Let $\Delta$ be a balanced spherical triangle with vertices $x_1,x_2,x_3$, then
\begin{equation}\label{systrlength}2\sys(\Delta)= \min_{i,j}(\min(|x_ix_j|, 2\pi-|x_ix_j|)).
\end{equation}
Moreover the following two statements hold.
\begin{itemize}

\item[(i)] For any vertex $x_i$ of $\Delta$, the distance to the opposite side $x_ix_j$ is larger than $\sys(\Delta)$. 
\item[(ii)] Let $p\in \partial \Delta$ be a point that is not a vertex of $\Delta$. Suppose that $\eta$ is a geodesic segment in $\Delta$ that joins $p$ with $x_i$ and doesn't belong to $\partial \Delta$. Then $l(\eta)>\sys(\Delta)$.
\item[(iii)] There exists a geodesic segment $\gamma_{\Delta}\subset \Delta$ of length $2\sys(\Delta)$, that joins two vertices of $\Delta$.
\end{itemize}

\end{lemma}

\begin{proof} We will first prove statements (i), (ii) and then will deduce 
Equality (\ref{systrlength}).

(i) Let us show that for any $p$ in $x_2x_3$ we have $d(p,x_1)> \sys(\Delta)$. 
From  Remark \ref{vorbalanced} it follows that $p$ lies either in the Voronoi domain of $x_2$ or of $x_3$. Assume the former, then by definition of Voronoi domains, $d(p,x_1)\ge d(p,x_2)$. 

Suppose first that the strict inequality $d(p,x_1)> d(p,x_2)$ holds. Applying the triangle inequality to the points $x_1,x_2,p$, and using  $d(x_1,x_2)\ge 2\sys(\Delta)$ , we get
$$d(p,x_1)\ge d(x_1,x_2)-d(p,x_2)> d(x_1,x_2)-d(p,x_1)\ge 2\sys(\Delta)-d(p,x_1).$$
It follows, that $d(p,x_1)> \sys(\Delta)$. 

Suppose now that $d(p,x_1)= d(p,x_2)$, then by Remark  \ref{vorbalanced} the triangle $\Delta$ is semi-balanced, $p$ is the mid-point the segment $x_1x_2$, and there is a geodesic segment $x_1p$ that joins $x_1$ with $p$. It is clear then, that $2|x_1p|=|x_1p|+|x_2p|>2d(x_1,x_2)\ge 2\sys(\Delta)$. 

(ii) Consider two cases. If $p$ lies on the side of $\Delta$ opposite to $x_i$ then by (i) we have $\ell(\eta)\ge d(x_i,p)>\sys(\Delta)$. Suppose now $p$ lies on a side adjacent to $x_j$. In this case $\eta$ cut's out of $\Delta$ a digon with angles less than $\pi$ (since $p$ is an interior point of an edge). So $e(\eta)=\pi$ and the statement follows from Corollary \ref{balancedsystole}.

(ii) Using (i) and Definition \ref{sysdefinition}, we see that $2\sys(S)=\min_{i,j}d(x_i,x_j)$. Hence there is a geodesic segment $\gamma_{\Delta}$ of length $2\sys(S)$ that joins two vertices of $\Delta$.

Let us now prove Equality (\ref{systrlength}). Let us take the geodesic $\gamma_{\Delta}$ given by {\it 3}. It cuts out of $\Delta$ a digone, one of whose sides is a side $x_ix_j$ of the triangle $\Delta$. If follows that either $2\sys(\Delta)=|x_ix_j|$ of $2\pi-|x_ix_j|$. This shows that $2\sys(\Delta)$ no smaller than the right hand expression in (\ref{systrlength}). On the other hand, the opposite inequality follows immediately from Corollary \ref{balancedsystole}.
\end{proof}

\begin{lemma}\label{sysequality} For any balanced triangle $\Delta$ and the corresponding spherical torus $(T(\Delta),x)$ we have $\sys(\Delta)=\sys(T(\Delta))$. Conversely, for any spherical torus $T$ and the corresponding balanced spherical triangle $\Delta(T)$ we have $\sys(T)=\sys(\Delta(T))$.
\end{lemma}

\begin{proof} The first and the second statements are equivalent, so will prove just the first statement. By Lemma \ref{trianglesystole} (iii), there is a geodesic segment $\gamma_{\Delta}$ in $\Delta$ of length $2\sys(\Delta)$ that joins two vertices of $\Delta$. Such $\gamma_{\Delta}$ is embedded as a geodesic loop in $T(\Delta)$, which clearly implies $\sys(\Delta)\ge\sys(T(\Delta))$. Let us prove  that $\sys(\Delta)\le\sys(T(\Delta))$. Indeed, let $\gamma_{T(\Delta)}$ be the systole geodesic loop in $T(\Delta)$. Let $\Delta_1, \Delta_2$ be two balanced triangles isometric to $\Delta$ from which $T(\Delta)$ is glued. It will be enough to prove that $\gamma_{T(\Delta)}$ lies entirely in $\Delta_1$ or $\Delta_2$. Assume by contradiction that this is not so. Then  $\gamma_{T(\Delta)}$  contains  two sub-segments $\eta, \eta'$ whose interior lie in the interior of $\Delta_1$ or $\Delta_2$ and which satisfy the conditions of Lemma \ref{trianglesystole} (ii).  Applying this lemma, we get  $l(\gamma_{T(\Delta)})\ge l(\eta)+l(\eta')>2\sys(\Delta)$, which contradicts the  established inequality $\sys(\Delta)\ge\sys(T(\Delta))$.
\end{proof}

\begin{corollary}\label{sysproper} The function $\sys(\Delta)^{-1}=2(\min_{i,j}(\min(|x_i,x_j|, 2\pi-|x_i,x_j|))^{-1}$ is proper on $\MT_{bal}(\th)$ in the analytic topology.
\end{corollary}
\begin{proof} The function $\sys(\Delta)^{-1}$ is proper on $\MT(\th)$ in $\cal L$-topology by Corollary \ref{Lippolygon}. At the same time, by Proposition \ref{threetop} the $\cal L$-topology and the analytic topology coincide on $\MT_{bal}(\th)$.
\end{proof}

\subsection{Proof of Theorem \ref{torusbalancehomeo}}

Here we finally prove Theorem \ref{torusbalancehomeo} concerning $2$-marked tori. We note first that Theorem \ref{propersys} holds for $2$-marked tori as well, namely the function $\sys^{-1}$ is proper in Lipschitz topology on the space  $\MSPH_{1,1}^{(2)}\lA$ of such tori of area at most $A$.

We will use the following standard lemma, whose proof we omit.

\begin{lemma}\label{hausslemma} Let $X$ and $Y$ be locally compact Hausdorff topological spaces and let $\varphi: X\to Y$ be a continuous bijective  map.
\begin{itemize}
\item[(i)] If $\varphi$ is proper then it is a homeomorphism.
\item[(ii)] Suppose there exist proper functions $s_X:X\rightarrow \mathbb R$ and $s_Y:Y\rightarrow \mathbb R$ such that $s_X=s_Y\circ\varphi$.

$f$ and $g$ on $X$ and $Y$, such that $\varphi^*(g)=f$. Then $\varphi$ is a homeomorphism.
\end{itemize}
\end{lemma}
%\begin{proof}[Proof of Lemma \ref{hausslemma}] (i) It is enough to prove that $\varphi$ is a closed map. Suppose $C\subset X$ is closed. Suppose $y \in \overline{\varphi[C]}$ and let $K$ be a compact neighbourhood of $y$ in $Y$. Then $\varphi^{-1}[K] \cap C$ is compact in $X$ and hence so is $\varphi[\varphi^{-1}[K] \cap C]= K \cap \varphi[C]$. By Hausdorffness of $Y$, $\phi[C] \cap K$ is closed in $K$, which implies $y \in \varphi[C]$ and so $\varphi[C]$ is closed. 
%
%(ii) This follows immediately from (i).
%\end{proof}

\begin{proof}[Proof of Theorem \ref{torusbalancehomeo}] (i)  By Proposition \ref{or-bal-tri} (i), $\MT_{bal}^{\pm}(\th)$ is a surface, so we need to show that $T^{(2)}$ is a homeomophism. To do this let's explain that
we are in the set up of Lemma \ref{hausslemma}. Indeed, we have the following 

1) The map $ T^{(2)}$ is short with respect to $\overline{\cal L}$-metric on $\MT_{bal}^{\pm}(\th)$, since any bi-Lipschitz map $\Delta_1\to \Delta_2$ that restricts to homothety  on edges gives rise to a $\sigma$-equivariant bi-Lipschitz map $T^{(2)}(\Delta_1)\to T^{(2)}(\Delta_2)$ with the same Lipschitz constant. Hence, $T^{(2)}$ is continuous. Moreover, $ T^{(2)}$ is bijective by Lemma \ref{T-bijective}. 

2) Since $\MT_{bal}^{\pm}(\th)$ is a surface, it is locally compact, and the function $\sys^{-1}$ is proper on it by Corollary \ref{sysproper}.

3) The space $(\MSPH_{1,1}^{(2)}(\th),\cal L)$ is locally compact, and the function $\sys^{-1}$ is proper on it by Theorem \ref{propersys}.

4) The map $T^{(2)}$ preserves the function $\sys^{-1}$ by  Lemma \ref{sysequality}.

To sum up, the map $T^{(2)}$ satisfies all the properties of Lemma \ref{hausslemma}, which proves the claim.

(ii) The proof of this claim is the same and so we omit it.
\end{proof} 

\begin{remark}[Orbifold structures on $\MSPH_{1,1}(\th)$ and $\MSPH_{1,1}^{(2)}(\th)$]\label{orbifold-correct}
Let $\MSPH_{1,1}^{(4)}(\th)$ be the set of spherical tori $T$ endowed with a $4$-marking, namely
an isomorphism $H_1(T;\ZZ_4)\cong (\ZZ_4)^2$. We endow $\MSPH_{1,1}^{(4)}(\th)$ with the Lipschitz distance
measured among maps  between tori that respect the $4$-marking. Since $4$-marked tori have no nontrivial
conformal automorphisms, such $\MSPH_{1,1}^{(4)}(\th)$ is a moduli space for such $4$-marked tori.
It is easy to see that the forgetful map $\MSPH_{1,1}^{(4)}(\th)\rightarrow \MSPH_{1,1}^{(2)}(\th)$
is a local isometry and in fact an unramified Galois cover with group $K/\{\pm 1\}$, where
$K=\ker(\mathrm{SL}(2,\ZZ_4)\rar\mathrm{SL}(2,\ZZ_2))$. 

Assume first $\th$ not odd. The space
$\MSPH_{1,1}^{(2)}(\th)$ is an orientable
surfaces of finite type by Theorem \ref{torusbalancehomeo} and Proposition \ref{or-bal-tri} (i),
and so the same holds for $\MSPH_{1,1}^{(4)}(\th)$. 
We endow $\MSPH_{1,1}^{(2)}(\th)$ with the orbifold structure given by $\MSPH_{1,1}^{(4)}(\th)/K$
and $\MSPH_{1,1}(\th)$ with the orbifold structure $\MSPH_{1,1}^{(4)}(\th)/\mathrm{SL}(2,\ZZ_4)$.
As a consequence, $\MSPH_{1,1}^{(2)}(\th)\rar\MSPH_{1,1}(\th)$ is an unramified Galois cover
with group $\mathrm{SL}(2,\ZZ_2)\cong \Sy_3$.

Assume now $\th=2m+1$ odd.
Again, the space  $\MSPH_{1,1}^{(2)}(2m+1)^\sigma$ is a disjoint union of finitely many
two-dimensional disks by Theorem \ref{torusbalancehomeo} and Proposition \ref{int-or-bal}, and so the same holds
for the moduli space $\MSPH_{1,1}^{(4)}(2m+1)^\sigma$.
The same argument as in Construction \ref{fund4} shows that $\MSPH_{1,1}^{(4)}(2m+1)$
fibres over $\MSPH_{1,1}^{(4)}(2m+1)^\sigma$ with fiber $\RR$, and so it is a three-dimensional manifold.
We then put on $\MSPH_{1,1}^{(2)}(2m+1)$ and $\MSPH_{1,1}(2m+1)$ the orbifold structures
induced by $\MSPH_{1,1}^{(2)}(2m+1)=\MSPH_{1,1}^{(4)}(2m+1)/K$ and
$\MSPH_{1,1}(2m+1)=\MSPH_{1,1}^{(4)}(2m+1)/\mathrm{SL}(2,\ZZ_4)$. 
We put a similar structure on the moduli spaces of $\sigma$-invariant metrics.

In all cases, the orbifold order of a point in such moduli spaces is the number of
automorphisms of the corresponding (possibly marked) spherical torus.
\end{remark}

%%%%%%%%%%%%%%%%%%%%%%%%%%%%%%%%

%%%%%%%%%%%%%%%%%%%%%%%%%%%%%%%%

\begin{appendices}

\section{Monodromy and coaxiality}\label{sec:monodromy}

In this section we prove that a spherical torus with one conical point of angle $2\pi\th$ is coaxial if and only if $\th$ is an odd integer. This was already shown in \cite{Lin2}.

In order to prove such result,
we recall that monodromy representation
of spherical surfaces can be lifted to $\SU(2)$
as shown in the following proposition.

\begin{proposition}[Lift of the monodromy to $\SU(2)$] \label{prop:lift}
Let $(S,\bm{x})$ be a spherical surface with conical points of angles $2\pi\cdot\bm{\th}$
and let $p\in\dot{S}$ be a basepoint. Let $(\tilde{\dot{S}},\tilde{p})$ be a universal cover
of $(\dot{S},p)$, endowed with the pull-back spherical metric, and let $\iota:\tilde{\dot{S}}\rar\Sph
\cong\CC\PP^1$ be an
associated developing map
with monodromy representation $\rho:\pi_1(\dot{S},p)\rar\SO_3(\RR)$.
Then there exists a lift $\hat{\rho}:\pi_1(\dot{S},p)\rar\SU(2)$ of $\rho$ such that
\begin{itemize}
\item[(a)]
the developing map $\iota$ extends to the completion $\hat{S}$ of $\tilde{\dot{S}}$
and each point of $\hat{S}\setminus \tilde{\dot{S}}$ corresponds to a loop based at $p$
that simply winds about some $x_j$;
\item[(b)]
if $\gamma_j\in\pi_1(\dot{S},p)$ is a loop that simply winds about $x_j$
corresponding to a point $\hat{x}_j$ in $\hat{S}\setminus \tilde{\dot{S}}$,
then $\hat{\rho}(\gamma_j)\in\SU(2)$
acts on the complex line $\iota(\hat{x}_j)\subset\CC^2$ as the multiplication
by $e^{i\pi(\th_j-1)}$.
\end{itemize}
Moreover two such lifts multiplicatively differ by a homomorphism $\pi_1(S,p)\rar \{\pm I\}$.
\end{proposition}
\begin{proof}
In \cite[Proposition 2.12]{mondello-panov:constraints} the statement was proven for a surface $S$ of genus $0$.
For a surface of arbitrary genus, the proof of existence
for a lift is entirely analogous with minor modifications.
In particular, $D$ will be the complement $S\setminus\{q\}$
of an unmarked point $q$ in $S$, the vector field $V$
is chosen to be nowhere vanishing on $D$ and having
vanishing order $2-2g$ at $q$, so that the unit normalized
vector field $\hat{V}$ on $D$ has even winding number
about $q$.

Finally, two lifts certainly differ by multiplication
by a homomorphism $\pi(\dot{S},p)\rightarrow\{\pm I\}$. Since the eigenvalues
of the monodromy about the punctures are fixed by condition (b),
such homomorphism factors through $\pi(S,p)\rightarrow \{\pm I\}$.
\end{proof}

We use the above $\SU(2)$-lifting property to characterize
1-punctured tori $(S,x)$ with integral angles.
In order to do that, choose standard generators $\{\alpha,\beta,\gamma\}$ of $\pi_1(\dot{S})$ such that $\gamma=[\alpha,\beta]$.
Given a spherical metric on $(S,x)$, its monodromy representation $\rho$
can be lifted to an $\SU_2$-valued representation $\hat{\rho}$ by Proposition \ref{prop:lift}.
Call $A=\hat{\rho}(\alpha)$, $B=\hat{\rho}(\beta)$ and $C=\hat{\rho}(\gamma)$,
and note that $C$ has eigenvalues $e^{\pm i\pi(\th-1)}$.

\begin{corollary}[Monodromy of tori with odd $\th$]\label{coaxtorus}
Let $(S,x)$ be a spherical torus with one conical point of angle
$2\pi\cdot\th$. Then $(S,x)$ has coaxial monodromy if and only if
$\th$ is odd integral.
\end{corollary}
\begin{proof}
The monodromy $\rho$ is coaxial if and only if $\hat{\rho}$ is.
On the other hand, since elements in $\SU(2)$ are diagonalizable,
$\hat{\rho}$ is coaxial if and only if it is Abelian.
Finally, $\hat{\rho}$ is Abelian if and only if $\hat{\rho}(\gamma)=I$,
which implies that $\th$ is odd integral.
\end{proof}

\begin{corollary}[Monodromy of tori with even $\th$]\label{evenklein}
Let $(S,x)$ be a spherical torus with one conical point of angle $2\pi\cdot\th$.
Then the monodromy of $(S,x)$ is isomorphic to the Klein group $K_4\cong\ZZ/2\oplus\ZZ/2$
if and only if $\th$ is even integral. In this case the three nontrivial elements
in the monodromy group are rotations of angle $\pi$ along mutually orthogonal axes of $\Sph$.
\end{corollary}
\begin{proof}
The monodromy is isomorphic to $K_4$ if and only if $\rho(\alpha)^2=\rho(\beta)^2=[\rho(\alpha),\rho(\beta)]=I$.

If $\th$ is even integral, then $C=-I$.
Up to conjugacy, we can assume that $A$ is diagonal.
The relation $AB=-BA$ gives that $A$ has eigenvalues $\pm i$
and $B$ has zero entries on the diagonal. It follows that $A^2=B^2=-I$
and so $\rho(\alpha)^2=\rho(\beta)^2=[\rho(\alpha),\rho(\beta)]=I$.
It can be observed that $A$, $B$ and $AB$ act on $\Sph$ as rotations of angle $\pi$
along mutually orthogonal axes.

Vice versa, suppose that the monodromy if isomorphic to the Klein group.
Then $C=\pm I$ and so $\th$ must be integral, but $\th$ cannot be odd by Corollary \ref{coaxtorus}.
Hence, $\th$ is even.
\end{proof}

\end{appendices}
\bibliographystyle{amsplain}
\bibliography{torus-one-point-biblio}

\providecommand{\bysame}{\leavevmode\hbox to3em{\hrulefill}\thinspace}
\providecommand{\MR}{\relax\ifhmode\unskip\space\fi MR }
\providecommand{\MRhref}[2]{%
  \href{http://www.ams.org/mathscinet-getitem?mr=#1}{#2}
}
\providecommand{\href}[2]{#2}
\begin{thebibliography}{10}

\bibitem{BBI}
Dmitri Burago, Yuri Burago, and Sergei Ivanov, \emph{A course in metric
  geometry}, Graduate Studies in Mathematics, vol.~33, American Mathematical
  Society, Providence, RI, 2001. \MR{1835418}

\bibitem{Lin2}
Ching-Li Chai, Chang-Shou Lin, and Chin-Lung Wang, \emph{Mean field equations,
  hyperelliptic curves and modular forms: {I}}, Camb. J. Math. \textbf{3}
  (2015), no.~1-2, 127--274. \MR{3356357}

\bibitem{Lin3}
Chiun-Chuan Chen and Chang-Shou Lin, \emph{Mean field equation of {L}iouville
  type with singular data: topological degree}, Comm. Pure Appl. Math.
  \textbf{68} (2015), no.~6, 887--947. \MR{3340376}

\bibitem{Lin4}
Zhijie Chen, Ting-Jung Kuo, and Chang-Shou Lin, \emph{Existence and
  non-existence of solutions of the mean field equations on flat tori}, Proc.
  Amer. Math. Soc. \textbf{145} (2017), no.~9, 3989--3996. \MR{3665050}

\bibitem{eisenstein}
Zhijie Chen and Chang-Shou Lin, \emph{Critical points of the classical
  {E}isenstein series of weight two}, J. Differential Geom. \textbf{113}
  (2019), no.~2, 189--226. \MR{4023291}

\bibitem{orbifold}
Daryl Cooper, Craig~D. Hodgson, and Steven~P. Kerckhoff,
  \emph{Three-dimensional orbifolds and cone-manifolds}, MSJ Memoirs, vol.~5,
  Mathematical Society of Japan, Tokyo, 2000, With a postface by Sadayoshi
  Kojima. \MR{1778789}

\bibitem{SG}
Henri~Paul de~Saint-Gervais, \emph{Uniformisation des surfaces de {R}iemann},
  ENS \'{E}ditions, Lyon, 2010. \MR{2768303}

\bibitem{Del}
Boris Delaunay, \emph{Sur la sph{\`e}re vide. {A} la m{\'e}moire de {G}eorges
  {V}orono{\"i}}, Bulletin de l'Acad\'emie des Sciences de l'URSS. Classe des
  sciences mathématiques (1934), no.~6, 793--800.

\bibitem{eremenko:three}
Alexandre Eremenko, \emph{Metrics of positive curvature with conic
  singularities on the sphere}, Proc. Amer. Math. Soc. \textbf{132} (2004),
  no.~11, 3349--3355. \MR{2073312}

\bibitem{E}
\bysame, \emph{Metrics of constant positive curvature with four conic
  singularities on the sphere}, Proc. Amer. Math. Soc. \textbf{148} (2020),
  no.~9, 3957--3965.

\bibitem{EGnew}
Alexandre Eremenko and Andrei Gabrielov, \emph{The space of {S}chwarz-{K}lein
  triangles}, arXiv:2006.16874.

\bibitem{EG}
\bysame, \emph{On metrics of curvature 1 with four conic singularities on tori
  and on the sphere}, Illinois J. Math. \textbf{59} (2015), no.~4, 925--947.
  \MR{3628295}

\bibitem{EGMP}
Alexandre Eremenko, Andrei Gabrielov, Gabriele Mondello, and Dmitri Panov,
  \emph{Moduli spaces for {L}am{\'e} functions and {A}belian integrals of the
  second kind}, arXiv:2006.16837.

\bibitem{ET}
Alexandre Eremenko and Vitaly Tarasov, \emph{Fuchsian equations with three
  non-apparent singularities}, SIGMA Symmetry Integrability Geom. Methods Appl.
  \textbf{14} (2018), Paper No. 058, 12. \MR{3814567}

\bibitem{Gro}
Misha Gromov, \emph{Metric structures for {R}iemannian and non-{R}iemannian
  spaces}, Modern Birkh\"{a}user Classics, Birkh\"{a}user Boston, Inc., Boston,
  MA, 2007, Based on the 1981 French original. \MR{2307192}

\bibitem{hz}
J.~Harer and D.~Zagier, \emph{The {E}uler characteristic of the moduli space of
  curves}, Invent. Math. \textbf{85} (1986), no.~3, 457--485. \MR{848681}

\bibitem{Kl}
Felix Klein, \emph{Vorlesungen \"{u}ber die hypergeometrische {F}unktion},
  Grundlehren der Mathematischen Wissenschaften, vol.~39, Springer-Verlag,
  Berlin-New York, 1981, Reprint of the 1933 original. \MR{668700}

\bibitem{Lin1}
Chang-Shou Lin and Chin-Lung Wang, \emph{Elliptic functions, {G}reen functions
  and the mean field equations on tori}, Ann. of Math. (2) \textbf{172} (2010),
  no.~2, 911--954. \MR{2680484}

\bibitem{Lin5}
\bysame, \emph{Mean field equations, hyperelliptic curves and modular forms:
  {II}}, J. \'{E}c. polytech. Math. \textbf{4} (2017), 557--593. \MR{3665608}

\bibitem{Ma}
Robert~S. Maier, \emph{Lam\'{e} polynomials, hyperelliptic reductions and
  {L}am\'{e} band structure}, Philos. Trans. R. Soc. Lond. Ser. A Math. Phys.
  Eng. Sci. \textbf{366} (2008), no.~1867, 1115--1153. \MR{2377687}

\bibitem{mondello-panov:constraints}
Gabriele Mondello and Dmitri Panov, \emph{Spherical metrics with conical
  singularities on a 2-sphere: angle constraints}, Int. Math. Res. Not. IMRN
  (2016), no.~16, 4937--4995. \MR{3556430}

\bibitem{MP:systole}
\bysame, \emph{Spherical surfaces with conical points: systole inequality and
  moduli spaces with many connected components}, Geom. Funct. Anal. \textbf{29}
  (2019), no.~4, 1110--1193. \MR{3990195}

\bibitem{Pak}
Igor Pak, \emph{Lectures on discrete and polyhedral geometry}, To be published.
  Preliminary version (2010) available at
  {\texttt{https://www.math.ucla.edu/$\sim$pak/book.htm}}.

\bibitem{T}
Gabriella Tarantello, \emph{Analytical, geometrical and topological aspects of
  a class of mean field equations on surfaces}, Discrete Contin. Dyn. Syst.
  \textbf{28} (2010), no.~3, 931--973. \MR{2644774}

\bibitem{Th}
William~P. Thurston, \emph{Shapes of polyhedra and triangulations of the
  sphere}, The {E}pstein birthday schrift, Geom. Topol. Monogr., vol.~1, Geom.
  Topol. Publ., Coventry, 1998, pp.~511--549. \MR{1668340}

\bibitem{WW}
E.~T. Whittaker and G.~N. Watson, \emph{A course of modern analysis}, Cambridge
  Mathematical Library, Cambridge University Press, Cambridge, 1996, Reprint of
  the fourth edition (1927). \MR{1424469}

\end{thebibliography}

\end{document}